\documentclass[preprint]{elsarticle} % 

\usepackage[all]{xy}
\usepackage{amsfonts}
\usepackage{amssymb}
\usepackage{amsmath}
\usepackage{amstext}
\usepackage{amsbsy}
\usepackage{amsopn}
\usepackage{amsthm}
\usepackage{epsfig,rotating}

\usepackage{amscd}
\usepackage{amsxtra}
\usepackage{mathrsfs}
\usepackage{upref}
\usepackage[title]{appendix}
\usepackage[english]{babel}

\usepackage[utf8]{inputenc}
\usepackage[T1]{fontenc}
\usepackage{lmodern}
\usepackage[a4paper]{geometry}
\usepackage{amsmath}
\usepackage{cleveref}

\usepackage{nag}
\usepackage{xcolor}
\usepackage{amsmath}
\usepackage{mathtools}
\usepackage{amsthm}
\usepackage{amssymb}
\usepackage{graphicx}
\usepackage{caption}
\usepackage{enumitem}

                \newcommand{\R}{\mathbb R}        

\newcommand{\f}{\varphi}                
\newcommand{\D}{\partial}               
                 \newcommand{\g}{\gamma} 
            \newcommand{\G}{\Gamma} 
                 \def\l{\lambda}
\newcommand{\sph}{\mathbb S}                        
\newcommand{\comp}{{\small{\circ}}}     \def\a{\alpha}
\newcommand{\C}{\mathcal C}
               \newcommand{\E}{{\mathcal E}}

            \newcommand{\RP}{{\mathbb R\mathrm P}}

\DeclareMathOperator{\Tr}{Tr}           \newcommand{\K}{\mathcal{K}}
          
\def\t{\tau}
\def\hat{\widehat}                      \def\P{\mathbb P}  
\newcommand{\flecha}{\overrightarrow}
                 
\newcommand{\Abf}{\boldsymbol{A}}               \newcommand{\Bbf}{\boldsymbol{B}}
\newcommand{\Cbf}{\boldsymbol{C}}               \newcommand{\Hbf}{\boldsymbol{H}}
\newcommand{\Nbf}{\boldsymbol{N}}                 
\newcommand{\ebf}{\boldsymbol{e}}               \newcommand{\nbf}{\boldsymbol{n}}
\newcommand{\ubf}{\boldsymbol{u}}               \newcommand{\kbf}{\boldsymbol{k}}
\newcommand{\Obf}{\boldsymbol{O}}               \newcommand{\Idbf}{\boldsymbol{Id}}
\newcommand{\Rbf}{\boldsymbol{R}}               
\newcommand{\lbf}{\boldsymbol{\lambda}}         \newcommand{\vbf}{\boldsymbol{v}}

\journal{Journal of Geometry and Physics}

\begin{document}

\captionsetup[figure]{labelfont={bf},labelformat={default},labelsep=period,name={\small Fig.}}

%%%%%%%%%%%%%%%%%%%%%%%%%%%%%%%%%%%%%%%%%%%%%%%%%%%%%%%%%%%%%%%%%%%%%%%%%%%%%%%%%%%%%%%%%%%%%%%%%%%%%%%%
%%%%%%%%%%%%%%%%ù     STYLE  THEOREM    %%%%%%%%%%%%%%%%%%%%%%%%%%%%%%%%%%%%%%%%%%%%%%%%%%%%%%%%%%%%%%%%
%%%%%%%%%%%%%%%%%%%%%%%%%%%%%%%%%%%%%%%%%%%%%%%%%%%%%%%%%%%%%%%%%%%%%%%%%%%%%%%%%%%%%%%%%%%%%%%%%%%%%%%%

\numberwithin{equation}{section}                % numerazione delle equazioni
\theoremstyle{plain}% default
\newtheorem{theorem}{\bf Theorem}
\numberwithin{theorem}{subsection}
\newtheorem*{theorem*}{\bf Theorem}
\newtheorem{teorema}{\bf Theorem}
\def\theteorema{\Roman{teorema}}

\newtheorem{proposition}{\bf Proposition}
\numberwithin{proposition}{subsection}
\newtheorem*{proposition*}{\bf Proposition}
\newtheorem{lemma}[proposition]{\bf Lemma}
\newtheorem*{lemma*}{\bf Lemma}
\newtheorem*{mlemma}{\bf Main Lemma}

\newtheorem{corollary}[proposition]{\bf Corollary}
\newtheorem*{corollary*}{\bf Corollary}
\newtheorem*{conjecture}{\bf Conjecture}%[section]
\newtheorem*{fact*}{\bf Fact}
\newtheorem{fact}[proposition]{\bf Fact}

\newtheorem{property}{\bf Property}

%%%%%%%%%%%%%%%%%%%%%%%%%%%%%%%%%%%%%%%%%%%%%%%%%%%%%%%%%%%%%%%%%%%%%%%%%%%%%%%%%%%%%%%%%%%%%%%%%%%%%%%%
%%%%%%%%%%%%%%%%ù     STYLE DEFINITION    %%%%%%%%%%%%%%%%%%%%%%%%%%%%%%%%%%%%%%%%%%%%%%%%%%%%%%%%%%%%%%
%%%%%%%%%%%%%%%%%%%%%%%%%%%%%%%%%%%%%%%%%%%%%%%%%%%%%%%%%%%%%%%%%%%%%%%%%%%%%%%%%%%%%%%%%%%%%%%%%%%%%%%%

\theoremstyle{definition}
\newtheorem{definition}{\bf Definition}%[section]
\newtheorem*{definition*}{\bf Definition}

\newtheorem{example}{\bf Example}
\newtheorem*{example*}{\bf Example}
\newtheorem{Example}{\bf Example}[section]
\theoremstyle{remark}
\newtheorem*{remark*}{\bf Remark}
\newtheorem{remark}{\bf Remark}
\numberwithin{remark}{section}

\newtheorem*{problem}{\bf Problem}
\newtheorem*{exercise*}{\bf Exercise}
\newtheorem*{note*}{\bf Note}
\newtheorem{note}[proposition]{\bf Note}

%%%%%%%%%%%%%%%%%%%%%%%%%%%%%%%%%%%%%%%%%%%%%%%%%%%%%%%%%%%%%%%%%%%%%%%%%%%%%%%%%%%%%%%%%%%%%%%%%%%%%%%%
%%%%%%%%%%%%%%%%ù     TITLE    %%%%%%%%%%%%%%%%%%%%%%%%%%%%%%%%%%%%%%%%%%%%%%%%%%%%%%%%%%%%%%%%%%%%%%%%%
%%%%%%%%%%%%%%%%%%%%%%%%%%%%%%%%%%%%%%%%%%%%%%%%%%%%%%%%%%%%%%%%%%%%%%%%%%%%%%%%%%%%%%%%%%%%%%%%%%%%%%%%

\begin{frontmatter}

\title{On Surfaces in $\R^n$ via Gauss Map, Caustics, Duality and\\ 
Pseudo Euclidean Geometry of Quadratic Forms}

\author[label1,label2]{Ricardo\,Uribe-Vargas}
\affiliation[label1]{organization={Institut de Math\'ematiques de Bourgogne, UMR 5584, CNRS \& 
		Universit\'e Bourgogne Europe},%Department and Organization
            %addressline={}, 
            city={Dijon},
           % postcode={21000}, 
            %state={},
            country={France}}
\affiliation[label2]{organization={Laboratory Solomon Lefschetz, UMI 2001, CNRS \& 
Universidad Nacional Auton\'oma de M\'exico},%Department and Organization
            %addressline={}, 
            city={Ciudad de M\'exico},
           % postcode={}, 
           % state={},
            country={M\'exico}}
%\footnote{Institut de Math\'ematiques de Bourgogne, UMR 5584, CNRS \& 
%		Universit\'e Bourgogne Europe, Dijon France and 
%  Laboratory Solomon Lefschetz, UMI 2001, CNRS \&  Universidad Nacional Auton\'oma de M\'exico, Cuernavaca, M\'exico. r.uribe-vargas@u-bourgogne.fr}}

\date\empty                     % Data
% \maketitle
% \vspace*{-0.5cm}

%\addtolength{\topmargin}{-67pt}
%\addtolength{\textheight}{100pt}
%\oddsidemargin=0pt
%\textwidth=6in

%\newtheorem{theorem}{Theorem}[section]
%\newtheorem{lemma}[theorem]{Lemma}
%\newtheorem{conjecture}[theorem]{Conjecture}
%\newtheorem{proposition}[theorem]{Proposition}
%\newtheorem{corollary}[theorem]{Corollary}
%{\theorembodyfont{\rmfamily}
%\newtheorem{remark}[theorem]{Remark}
%\newtheorem{example}[theorem]{Example}
%\newtheorem{definition}[theorem]{Definition}
%\newtheorem{problem}[theorem]{Problem}
%}

\let\h\theta
\let\t\tau

\begin{abstract}
\noindent
We get new results (and rederive some know ones) on smooth surfaces in %Euclidean space 
$\R^n$ by unifying several view points %(some of them new) 
into a coherent general view. Namely, we show and use new relations of the evolute (caustic) 
with the curvature ellipse, the Gauss map and the pseudo-Euclidean geometry of the $3$-space of quadratic forms on $\R^2$. 
A key result (Th.\,\ref{th:focal_quadric-indicatrix}): \textit{for a surface $M$ in $\R^n$ the intersection 
of its caustic with the normal space $N_pM$ is the polar dual hypersurface (in $N_pM$) of the curvature ellipse 
at $p$}. Moreover, all local objects $X$ (cf. the invariants and their relations) have a ``paired'' version $X^*$ 
(with ${X^*}^*=X$) - this provides new results on the original objects. 
\end{abstract} 

\begin{highlights}
\item The natural generalisation of Gaussian curvature at a point $p$ 
of a surface $M$ in $\R^{n\geq 4}$ is a quadratic form $\mathcal{G}_p:N_pM\to\R$
\item The coefficients of characteristic polynomial of the quadratic fom $\mathcal{G}_p:N_pM\to\R$ 
are the usual local invariants of $M$ at $p$ 
\item The eigenvalues of $\mathcal{G}_p$ are alternative local invariants (\textit{principal focal curvatures}) 
\item In $N_pM$, the local caustic $\mathcal{C}_p$ is the polar dual hypersurface of the indicatrix ellipse
\item The local caustic $\mathcal{C}_p$ is a level set of the Gauss quadratic form $\mathcal{G}_p:N_pM\to\R$
\item The local invariants of a surface in $\R^4$ or in $\R^5$ are provided by the pseudo-Euclidean 
$3$-space of quadratic forms on $\R^2$ with ist Lie algebra structure 
\item There are new inequalities between the local invariants for $M$ in $\R^4$ (and $\R^5$) 
\item The local quadratic map $\f_p$ of a surface in $\R^n$ has a ``paired'' quadratic map $\f_p^*$, satisfying ${\f_p^*}^*=\f_p$, 
which has its own paired invariants 
\end{highlights}

\end{frontmatter}
% {\footnotesize
% \noindent 
% \textbf{Keywords}. Surface, optical caustic, evolute, Poisson bracket, singularity, curvature, Gauss map.
% }
% \smallskip
%
\bigskip

\centerline{\textit{To Jos\'e ``Pepe'' Seade in his 65th birthday}}

%%%%%%%%%%%%%%%%%%%%%%%%%%%%%%%%%%%%%%%%%%%%%%%%%%%%%%%%%%%%%%%%%%%%%%%%%%%%%%%%%%%%%%%%%%%%%%%%%%%%%%%%
%%%%%%%%%%%%%%%%%     BODY     %%%%%%%%%%%%%%%%%%%%%%%%%%%%%%%%%%%%%%%%%%%%%%%%%%%%%%%%%%%%%%%%%%%%%%%%%
%%%%%%%%%%%%%%%%%%%%%%%%%%%%%%%%%%%%%%%%%%%%%%%%%%%%%%%%%%%%%%%%%%%%%%%%%%%%%%%%%%%%%%%%%%%%%%%%%%%%%%%%

%%%%%%%%%%%%%%%%%%%%%%%%%%%%%%%%%%%%%%%%%%%%%%%%%%%%%%%%%%%%%%%%%%%%%%%%%%%%%%%%%%%%%%%%%%%%%%%%%%%%%%%
%%%%%%%%%%%%%%%%%%%%%%%%%%%%%%%%%%%%%%%%%%%%%%%%%%%%%%%%%%%%%%%%%%%%%%%%%%%%%%%%%%%%%%%%%%%%%%%%%%%%%%%
\section*{Introduction and Results}
%%%%%%%%%%%%%%%%%%%%%%%%%%%%%%%%%%%%%%%%%%%%%%%%%%%%%%%%%%%%%%%%%%%%%%%%%%%%%%%%%%%%%%%%%%%%%%%%%%%%%%% 
{\footnotesize
	The geometry of submanifolds in $\R^n$ is fundamental in thermodynamics, mechanics, relativity, optics, etc. The simplest local invariants of a smooth submanifold in Euclidean space are the second order 
ones.} %- which depend on the degree 2 terms of the Taylor expansion. 
%
% \begin{example*} 
% For the plane curve $\g:t\mapsto (t,f(t))$, with $f(t)=\frac{1}{2}kt^2+\rm{h.o.t.}(t)$, 
% the curvature at $p=(0,0)$ is $f''(0)=k$. %and $R=1/a$ is the radius of the osculating circle of $\g$ at $p$. 
% \end{example*}
%
\begin{example*} 
For a smooth surface in $\R^3$ locally given by $\g(s,t)=(s,t,f(s,t))$, with 
$f(s,t)=\frac{1}{2}as^2+bst+\frac{1}{2}ct^2+\rm{h.o.t.}(s,t)$, the 
Gauss curvature at $p=(0,0,0)$ is $ac-b^2$. 
%The radii of the ``osculating spheres'' of $M$at $p$ are $R_1=1/u$ and $R_2=1/v$. 
\end{example*}

A smooth surface $M$ in $\R^{2+\ell}$ can be always locally given in Monge form 
\[\g:(s,t)\mapsto (s,t;f_1(s,t),\ldots,f_\ell(s,t))\,,\] 
where $f_i(s,t)=Q_i(s,t)+\mathrm{h.o.t.}(s,t)$ and 
$Q_i(s,t)=\frac{1}{2}(a_is^2+2b_ist+c_it^2)$. 

A surface in $\R^4$ has the following local invariants obtained from $Q_1$ and $Q_2$ (\protect\cite{Little}): 
\smallskip

\noindent 
$\bullet$ \textit{Gaussian curvature} $K=k_1+k_2$, \ \  where $k_1=a_1c_1-b_1^2$ and $k_2=a_2c_2-b_2^2$; 
\smallskip

\noindent 
$\bullet$ \textit{Normal curvature} $N=(a_1-c_1)b_2-(a_2-c_2)b_1$; 
\smallskip

\noindent 
$\bullet$ \textit{Determinant invariant} $\Delta=\frac{1}{4}\left(4(a_1c_1-b_1^2)(a_2c_2-b_2^2)-((a_1c_2+a_2c_1)-2b_1b_2)^2\right)$
\smallskip

\noindent
$\bullet$ \textit{Mean curvature vector} $\Hbf=((a_1+c_1)/2, (a_2+c_2)/2)$. 
\smallskip

\noindent
$\bullet$ An ellipse $\mathcal{E}_p$ in the normal plane $N_pM$, centred at $\Hbf$, called \textit{indicatrix ellipse}. 
\medskip

%The simplest relations between the indicatrix ellipse $\mathcal{E}_p$ and the other four local invariants 
%are described in \protect\cite{Little, Goncalves}. 
In the last 40 years, the singularity theory teams in Universities of Valencia (led by C. Romero-Fuster, J.J. Nu\~no Ballesteros), Hokkaido (led by S. Izumiya, G. Ishikawa), 
Liverpool (led by J.W. Bruce, P. Giblin) and Sao Carlos (led by M.A.S. Ruas, F. Tari) 
made lots of contributions to the theory of surfaces in $\R^n$, collected in the book \protect\cite{IRFRT}. 
Trying to explain to S.\,L\'opez de Medrano those contributions ``up to order two'' %(our initial aim), 
we unify distinct view points 
into a coherent one, discovering new results and insights, and we rederive and ``explain'' some know results from another perspective. 
\smallskip

We start, in \S\ref{sect-Indic-ellipse&point-class}, by recalling the notion of indicatrix ellipse $\E_p$, at $p\in M$, 
and the classification of points of a surface $M$ in terms of the indicatrix ellipse. 

We show, in \S\ref{sect-GaussMap-GQF}, that the natural generalisation of Gaussian curvature at a point $p$ 
of a surface $M$ in $\R^{n\geq 4}$ is not a number but a quadratic form $\mathcal{G}_p:N_pM\to\R$, and that
the coefficients of its characteristic polynomial are local invariants of $M$ at $p$. 
For example, for $M$ in $\R^4$ the characteristic polynomial is $P_{\mathcal{G}_p}(\l)=\l^2-K\l+\Delta$. 
We present the eigenvalues of $\mathcal{G}_p$ as alternative local invariants (\textit{principal focal curvatures}) 
and the eigendirections as characteristic normal directions.  
%We also got local information from the linear map $\flecha{\mathcal{G}_p}:N_pM\to N_pM$ associted to $\mathcal{G}_p$. 

% The quadratic part of the Monge form of $M\subset\R^{2+\ell}$ at $p$, $\phi_p:(s,t)\mapsto(Q_1,\ldots,Q_\ell)$ 
% determines the following \textit{indicatrix ellipse} in $N_pM$: $\E_p(\ubf)=2\f_p(\ubf)$, with $\|\ubf\|=1$. 

In \S\ref{section-C-S}, we recall %the caustic of $M$ in $\R^n$ and 
the \textit{local caustic $\mathcal{C}_p$ of $M\subset\R^{2+\ell}$ at $p$} (formed by the centres of the hyperspheres having higher contact with $M$ at $p$). We show $\mathcal{C}_p$ 
is a quadric in $N_pM$ - we note $\Rbf$ its centre if it exists. 
%if $p$ is not parabolic and $\mathcal{C}_p$ is not degenerate, then the quadric $\mathcal{C}_p$. 
Theorem\,\ref{th:focal_quadric-indicatrix} states: \textit{in $N_pM$, the local caustic $\mathcal{C}_p$ is 
the polar dual hypersurface of the indicatrix ellipse} $\E_p$ and any 
line through $p$ in $N_pM\sqcup\RP^{\ell-1}_\infty$ cuts $\mathcal{C}_p$ at two real points counting multiplicities. 
%and provides geometric properties of $\mathcal{C}_p$.
Most known results on %the local caustic 
$\mathcal{C}_p$ become natural %and \guillemotleft explained\guillemotright\,
(without calculations) after Th.\,\ref{th:focal_quadric-indicatrix}, which also has new corollaries. 

In \S\ref{Section-relation-inequalities}, noting $\flecha{\mathcal{G}_p}:N_pM\to N_pM$ the symmetric linear map 
associated to $\mathcal{G}_p$, we prove that $\Hbf=\flecha{\mathcal{G}_p}\Rbf$, and
get the equalities \,$\langle \Rbf\,,\,\Hbf\rangle=1-N^2/4\Delta$\, 
for $p\in M$ in $\R^4$ and \,$\langle \Rbf\,,\,\Hbf\rangle=1$\, for $p\in M$ in $\R^5$. 
We describe the possible locus of $\Rbf$ and $\Hbf$ in terms of the local invariants. The local caustic $\mathcal{C}_p$ is written as a level set of the 
Gauss quadratic form $\mathcal{G}_p$ and is related to the so-called binormal vectors. 

In \S\ref{Section-Pseudo-Euclidean}, we describe the local invariants of a surface in $\R^4$ or in $\R^5$ in terms 
of the features of the parallelogram (or parallelepiped) formed by $Q_1$, $Q_2$ (or $Q_1$, $Q_2$, $Q_3$) in the pseudo-Euclidean 
$3$-space of quadratic forms on $\R^2$ with ist Lie algebra structure. 

In \S\ref{section:Inequalities-Local-Invariants}, we get new inequalities between the local invariants for $M$ in 
$\R^4$ (and $\R^5$). 

In \S\ref{section:Paired Quadratic Form}, to a local quadratic map $\f_p$ of a surface in $\R^n$ 
(see \S\S\ref{subsect-indic-ellipse&paramtrisation}) we associate a ``paired'' quadratic map $\f_p^*$, satisfying ${\f_p^*}^*=\f_p$, 
which has its own paired invariants ($K^*$, $\Delta^*$, $\E_p^*$, $\mathcal{C}_p^*$, $\mathcal{G}_p^*$, etc.). 
The relations between the original objects and the paired ones (cf. $\Hbf^*=\Rbf$, 
$\flecha{\mathcal{G}_p}(\mathcal{C}_p)=\mathcal{C}_p^*$) lead to new results on the original objects. 
\smallskip

{\footnotesize
\noindent
\textbf{Note}. Most results of this paper have natural generalisations for $k$-dimensional submanifolds in $\R^{k+\ell}$,  
which will appear in a forthcoming paper \protect\cite{Uribe_k-dim_submnfd-Rn}. Those general results arise from basic properties 
on the singularities of homogeneous quadratic maps $\R^k\to\R^\ell$, to appear in \protect\cite{Uribe_QM-CVS-D}.  
}
\smallskip

\noindent
{\footnotesize 
\textbf{Acknowledgements.} These investigations answer a series of questions of Santiago L\'opez de Medrano 
(I'm very grateful to him) and were initiated and mostly done in Laboratory UMI2001 CNRS Solomon Lefschetz UNAM M\'exico 
(I'm very grateful to Jos\'e Seade for his kind invitation). I'm grateful to Efrain Vega Landa for helpful conversations and questions, 
and to F.\,S\'anchez Bringas who let me know his paper \protect\cite{Bayard_Mendez_S-B} on surfaces in $\R^5$ (where some invariants, 
including $\Delta$ and $\tau$, were algebraically found). %and an interesting ellipse in $N_pM$, 
%noted $\delta$, was introduced (one obtains it from $\E_p^*$ and $\tau$, namely,  $\delta=\tau\,\E_p^*$). 
} 
\newpage

{\footnotesize
\noindent 
\textbf{Contents}
\medskip

\noindent
\textbf{1 Indicatrix Ellipse and Classification of Points}
\smallskip
% $\bullet$ Indicatrix Ellipse $\E_p$ and its Parametrisation. 
% \smallskip
% $\bullet$ Classification of Points of $M\subset\R^4$ by the Indicatrix Ellipse $\E_p$.
% \smallskip

\noindent
\textbf{2 Gauss Map, ``Gauss Quadratic Form'' $\mathcal{G}_p$ and Local Invariants of Surfaces}

$\bullet$ The Gauss Quadratic Form $\mathcal{G}_p$. 

$\bullet$ The Matrix of the Gauss Quadratic Form $\mathcal{G}_p$ and Local Invariants of $M$. 

$\bullet$ Principal Normal Directions and Principal Focal Curvatures. 

$\bullet$ The Paired Quadratic Form $\mathcal{G}_p^*$ of the Gauss Quadratic Form $\mathcal{G}_p$. 
\smallskip

\noindent
\textbf{3 Optical Caustic and Duality Between the Local Caustic and the Indicatrix Ellipse} 

$\bullet$ Focal Set and Optical Caustic. 

$\bullet$ The Polar Dual of a Subvariety.

$\bullet$ Duality Between the Local Caustic $\mathcal{C}_p$ and the Indicatrix Ellipse $\E_p$ for sufaces in $\R^n$. 

$\bullet$ Classification of Points of a Surface in $\R^4$ ($\R^5$) by the Local Caustic. 
\smallskip

\noindent
\textbf{4 Relations Between $\E_p$, $\Hbf$, $\mathcal{C}_p$, $\Rbf$ and the Local Invariants}

$\bullet$ The Relation Between $\Hbf$ and $\Rbf$ by the Gauss Quadratic Form. 

$\bullet$ The Locus of the Centres $\Hbf$ and $\Rbf$ in Terms of the Local Invariants.

$\bullet$ The Local Caustic $\mathcal{C}_p$ as Level Set of the Gauss Quadratic Form $\mathcal{G}_p$. 

$\bullet$ Local Caustic Related to Degenerate and Binormal Vectors. 
\smallskip

\noindent
\textbf{5 Pseudo-Euclidean Geometry of Quadratic Forms on $\R^2$ and Local Invariants of Surfaces}

$\bullet$ Pseudo-Euclidean Space of Quadratic Forms on $\R^2$. 

$\bullet$ Poisson Bracket vs Vector Product. 

$\bullet$ From Pseudo-Orthogonality in $\R^3$ to Polar Duality in $\RP^2$.

$\bullet$ Local Invariants of Surfaces in $\R^4$ ($\R^5$) Via Pseudo-Euclidean Geometry of Quadratic Forms. 
\smallskip

\noindent
\textbf{6 Inequalities Between Local Invariants of Surfaces in $\R^4$ and in $\R^5$}

$\bullet$ Inequalities that Restrict the Local Invariants of Surfaces in $\R^4$ 

$\bullet$ Inequalities Bounding the Local Invariants of Surfaces in $\R^5$ and $\R^n$ 
\smallskip

\noindent
\textbf{7 The Paired Quadratic Map, its Paired Quadratic Form $\mathcal{G}_p^*$ and its Invariants}

$\bullet$ Legendre Transform and Paired Quadratic Map. 

$\bullet$ Paired Indicatrix $\E_p^*$, Paired Local Caustic $\mathcal{C}_p^*$ and Paired Invariants. 

$\bullet$ Relations Between $\E_p$, $\mathcal{C}_p$ and their Paired Versions $\E_p^*$, $\mathcal{C}_p^*$.  

$\bullet$ The Cones $\Sigma_p$, $\Sigma_p^*$ related to $\E_p$, $\E_p^*$ and to $K$, $K^*=\mathcal{A}/\Delta$ for $M$ in $\R^5$.

$\bullet$ Relations Between the Components of Two Paired Quadratic Maps. 
}

%\tableofcontents

%%%%%%%%%%%%%%%%%%%%%%%%%%%%%%%%%%%%%%%%%%%%%%%%%%%%%%%%%%%%%%%%%%%%%%%%%%%%%%%%%%%%%%%%%%%%%%%%%%%%%%%%%%%%%%%%%%%
\section{Indicatrix ellipse and classification of points}\label{sect-Indic-ellipse&point-class}
%%%%%%%%%%%%%%%%%%%%%%%%%%%%%%%%%%%%%%%%%%%%%%%%%%%%%%%%%%%%%%%%%%%%%%%%%%%%%%%%%%%%%%%%%%%%%%%%%%%%%%%%%%%%%%%%%%% 
Along this paper we  shall use the fact that a $k$ dimensional smooth submanifold in $\R^n$ is locally the 
graph of a map $f:\R^k\to \R^{n-k}$ (by the Implicit Function Theorem). 
%%%%%%%%%%%%%%%%%%%%%%%%%%%%%%%%%%%%%%%%%%%%%%%%%%%%%%%%%%%%%%%%%%%%%%%%%%%%%%%%%%%%%%%%%%%%%%%%%%%%%%%%%%%%%%%%%%%%%%%%%%%%%%
\subsection{\textbf{Indicatrix ellipse $\mathcal{E}_p$ and its parametrisation}}\label{subsect-indic-ellipse&paramtrisation}
%%%%%%%%%%%%%%%%%%%%%%%%%%%%%%%%%%%%%%%%%%%%%%%%%%%%%%%%%%%%%%%%%%%%%%%%%%%%%%%%%%%%%%%%%%%%%%%%%%%%%%%%%%%%%%%%%%%%%%%%%%%%%%
Consider a point $p$ of a smooth surface $M$ in $\R^{2+\ell}$ and take an orthonormal frame 
$\{\ebf_1, \ebf_2, \nbf_1,\ldots,\nbf_\ell\}$ of $\R^{2+\ell}$ such that $\{\ebf_1, \ebf_2\}$ is a basis of the 
tangent plane $T_pM$ and $\{\nbf_1,\ldots,\nbf_\ell\}$ is a basis of the normal space $N_pM$. 
Then the local Monge form of $M$ at $p$ (the origin) is given by: 
\[\g:(s,t)\mapsto (s,t; f_1(s,t),\ldots,f_\ell(s,t))\,, \qquad \g(0,0)=p\,,\] 
where the functions $f_1(s,t),\ldots,f_\ell(s,t)$ have vanishing $1$-jet at $(0,0)$ and Taylor expansion 
$f_j(s,t)=Q_j(s,t)+\rm{h.o.t.}(s,t)$,\, with $Q_j=\frac{1}{2}(a_js^2+2b_jst+c_jt^2)$.

The second order invariants of $M$ at $p$ depend only on the quadratic part of $\g$:
\smallskip

\noindent 
\textit{\textbf{\small Local Quadratic Map}}. The quadratic part $\f_p:T_pM\to N_pM$ of the local Monge form at $p$, 
$\f_p:(s,t):\mapsto (Q_1,\ldots, Q_\ell)$ is called the \textit{local quadratic map} of $M$ at $p$. 
Of course, its components $Q_1, \ldots, Q_\ell$ depend on the choice 
of the frame $\nbf_1,\ldots,\nbf_\ell$ on $N_pM$. 
\medskip

\noindent
\textbf{Notation}. 
The $\ell$ quadratic forms $Q_j$ determine the following 3 vectors of %the normal space 
$N_pM$: 
\begin{equation}\label{eq:vectors-ABC}
\Abf=a_1\nbf_1+\ldots+a_\ell\nbf_\ell\,, \quad 
\Bbf=b_1\nbf_1+\ldots+b_\ell\nbf_\ell\,, \quad \Cbf=c_1\nbf_1+\ldots+c_\ell\nbf_\ell\,.  
\end{equation}

Using these vectors, the Taylor expansion of the Monge form of $M$ at $p$ is written as 
\begin{equation}\label{eq:simple-Monge-form}
\g(s,t)=s\ebf_1+t\ebf_2+\frac{1}{2}\left(s^2\Abf+2st\Bbf+t^2\Cbf\right)+{\rm h.o.t.}(s,t)\,. 
\end{equation}

Thus the local quadratic map $\f_p:T_pM\to N_pM$ is expressed as 
\begin{equation}\label{eq:LQM}
  \f_p(s,t)=(Q_1,\ldots,Q_\ell)=\frac{1}{2}\left(s^2\Abf+2st\Bbf+t^2\Cbf\right)\,.
\end{equation}

\noindent 
\textit{\textbf{\small Indicatrix ellipse}}. 
Given a unit vector $\ubf_\theta=u_1\ebf_1+u_2\ebf_2\in T_pM$, with $\ubf_1=\cos\theta$, 
$\ubf_2=\sin\theta$ for some $\theta\in\R$, write $\a_{\theta}$ for the normal section 
of $M$ in the direction of $\ubf_\theta$ and denote $\kbf(\ubf_\theta)$ the curvature vector of $\a_{\theta}$ at $p$. 
The \textit{indicatrix set} $\mathcal{E}_p$ of $M$ at $p$ is defined as the set of curvature vectors 
$\kbf(\ubf_\theta)$ taken over all unit vectors $\ubf_\theta\in T_pM$.

%%%%%%%%%%%%%%%%%%%%%%%%%%%%%%%%%%%% ELLIPSE INDICATRIX %%%%%%%%%%%%%%%%%%%%%%%%%%%%%%%%%%%%%%%%%%%%%%%%%%%
\begin{proposition*}[\protect\cite{Moore-Wilson}]
The indicatrix set $\mathcal{E}_p$ is an ellipse {\rm (possibly singular)}. 
\end{proposition*}
\begin{proof}
Taking the curve $\a_{\theta}(r)=\g(r\ubf_\theta)$ on $M$, the curvature vector $\kbf(\ubf_\theta)$
is equal to $\a_{\theta}''(0)$, that is $\kbf(\ubf_\theta)=u_1^2\Abf+2u_1u_2\Bbf+u_2^2\,\Cbf$. 
Using the equality $u_1^2+u_2^2=1$, we get 
\[\kbf(\ubf_\theta)=\frac{\Abf+\Cbf}{2}+(u_1^2-u_2^2)\frac{\Abf-\Cbf}{2}+2u_1u_2\Bbf\,.\]
The equality $\kbf(\ubf_\theta)=\kbf(-\ubf_\theta)$ means that the indicatrix set is parametrised by the 
projective line $\RP^1=\{\theta\in\R\mod \pi\}$, $\E_p: \RP^1\to N_pM$: 
\begin{equation}\label{eq:parametrisation-E_p}
\E_p(\theta)=\frac{\Abf+\Cbf}{2}+\cos 2\theta\frac{\Abf-\Cbf}{2}+\sin 2\theta \Bbf\,.
\end{equation}
The point $[u_1:u_2]\in\RP^1=\sph^1/\{\pm 1\}$ corresponds to the line $\ell_\theta$ 
directed by $\pm\ubf_\theta$. 
\end{proof}

\noindent
\textbf{\small Notation}. We note $\E_p$ the indicatrix ellipse (as a set) and also its 
parametrisation \eqref{eq:parametrisation-E_p}.
Its double covering, defined on $\sph^1=\{\ubf\in T_pM:\|u\|=1\}$, is noted $\hat{\E}_p:\sph^1\to N_pM$. Notice that  
\begin{equation}\label{eq:double-covering-E_p}
 \hat{\E}_p(\ubf):=2\f_p(\ubf)\qquad \mbox{with }\, \ubf\in T_pM, \hspace*{0.4cm} \|\ubf\|=1 \,. 
\end{equation}

%%%%%%%%%%%%%%%%%%%%%%%%%%%%%%%%%%%%%%%%%%%%  DEFINITION  %%%%%%%%%%%%%%%%%%%%%%%%%%%%%%%%%%%%%%%%%%%%%%%%%%%%%%%%%
%\subsection{\textbf{Classification of Points of a Surface by the Indicatrix Ellipse $\E_p$}} 
%%%%%%%%%%%%%%%%%%%%%%%%%%%%%%%%%%%%%%%%%%%%%%%%%%%%%%%%%%%%%%%%%%%%%%%%%%%%%%%%%%%%%%%%%%%%%%%%%%%%%%%%%%%%%%%%%%%
\noindent
\textbf{Types of points of surfaces in $\R^4$} \protect\cite{Little}. 
A point $p\in M$ is called \textit{elliptic} (\textit{hyperbolic, parabolic}) iff the 
ellipse $\E_p$ is non degenerate and $p$ lies inside (resp. outside, on) $\E_p$, and  $p$ is called 
\textit{semiumbilic} iff $\E_p$ degenerates to a non radial segment. If $\E_p$ is a radial segment, 
then $p$ is called \textit{inﬂection}. An inﬂection point is of \textit{real type} 
(\textit{imaginary type}, \textit{ﬂat}) if $p$ is an interior point of $\E_p$ 
(resp. does not belong to $\E_p$, is an end point of $\E_p$). 

The point $p$ is called \textit{umbilic} if the indicatrix ellipse $\E_p$ degenerates to a point distinct from $p$, 
and the point $p$ is called a \textit{flat umbilic} if $\E_p$ reduces to $p$. Moreover (\protect\cite{Little}),
\begin{equation}\label{eq:Delta>0-elliptic}
  \mbox{\textit{$p$ is elliptic iff }}\, \Delta>0 \quad \mbox{ and } \quad   \mbox{\textit{$p$ is hyperbolic  iff }} \, \Delta<0\,. 
\end{equation} 

A classification of points of generic surfaces in $\R^5$, up to order to, is given in \S\,\ref{sect-Class-Points-R5}.

% % passes through $p$, the point $p$ is called \textit{flat-elliptic} (resp. \textit{-hyperbolic}, \textit{-parabolic}) 
% % if it lies inside (resp. outside, on) $\E_p$. A point $p$ is called \textit{semi-umbilic} if 
% % $\E_p$ is a non radial segment. 
% If $\E_p$ is a radial segment, then $p$ is called \textit{real} (resp. \textit{imaginary}, 
% \textit{flat}) \textit{inflection} if $p$ is an interior point of $\E_p$ (resp. does not belong to $\E_p$, is an end-point 
% of $\E_p$). A point $p$ is called \textit{umbilic} if $\E_p$ reduces to a point distinct from $p$, and $p$ is called \textit{flat umbilic} if 
% $\E_p$ reduces to $p$. 

%%%%%%%%%%%%%%%%%%%%%%%%%%%%%%%%%%%%%%%%%%%%%%%%%%%%%%%%%%%%%%%%%%%%%%%%%%%%%%%%%%%%%%%%%%%%%%%%%%%%%%%%%%%%%%%%%%%%%%%%%
%%%%%%%%%%%%%%%%%%%%%%%%%%%%%%%%%%%%%%%%%%%%%%%%%%%%%%%%%%%%%%%%%%%%%%%%%%%%%%%%%%%%%%%%%%%%%%%%%%%%%%%%%%%%%%%%%%%%%%%%%
\section{Gauss Map and its ``Gauss Quadratic Form'' $\mathcal{G}_p$}\label{sect-GaussMap-GQF}
%%%%%%%%%%%%%%%%%%%%%%%%%%%%%%%%%%%%%%%%%%%%%%%%%%%%%%%%%%%%%%%%%%%%%%%%%%%%%%%%%%%%%%%%%%%%%%%%%%%%%%%%%%%%%%%%%%%%%%%%%
%%%%%%%%%%%%%%%%%%%%%%%%%%%%%%%%%%%%%%%%%%%%%%%%%%%%%%%%%%%%%%%%%%%%%%%%%%%%%%%%%%%%%%%%%%%%%%%%%%%%%%%%%%%%%%%%%%%%%%%%%

%%%%%%%%%%%%%%%%%%%%%%%%%%%%%%%%%%%%%%%%%%%%%%%%%%%%%%%%%%%%%%%%%%%%%%%%%%%%%%%%%%%%%%%%%%%%%%%%%%%%%%%%%%%%%%%%%%%%%%%%%
\subsection{\textbf{The Gauss Quadratic Form}}\label{sect-Gauss-quad-form}
%%%%%%%%%%%%%%%%%%%%%%%%%%%%%%%%%%%%%%%%%%%%%%%%%%%%%%%%%%%%%%%%%%%%%%%%%%%%%%%%%%%%%%%%%%%%%%%%%%%%%%%%%%%%%%%%%%%%%%%%%
The \textit{unit normal bundle} of $M$ is the set of all unit vectors that are normal to $M$: 
\[N^1M=\{(p,\nbf):p\in M,\, \nbf\in N_pM\, \mbox{ and }\, \|\nbf\|=1\}\,.\] 
The \textit{Gauss map} of a submanifold $M$ in $\R^{n}$ is a map from the unit normal bundle of $M$ 
to the unit hypersphere $\sph^{n-1}$ of $\R^n$, 
$\G:N^1M\to\sph^{n-1}$. It translates the unit normal vector $\nbf\in N_pM$, associated to $(p,\nbf)\in N^1M$,  
to the origin, giving a vector of $\sph^{n-1}$, 
\[\G:(p,\nbf)\mapsto\nbf\,.\]

We shall consider a surface $M$ in $\R^{n=2+\ell}$. Notice that the fibre of the natural projection $N^1M\to M$, 
over a point $p\in M$, is the sphere $\sph_p^{\ell-1}$ of unit vectors normal to $M$ at $p$, 
and the restriction of the Gauss map to the fibre over $p$ is the translation 
of the sphere $\sph_p^{\ell-1}$ into $\sph^{n-1}$. 
Thus, the differential along the tangent space to the fibre $\sph^{\ell-1}_p$ is the identity, 
and hence \textit{the differentiation is relevant only along the directions orthogonal to the fibre}. 
Hence the matrix of the differential has the form 
\[
[d_{(p,\nbf)}\Gamma]=\begin{pmatrix}
 [d_{(p,\nbf)}^\perp\Gamma]  & \Obf_{2\times\ell} \\ 
 \Obf_{\ell\times 2}  & \Idbf_{\ell-1}
\end{pmatrix}\,.\]

Therefore, the determinant of the Jacobian matrix of the Gauss map is equal to the determinant of 
its $2\times 2$ sub-matrix $[d_{(p,\nbf)}^\perp\Gamma]$.
% \[
% [d_{(p,\nbf)}\Gamma]=\begin{pmatrix}
%  \langle \Abf\,, \, \nbf\rangle  &  \langle \Bbf\,, \, \nbf\rangle & \Obf_{1\times \ell-1} \\ 
%   \langle \Bbf\,, \, \nbf\rangle & \langle \Cbf\,, \, \nbf\rangle  & \Obf_{1\times \ell-1} \\
%   \Obf_{\ell-1\times 1} & \Obf_{\ell-1\times 1} & \Idbf_{\ell-1}
% \end{pmatrix}\,,
%\]
If $M$ has local quadratic map \eqref{eq:LQM} at $p$ and 
$\nbf=\nu_1\nbf_1+\ldots+\nu_\ell\nbf_\ell$ (with $\nu_1^2+\ldots+\nu_\ell^2=1$), then a direct computation shows that 
\[\left[d_{(p,\nbf)}^\perp\Gamma\right]=\nu_1\begin{pmatrix}
 a_1 & b_1\\ 
 b_1 & c_1  
\end{pmatrix}+\ldots+\nu_\ell\begin{pmatrix}
 a_\ell & b_\ell\\ 
 b_\ell & c_\ell  
\end{pmatrix}\,, \,\mbox{ that is,}\]

\begin{equation}\label{eq:[dG(p,n)]}
  \left[d_{(p,\nbf)}^\perp\Gamma\right]=\begin{pmatrix}
 \langle \Abf\,, \, \nbf\rangle & \langle \Bbf\,, \, \nbf\rangle  \\ 
 \langle \Bbf\,, \, \nbf\rangle & \langle \Cbf\,, \, \nbf\rangle  
\end{pmatrix}\,.
\end{equation}
The eigenvalues of $[d_{(p,\nbf)}^\perp\Gamma]$ are the \textit{principal curvatures of $M$ in the 
direction of the normal vector $\nbf$}: they are the usual principal curvatures at $p$ of the orthogonal projection 
of $M$ on the $3$-space spanned by $\ebf_1, \ebf_2, \nbf$. Thus the Gauss curvature $K^{\nbf}$ of that projected surface 
is the determinant of $[d_{(p,\nbf)}^\perp\Gamma]$. 
\smallskip

%%%%%%%%%%%%%%%%%%%%%%%%%%%%%%%%%%%%%%%%%%%%%%%%%%%%%%%%%%%%%%%%%%%%%%%%%%%%%%%%%%%%%%%%%%%%%%%%%%%%%%%%%%%%%%%%%%%
%%%%%%%%%%%%%%%%%%%%%%%%%%%%%%%%%%%%%%%%%%%%  DEFINITION  %%%%%%%%%%%%%%%%%%%%%%%%%%%%%%%%%%%%%%%%%%%%%%%%%%%%%%%%%
\noindent
\textbf{\textit{\small Gauss Quadratic Form}}. 
For each $p$ in $M$ the determinant of the matrix $[d_{(p,\nbf)}^\perp\Gamma]$ is the quadratic form of the 
variable $q\in N_pM$
\begin{equation}\label{eq:CQF(q)}
\mathcal{G}_p:N_p M\to\R, \qquad  \mathcal{G}_p(q)=\frac{1}{2}\det\begin{pmatrix}
 \langle \Abf\,, \, q\rangle & \langle \Bbf\,, \, q\rangle  \\ 
 \langle \Bbf\,, \, q\rangle & \langle \Cbf\,, \, q\rangle  
\end{pmatrix}\,,
\end{equation}
evaluated at $q=\nbf$. We call the form \eqref{eq:CQF(q)} \textit{Gauss quadratic form} of $M$ at $p$. 
It provides the Gauss curvature $K^{\nbf}$ 
in the direction of $\nbf$, that is, $2\mathcal{G}_p(\nbf)=K^{\nbf}$. We write $[\mathcal{G}_p]$ 
for its Hessian matrix and $\flecha{\mathcal{G}_p}:N_pM\to N_pM$ for its associated symmetric linear map. 
\smallskip

Observe that $\mathcal{G}_p$ is determined by the local quadratic map $\f_p$ of $M$ at $p$. 

\begin{remark*}
The Gauss curvature at a point of a surface in $\R^3$ is a real number $K$ that 
should be considered as the quadratic form $2\mathcal{G}_p(q)=Kq^2$ on the normal line $N_pM=\R$, 
but restricted to the unit normal vectors $+1$ and $-1$: $2\mathcal{G}_p(\pm 1)=K(\pm 1)^2=K$. 
\end{remark*}

%%%%%%%%%%%%%%%%%%%%%%%%%%%%%%%%%%%%%%%%%%%%%%%%%%%%%%%%%%%%%%%%%%%%%%%%%%%%%%%%%%%%%%%%%%%%%%%%%%%%%%%%%%%%%%%%%%%%%%%
\subsection{\textbf{Matrix of the Gauss Quadratic Form and Local Invariants of $M$}}
%%%%%%%%%%%%%%%%%%%%%%%%%%%%%%%%%%%%%%%%%%%%%%%%%%%%%%%%%%%%%%%%%%%%%%%%%%%%%%%%%%%%%%%%%%%%%%%%%%%%%%%%%%%%%%%%%%%%%%%
We shall identify the vector space of quadratic forms $Q=\frac{1}{2}(as^2+2bst+ct^2)$ 
with $\R^3=\{(a,b,c)\}$, endowed with the pseudo-scalar product 
\begin{equation}\label{eq:psi-scalar-prod}
  \langle Q_i, Q_j \rangle_\psi:=(a_ic_j+a_jc_i)/2-b_ib_j\,.
\end{equation}

Consider a surface $M$ in $\R^4$ (or in $\R^5$) with local quadratic map \eqref{eq:LQM} at $p$. 

%%%%%%%%%%%%%%%%%%%%%%%%%%%%%%%%%%%%%%%%%%%%%%%%%%%%%%%%%%%%%%%%%%%%%%%%%%%%%%%%%%%%%%%%%%%%%%%%%%%%%%%%%%%%%%%%%%%%
%%%%%%%%%%%%%%%%%%%%%%%%%%%%%%%%%%%%%%%%%%%%  PROPOSITION  %%%%%%%%%%%%%%%%%%%%%%%%%%%%%%%%%%%%%%%%%%%%%%%%%%%%%%%%%
%%%%%%%%%%%%%%%%%%%%%%%%%%%%%%%%%%%%%%%%%%%%%%%%%%%%%%%%%%%%%%%%%%%%%%%%%%%%%%%%%%%%%%%%%%%%%%%%%%%%%%%%%%%%%%%%%%%%
\begin{proposition}
The matrix of the Gauss quadratic form $\mathcal{G}_p$ is the Gram matrix of the vectors 
$Q_1$, $Q_2$ $($or $Q_1$, $Q_2$, $Q_3$ for $M$ in $\R^5$$)$, for the pseudo-scalar product $\langle\cdot\,, \cdot\rangle_\psi:$
\begin{equation}\label{eq:Gp=[matrix]}
  [\mathcal{G}_p]=
\begin{pmatrix}
  k_1 & k_{12}\\ 
  k_{12} & k_2
\end{pmatrix}
\quad \mbox{for } M\subset\R^4\,, \quad 
[\mathcal{G}_p]=
\begin{pmatrix}
  k_1    & k_{12} & k_{13} \\ 
  k_{12} & k_2    & k_{23} \\
  k_{13} & k_{23} & k_3 
\end{pmatrix}
\quad \mbox{for } M\subset\R^5\,, 
\end{equation}
where $k_{ij}=\langle Q_i, Q_j \rangle_\psi$\, and \,$k_i:=k_{ii}=\langle Q_i, Q_i \rangle_\psi$. 
\end{proposition}

\begin{proof}
One gets the quadratic form $\mathcal{G}_p$ and its matrix from \eqref{eq:CQF(q)} by direct calculation. 
\end{proof}

%%%%%%%%%%%%%%%%%%%%%%%%%%%%%%%%%%%%%%%%%%%%%%%%%%%%%%%%%%%%%%%%%%%%%%%%%%%%%%%%%%%%%%%%%%%%%%%%%%%%%%%%%%%%%%%%
\subsubsection{\textbf{Local Invariants of Surfaces in $\R^4$ and in $\R^5$}} 
%%%%%%%%%%%%%%%%%%%%%%%%%%%%%%%%%%%%%%%%%%%%%%%%%%%%%%%%%%%%%%%%%%%%%%%%%%%%%%%%%%%%%%%%%%%%%%%%%%%%%%%%%%%%%%%%
%%%%%%%%%%%%%%%%%%%%%%%%%%%%%%%%%%%%%%%%%%%%%%%%%%%%%%%%%%%%%%%%%%%%%%%%%%%%%%%%%%%%%%%%%%%%%%%%%%%%%%%%%%%%%%%%%%%%
%%%%%%%%%%%%%%%%%%%%%%%%%%%%%%%%%%%%%%%%%%%%  PROPOSITION  %%%%%%%%%%%%%%%%%%%%%%%%%%%%%%%%%%%%%%%%%%%%%%%%%%%%%%%%%
%%%%%%%%%%%%%%%%%%%%%%%%%%%%%%%%%%%%%%%%%%%%%%%%%%%%%%%%%%%%%%%%%%%%%%%%%%%%%%%%%%%%%%%%%%%%%%%%%%%%%%%%%%%%%%%%%%%%
\begin{proposition}[On surfaces in $\R^4$]\label{prop:inv-char_polynomial}
For $p\in M$ in $\R^4$ the local invariants $\Delta$ and $K$ are the coefficients of the characteristic polynomial 
of the Gauss quadratic form $\mathcal{G}_p$$:$
\[P_{\mathcal{G}_p}(\l)=\l^2-K\l+\Delta\,.\] 
\end{proposition}

\noindent
\textit{Proof}.\, 
The intrinsic Gauss curvature $k_1+k_2$ is, indeed, the trace of $[\mathcal{G}_p]$ - see \eqref{eq:Gp=[matrix]}$:$
\[
\Tr[\mathcal{G}_p]=k_1+k_2=K \quad \text{(Gauss curvature)\,,}
\]
and one directly verifies that the invariant $\Delta$ (of \protect\cite{Little}) is equal to the determinant of $[\mathcal{G}_p]$:
\[
\det[\mathcal{G}_p]= \,k_1k_2-k_{12}^2\, =\Delta \quad \text{(determinant invariant)}\,.
\eqno{\square}\]

Since the Gauss map and its differential are geometrical objects (i.e. they are independent of our choice of orthonormal basis), 
the Gauss quadratic form $\mathcal{G}_p$ is a well defined local invariant of $M$ at $p$. This implies the 
 
%%%%%%%%%%%%%%%%%%%%%%%%%%%%%%%%%%%%%%%%%%%%%%%%%%%%%%%%%%%%%%%%%%%%%%%%%%%%%%%%%%%%%%%%%%%%%%%%%%%%%%%%%%%%%%%%%%%%
%%%%%%%%%%%%%%%%%%%%%%%%%%%%%%%%%%%%%%%%%%%%  PROPOSITION  %%%%%%%%%%%%%%%%%%%%%%%%%%%%%%%%%%%%%%%%%%%%%%%%%%%%%%%%%
%%%%%%%%%%%%%%%%%%%%%%%%%%%%%%%%%%%%%%%%%%%%%%%%%%%%%%%%%%%%%%%%%%%%%%%%%%%%%%%%%%%%%%%%%%%%%%%%%%%%%%%%%%%%%%%%%%%%
\begin{proposition}[On surfaces in $\R^5$]\label{prop:inv-char_polynomial5}
For a point $p$ of a surface $M$ in $\R^5$ the coefficients of the characteristic polynomial of the Gauss 
quadratic form $\mathcal{G}_p$, 
\[P_{\mathcal{G}_p}(\l)=-\l^3+K\l^2-\mathcal{A}\l+\Delta\,,\] 
are three independent local invariants of $M$ at $p$: 
\begin{equation}
K=\Tr[\mathcal{G}_p]=k_1+k_2+k_3 \quad (\text{Gauss curvature})\,,
\end{equation}
\begin{equation}\label{eq:invariant-A}
\mathcal{A}=(k_1k_2-k_{12}^2)+(k_2k_3-k_{23}^2)+(k_3k_1-k_{31}^2)
\end{equation}
\begin{equation}
\Delta=\det[\mathcal{G}_p] \quad (\text{determinant invariant})\,.
\end{equation}
\end{proposition}

% \begin{remark*}
% The intrinsic Gauss curvature $K=k_1+k_2+k_3$ is known. 
% \end{remark*}

For a surface in $\R^5$ there is an invariant related to $\Delta$, determined by the vectors that take part in the 
parametrisation \eqref{eq:parametrisation-E_p} of the indicatrix ellipse,  $\frac{1}{2}(\Abf-\Cbf)$, $\Bbf$ and $\Hbf$: 
\medskip

\noindent
\textbf{\textit{\small Torsion}}. We call \textit{torsion} of $M$ at $p$ the oriented volume of the parallelepiped formed by 
the vectors $\frac{1}{2}(\Abf-\Cbf)$, $\Bbf$ and $\Hbf$: 
\begin{equation}\label{eq:def-torsion}
\tau:=\left\langle\, \textstyle\frac{1}{2}(\Abf-\Cbf)\times\Bbf\,,\, \Hbf\, \right\rangle\,. 
\end{equation}

The reader can check (and below we prove it) that $\tau^2=\Delta$. Thus $\Delta\geq 0$. 

%%%%%%%%%%%%%%%%%%%%%%%%%%%%%%%%%%%%%%%%%%%%%%%%%%%%%%%%%%%%%%%%%%%%%%%%%%%%%%%%%%%%%%%%%%%%%%%%%%%%%%%%%%%%%%%%%%%%%%%%%%%
\subsection{\textbf{Principal Normal Directions and Principal Focal Curvatures}}
%%%%%%%%%%%%%%%%%%%%%%%%%%%%%%%%%%%%%%%%%%%%%%%%%%%%%%%%%%%%%%%%%%%%%%%%%%%%%%%%%%%%%%%%%%%%%%%%%%%%%%%%%%%%%%%%%%%%%%%%%%%
%%%%%%%%%%%%%%%%%%%%%%%%%%%%%%%%%%%%%%%%%%%%%%%%%%%%%%%%%%%%%%%%%%%%%%%%%%%%%%%%%%%%%%%%%%%%%%%%%%%%%%%%%%%%%%%%%%%
%%%%%%%%%%%%%%%%%%%%%%%%%%%%%%%%%%%%%%%%%%%%  DEFINITION  %%%%%%%%%%%%%%%%%%%%%%%%%%%%%%%%%%%%%%%%%%%%%%%%%%%%%%%%%
At most points of a generic surface there is a well defined characteristic normal frame: 
\smallskip

\noindent
\textbf{\textit{\small Principal Normal Directions}}. For $M$ in $\R^4$ or in $\R^5$, 
we call \textit{principal normal directions} of $M$ at $p$ 
the eigendirections of the Gauss quadratic form $\mathcal{G}_p:N_pM\to\R$. 
\smallskip

\noindent
\textbf{\textit{\small Principal Basis}}. A basis of the normal space $N_pM$ is said to be 
\textit{principal} if it is formed by unit eigenvectors of the Gauss quadratic form $\mathcal{G}_p$.
\smallskip

Take a principal basis of $N_pM$, $\nbf_1$, $\nbf_2$, (or $\nbf_1$, $\nbf_2$, $\nbf_3$ for $M$ 
in $\R^5$) and take a normal vector $q=q_1\nbf_1+q_2\nbf_2$ (or $q=q_1\nbf_1+q_2\nbf_2+q_3\nbf_3$). 
The Gauss quadratic form $\mathcal{G}_p$ is expressed in terms of its eigenvalues $\mathcal{K}_1$, $\mathcal{K}_2$ 
(or $\mathcal{K}_1$, $\mathcal{K}_2$, $\mathcal{K}_3$) as:
\[
\mathcal{G}_p(q)=\frac{1}{2}\left(\,\mathcal{K}_1q_1^2+\mathcal{K}_2q_2^2\,\right) \qquad \mbox{for $M$ in } \R^4. 
\]
\[
\mathcal{G}_p(q)=\frac{1}{2}\left(\,\mathcal{K}_1q_1^2+\mathcal{K}_2q_2^2+\mathcal{K}_3q_3^2\,\right) \qquad \mbox{for $M$ in } \R^5. 
\]

%%%%%%%%%%%%%%%%%%%%%%%%%%%%%%%%%%%%%%%%%%%%%%%%%%%%%%%%%%%%%%%%%%%%%%%%%%%%%%%%%%%%%%%%%%%%%%%%%%%%%%%%%%%%%%%%%%%
%%%%%%%%%%%%%%%%%%%%%%%%%%%%%%%%%%%%%%%%%%%%  DEFINITION  %%%%%%%%%%%%%%%%%%%%%%%%%%%%%%%%%%%%%%%%%%%%%%%%%%%%%%%%%
\noindent
\textbf{\textit{\small Principal Focal Curvatures}}. We call \textit{principal focal curvatures} of $M$ at $p$ the eigenvalues 
$\mathcal{K}_1$, $\mathcal{K}_2$ (or $\mathcal{K}_1$, $\mathcal{K}_2$, $\mathcal{K}_3$) of the 
Gauss quadratic form $\mathcal{G}_p$. 

\begin{proposition}
For a smooth surface $M$ in $\R^4$ $($or in $\R^5$$)$ the principal focal curvatures are independent local invariants.
\end{proposition}

\noindent
\textit{Proof}.  
The invariants $K$, $\Delta$ for $M\subset\R^4$ or $K$, $\mathcal{A}$ and $\Delta$ 
for $M\subset\R^5$ are expressed as the elementary symmetric functions of 
the principal focal curvatures: 
\begin{equation}\label{eq:Delta=K1K1-K=K1+K2}
\R^4: \qquad \qquad \Delta=\K_1\K_2\,, \qquad K=\K_1+\K_2\,.\qquad 
\end{equation}
\[
\R^5: \qquad  \Delta=\K_1\K_2\K_3\,, \quad \mathcal{A}=\K_1\K_2+\K_2\K_3+\K_3\K_1\,, \quad K=\K_1+\K_2+\K_3\,. \eqno{\square}
\]

\begin{remark*}
The principal focal curvatures $\K_1$, $\K_2$ at $p\in M\subset\R^4$ (or $\mathcal{K}_1$, $\mathcal{K}_2$, $\mathcal{K}_3$ 
at $p\in M\subset\R^5$) are the critical values of the restriction to the unit circle $\sph^1_p$ of $N_pM$ 
of the Gauss quadratic form $\mathcal{G}_p:\nbf\mapsto K^{\nbf}$ (respectively, to the unit sphere $\sph^2_p$ of $N_pM$). 
\end{remark*}

%%%%%%%%%%%%%%%%%%%%%%%%%%%%%%%%%%%%%%%%%%%%%%%%%%%%%%%%%%%%%%%%%%%%%%%%%%%%%%%%%%%%%%%%%%%%%%%%%%%%%%%%%%%%%%%%%%%
\subsection{\textbf{The Paired Quadratic Form $\mathcal{G}_p^*$ of the Gauss Quadratic Form $\mathcal{G}_p$}}
%%%%%%%%%%%%%%%%%%%%%%%%%%%%%%%%%%%%%%%%%%%%%%%%%%%%%%%%%%%%%%%%%%%%%%%%%%%%%%%%%%%%%%%%%%%%%%%%%%%%%%%%%%%%%%%%%%%
The Legendre transform of a quadratic form $\mathcal{G}$ on $\R^\ell$ is another quadratic form on $\R^\ell$ 
that we note $\mathcal{G}^*$ and that we call the \textit{paired quadratic form} of $\mathcal{G}$. 
Its matrix is the inverse of the matrix of $\mathcal{G}$, that is, $[\mathcal{G}^*]=[\mathcal{G}]^{-1}$, 
and hence its associated linear map is the inverse of the linear map $\flecha{\mathcal{G}}$ associated to $\mathcal{G}$, 
that is $\flecha{\mathcal{G}}^*=\flecha{\mathcal{G}}^{-1}$. Moreover $(\mathcal{G}^*)^*=\mathcal{G}$. 
The symbol $\mathcal{G}^*$ and the name ``paired'' for the Legendre transform are explained in \S \ref{section:Paired Quadratic Form}. 
%(see Proposition \ref{prop:G^v=G^*}). 

In sections \ref{section-C-S} and \ref{Section-relation-inequalities} we shall use the paired quadratic 
form $\mathcal{G}_p^*$ of the Gauss quadratic form $\mathcal{G}_p$, 
its matrix $[\mathcal{G}_p^*]=[\mathcal{G}_p]^{-1}$ and its associated linear map 
$\flecha{\mathcal{G}_p}^*=\flecha{\mathcal{G}_p}^{-1}$.

%%%%%%%%%%%%%%%%%%%%%%%%%%%%%%%%%%%%%%%%%%%%%%%%%%%%%%%%%%%%%%%%%%%%%%%%%%%%%%%%%%%%%%%%%%%%%%%%%%%%%%%%%%%%%%%%%%%
%%%%%%%%%%%%%%%%%%%%%%%%%%%%%%%%%%%%%%%%%%%%%%%%%%%%%%%%%%%%%%%%%%%%%%%%%%%%%%%%%%%%%%%%%%%%%%%%%%%%%%%%%%%%%%%%%%%
\section{Optical Caustic and Duality Between the Local Caustic and the Indicatrix Ellipse}\label{section-C-S}
%%%%%%%%%%%%%%%%%%%%%%%%%%%%%%%%%%%%%%%%%%%%%%%%%%%%%%%%%%%%%%%%%%%%%%%%%%%%%%%%%%%%%%%%%%%%%%%%%%%%%%%%%%%%%%%%%%% 
%%%%%%%%%%%%%%%%%%%%%%%%%%%%%%%%%%%%%%%%%%%%%%%%%%%%%%%%%%%%%%%%%%%%%%%%%%%%%%%%%%%%%%%%%%%%%%%%%%%%%%%%%%%%%%%%%%%

%%%%%%%%%%%%%%%%%%%%%%%%%%%%%%%%%%%%%%%%%%%%%%%%%%%%%%%%%%%%%%%%%%%%%%%%%%%%%%%%%%%%%%%
\subsection{\textbf{Focal Set and Optical Caustic}\label{sect-FocalSet-OpticalCaustic}}
%%%%%%%%%%%%%%%%%%%%%%%%%%%%%%%%%%%%%%%%%%%%%%%%%%%%%%%%%%%%%%%%%%%%%%%%%%%%%%%%%%%%%%%
\noindent
\textit{\textbf{\small Osculating Hyperspheres and Focal Centres}}. 
Let $\g:U\subset\R^2\to\R^{2+\ell}$ be a local parametrisation of a surface $M$, with $\g(0,0)=p$, 
and $S_{q,R}$ be the hypersphere, centred at $q$, given by the zeroes of the function $g^q(x)=\|x-q\|^2-R^2$. 
Suppose the hypersphere $S_{q,R}$ is tangent to $M$ at $p$. We say that $S_{q,R}$ is an 
\textit{osculating hypersphere} of $M$ at $p$ if the quadratic part of the function 
$g^q\comp\g:U\subset\R^2\to\R$ is degenerate, 
that is, if the Hessian matrix of $g^q\comp\g$ at $(0,0)$ has zero determinant. The centre $q$ of a  
hypersphere $S_{q,R}$, which is osculating to $M$ at $p$, is called a \textit{focal centre} of $M$ at $p$. 
\medskip

\noindent
\textit{\textbf{\small Focal Set}}. The \textit{focal set} of a surface $M$ in $\R^{2+\ell}$ is the set of all 
focal centres of $M$. 
\smallskip

\noindent
\textit{\textbf{\small Local Focal Set}}. The \textit{local focal set} at $p$, noted $\mathcal{F}_p$, 
is the set of focal centres of $M$ at $p$. 
\smallskip

Each focal centre $q$ of $M$ at $p$ is written as $q=p+\a\nbf$, where $\nbf$ is a 
unit normal vector and $|\a|$ is the radius of the osculating hypersphere of $M$ at $p$, centred at $q$. 
\medskip

We are also interested in the hyperspheres of infinite radius (hyperplanes), thus we shall consider the 
local focal set $\mathcal{F}_p$ as a subset of the projective completion $N_pM\sqcup\RP^{\ell-1}_\infty$ of $N_pM$.  
(It is $N_pM\sqcup\RP^1_\infty$ for $M$ in $\R^4$ and $N_pM\sqcup\RP^2_\infty$ for $M$ in $\R^5$.) 

%%%%%%%%%%%%%%%%%%%%%%%%%%%%%%%%%%%%%%%%%%%%%%%%%%%%%%%%%%%%%%%%%%%%%%%%%%%%%%%%%%%%%%%%%%%%%%%%%%%%%%%%%%%%%%%%%%%%
%%%%%%%%%%%%%%%%%%%%%%%%%%%%%%%%%%%%%%%%%%%%  PROPOSITION  %%%%%%%%%%%%%%%%%%%%%%%%%%%%%%%%%%%%%%%%%%%%%%%%%%%%%%%%%
%%%%%%%%%%%%%%%%%%%%%%%%%%%%%%%%%%%%%%%%%%%%%%%%%%%%%%%%%%%%%%%%%%%%%%%%%%%%%%%%%%%%%%%%%%%%%%%%%%%%%%%%%%%%%%%%%%%%
\begin{proposition}\label{prop:focal_set-equation}
The focal set of $M\subset\R^{2+\ell}$ is a union of degree two algebraic varieties. 
Namely, the local focal set of $M$ at $p$, $\mathcal{F}_p\subset N_pM\sqcup\RP^{\ell-1}_\infty$, is a degree two hypersurface. 
If $M$ is locally given in Monge form \eqref{eq:simple-Monge-form}, then, in affine coordinates, $\mathcal{F}_p$ is given by 
\begin{equation}\label{eq:focal_set-equation}
\det\begin{pmatrix}
 \langle \Abf\,, \, q\rangle -1 &  \langle \Bbf\,, \, q\rangle \\ 
  \langle \Bbf\,, \, q\rangle & \langle \Cbf\,, \, q\rangle -1
\end{pmatrix} =0\,. 
\end{equation} 
\end{proposition}

\noindent
\textit{Proof}.  
Clearly, the centre of a hypersphere tangent to $M$ at a point $p\in M$ belongs to the 
normal space to $M$ at $p$. Conversely, every point of the normal space $N_pM$ is the centre of one hypersphere tangent to $M$ 
at $p$. Thus if $M$ is locally given in Monge form \eqref{eq:simple-Monge-form}, a point of $\R^{2+\ell}$ is the centre 
of a hypersphere tangent to $M$ at $p$ (the origin) iff that point is of the form $(0,0;q_1,\ldots, q_\ell)$.  
The equation of such hypersphere $S_q$ is
\[x_1^2+x_2^2+(y_1-q_1)^2+\ldots+(y_{\ell}-q_{\ell})^2-(q_1^2+\ldots+q_{\ell}^2)=0\,, \quad \mbox{ that is,} \]
\[x_1^2+x_2^2+y_1^2+\ldots+y_\ell^2-2y_1q_1-\ldots-2y_{\ell}q_{\ell}=0\,.\]

So writing 
\[g^q(x;y)=x_1^2+x_2^2+y_1^2+\ldots+y_\ell^2-2y_1q_1-\ldots-2y_{\ell}q_{\ell}\,,\] 
the hypersphere $S_q$ is osculating to $M$ at $p$ (the origin) iff the 
quadratic part of the function $g^q\comp\g$, in the variables $s$, $t$, is degenerate. 
That quadratic part, 
\[s^2+t^2-(Q_1(s,t)q_1+\ldots+Q_\ell(s,t) q_\ell)\,, \mbox{ is equal to }\]
\[(1-a_1q_1-\ldots-a_\ell q_\ell)s^2-2(b_1q_1+\ldots+b_\ell q_\ell)st+(1-c_1q_1-\ldots-c_\ell q_\ell)t^2\,,\]
that is, 
\begin{equation}\label{eq:focal-set}
(1-\langle \Abf\,,\,q\rangle)s^2-2\langle \Bbf\,,\,q\rangle st+(1-\langle \Cbf\,,\,q\rangle)t^2\,,
\end{equation}
which is degenerate if and only if 
\[\det\begin{pmatrix}
 \langle \Abf\,, \, q\rangle -1 &  \langle \Bbf\,, \, q\rangle \\ 
  \langle \Bbf\,, \, q\rangle & \langle \Cbf\,, \, q\rangle -1
\end{pmatrix} =0\,. \eqno{\square}\]
\medskip

% Proposition \ref{prop:focal_set-equation} has the following generalisation (whose proof is similar \protect\cite{Uribe_k-dim_submnfd-Rn}): 
% \begin{theorem}
%  The focal set of a $k$-dimensional submanifold $M\subset\R^{k+\ell}$ is a union of 
%  degree $k$ algebraic varieties.  
% \end{theorem}
%
\noindent
\textit{\textbf{\small Caustic of a Family of Functions}}. 
The \textit{caustic of a family of functions} depending smoothly on parameters 
consists of the parameter values for which the corresponding function
has a \textit{non-Morse critical point} (at which the Hessian determinant vanishes). 
\medskip

\noindent
\textit{\textbf{\small Optical Caustic or Evolute}}. 
Given a (local) parametrisation $\g:\R^k\to \R^n$ of a $k$-submanifold $M$,  
for each $q\in\R^n$ we take its ``squared distance function'' to $\g(x)$
\begin{equation}\label{eq:squared-distance}
  F^q(x)=\frac{1}{2}\left\langle\, q-\g(x)\,,\,q-\g(x) \,\right\rangle\,.
\end{equation}  
We get a family of functions $\{F^q:\R^k\to\R\}$ parametrised by the points $q$ of $\R^n$. The caustic of this family
is called the \textit{optical caustic} (or \textit{evolute}) \textit{of the submanifold} $M$.  

%%%%%%%%%%%%%%%%%%%%%%%%%%%%%%%%%%%%%%%%%%%%%%%%%%%%%%%%%%%%%%%%%%%%%%%%%%%%%%%%%%%%%%%%%%%%%%%%%%%%%%%%%%%%%%%%%%%%
%%%%%%%%%%%%%%%%%%%%%%%%%%%%%%%%%%%%%%%%%%%%  PROPOSITION  %%%%%%%%%%%%%%%%%%%%%%%%%%%%%%%%%%%%%%%%%%%%%%%%%%%%%%%%%
%%%%%%%%%%%%%%%%%%%%%%%%%%%%%%%%%%%%%%%%%%%%%%%%%%%%%%%%%%%%%%%%%%%%%%%%%%%%%%%%%%%%%%%%%%%%%%%%%%%%%%%%%%%%%%%%%%%%
\begin{proposition}\label{prop:caustic=focal-set}
The optical caustic $\mathcal{C}$ and the focal set $\mathcal{F}$ of a surface $M$ in $\R^{2+\ell}$ coincide. 
In particular, the local caustic $\mathcal{C}_p:=\mathcal{C}\cap N_pM$ is given by the equation 
\begin{equation}\label{eq:focal-quadric}
  (\,\langle \Abf\,, \,q \rangle-1\,)(\,\langle \Cbf\,, \,q \rangle-1\,)-\langle \Bbf\,, \,q \rangle^2=0\,, \quad q\in N_pM. 
\end{equation}
\end{proposition}

\begin{proof}  
Given a point $q\in\R^n$, the function $F^q$ has a critical point at $x\in\R^k$ if its partial 
derivatives vanish: 
\[\langle\,\g_{x_i}(x)\,,\,q-\g(x)\,\rangle=0, \ \ (i=1,\ldots,k)\,. \]
It means that the vector $q-\g(x)$ is orthogonal to $\g$ at $p=\g(x)$. Given $p=\g(x)$, 
the normal space $N_pM$ consists of the points $q\in\R^n$ such that $F^q$ has a critical point at $x$. 

Suppose a surface $M$ is locally given in Monge form \eqref{eq:simple-Monge-form} at a point $p\in M$, 
that is, $\g(s,t)=s\ebf_1+t\ebf_2+\frac{1}{2}(s^2\Abf+2st\Bbf+t^2\Cbf)+{\rm h.o.t.}(s,t)$. 
We get the equalities 
\begin{equation}\label{eq:trivial-equalities}
\g_s(0,0)=\ebf_1; \quad \g_t(0,0)=\ebf_2; \quad \g_{ss}(0,0)=\Abf; \quad \g_{st}(0,0)=\Bbf; \quad \g_{tt}(0,0)=\Cbf.
\end{equation} 
Using \eqref{eq:trivial-equalities} together with the following three relations 
\[-F^q_{ss}=\langle \g_{ss},\,q-\g\rangle-\langle \g_s,\g_s\rangle; \quad
-F^q_{st}=\langle \g_{st},q-\g\rangle-\langle \g_s,\g_t\rangle; \quad
-F^q_{tt}=\langle \g_{tt},q-\g\rangle-\langle \g_t,\g_t\rangle\,, 
\]
we get that the Hessian matrix of $F^q$ at $(s,t)=(0,0)$ is degenerate, that is  
\[\det\begin{pmatrix}
 F^q_{ss} & F^q_{st} \\ 
 F^q_{st} & F^q_{tt}
\end{pmatrix} =0\,, 
\qquad 
\mbox{iff } 
\qquad 
\det\begin{pmatrix}
 \langle \Abf\,, \, q\rangle -1 &  \langle \Bbf\,, \, q\rangle \\ 
  \langle \Bbf\,, \, q\rangle & \langle \Cbf\,, \, q\rangle -1
\end{pmatrix} =0\,.\]
It coincides with equation \eqref{eq:focal_set-equation} of $\mathcal{F}_p$ in Proposition\,\ref{prop:focal_set-equation}.
\end{proof}

Proposition \ref{prop:focal_set-equation} has a natural extension to $k$-dimensional submanifolds in $\R^{k+\ell}$ : 

%%%%%%%%%%%%%%%%%%%%%%%%%%%%%%%%%%%%%%%%%%%%%%%%%%%%%%%%%%%%%%%%%%%%%%%%%%%%%%%%%%%%%%%%%%%%%%%%%%%%%%%%%%%%%%%%%%%%
%%%%%%%%%%%%%%%%%%%%%%%%%%%%%%%%%%%%%%%%%%%%%   THEOREM   %%%%%%%%%%%%%%%%%%%%%%%%%%%%%%%%%%%%%%%%%%%%%%%%%%%%%%%%%%
%%%%%%%%%%%%%%%%%%%%%%%%%%%%%%%%%%%%%%%%%%%%%%%%%%%%%%%%%%%%%%%%%%%%%%%%%%%%%%%%%%%%%%%%%%%%%%%%%%%%%%%%%%%%%%%%%%%%
\begin{theorem}[\protect\cite{Uribe_k-dim_submnfd-Rn}]\label{th:local-caustic-algebraic}
  The optical caustic $\mathcal{C}$ of a $k$-dimensional smooth submanifold $M$ in $\R^{k+\ell}$ is a union of 
   degree $k$ algebraic varieties. Namely, for each point $p\in M$, the local caustic $\mathcal{C}_p:=\mathcal{C}\cap N_pM$ 
  is a degree $k$ hypersurface in $N_pM$.
\end{theorem}
\begin{proof}
One uses the fact that the Hessian matrix of $F^q$ is of size $k\times k$ and that its entries are 
affine functions of $q$. 
\end{proof}

Although the optical caustic $\mathcal{C}$ of a smooth submanifold $M^k$ in $\R^{k+\ell}$ is 
the union of %a $2$-parameter family of $(\ell-1)$-dimensional 
the degree $k$ subvarieties $\mathcal{C}_p$, as a rule $\mathcal{C}$ is not an algebraic hypersurface. 
\medskip

%%%%%%%%%%%%%%%%%%%%%%%%%%%%%%%%%%%%%%%%%%%%%%%%%%%%%%%%%%%%%%%%%%%%%%%%%%%%%%%%%%%%%%%%%%%%%%%%%%%%%%%%%%%%%%%%%%%
%%%%%%%%%%%%%%%%%%%%%%%%%%%%%%%%%%%%%%%%%%%%  DEFINITION  %%%%%%%%%%%%%%%%%%%%%%%%%%%%%%%%%%%%%%%%%%%%%%%%%%%%%%%%%
\noindent
\textbf{\textit{\small Sphere-Contact Direction}}.
%According to Theorem\,\ref{th:focal_quadric-indicatrix}-(\textit{\textbf{b}}, \textit{\textbf{c}}), 
%given $\nbf\in N_pM$ with $\|\nbf\|=1$, the normal line $\{\a\nbf:\l\in\R\sqcup\infty\}$ cuts the local 
%caustic $\mathcal{C}_p$ at two points, counting multiplicities. 
If $q=p+\a\nbf\in\mathcal{C}_p$, then $1/\a=:\mu$ is an eigenvalue of $\left[d^\perp_{p,\nbf}\G\right]$ 
whose associated eigenspace is the kernel of $[\mathrm{Hess} F^q]$ at $p$ for $F^q$ given in \eqref{eq:squared-distance}. 
The higher contact of $M$ with the osculating hypersphere $S_{q,\a}$ at $p$ takes place along this eigendirection, 
which we call a \textit{sphere-contact direction of $M$ at $p$ for the focal centre} $q$. 
\smallskip

\noindent
\textbf{\textit{\small Normal Map and Optical Caustic}}. The metric $\langle \cdot\,,\,\cdot\rangle_p$ on a Riemannian manifold $V$ 
associates to each tangent vector $v\in T_pV$ the covector $\langle v\,,\,\cdot\rangle_p\in T^*_pV$. 

The \textit{normal map} of a submanifold $M$ in $\R^n$, $N^*M\to\R^n$ sends each ``normal covector''  
$\langle v\,,\,\cdot\rangle_p$ (where $v\in N_pM$) to ``the end-point of the normal vector $v$''\,:  
\[
\langle v\,,\,\cdot\rangle_p \mapsto p+v\,.
\] 
It is a (so called) Lagrangian map. 
%The set of normal covectors $L^n=\{\langle v\,,\,\cdot\rangle: p\in\G, v\in N_p\G\}$ is a Lagrange submanifold of 
%$T^*\R^n$ and the natural projection $L\subset T^*\R^n\to\langle v\,,\,\cdot\rangle\R^n$ is a Lagrangian map 
%(all its fibres are Lagrenge submanifolds of $T^*\R^n$). 
The \textit{caustic of a Lagrangian map} is the set of its critical values. 
In the general theory of Lagrangian singularities and caustics of Lagrangian maps, initiated by V.I. Arnold (cf. \protect\cite{avg}), 
\textit{the caustic of the normal map of a submanifold $M$ in $\R^n$ coincides with the optical caustic of $M$}. 
In fact, \textit{the family \eqref{eq:squared-distance} of squared distance functions of $M$ is a generating 
family of the normal map of $M$}.

%%%%%%%%%%%%%%%%%%%%%%%%%%%%%%%%%%%%%%%%%%%%%%%%%%%%%%%%%%%%%%%%%%%%%%%%%%%%%%%%%%%%%%%%%%%%%%%%%%%%%%%%%
%%%%%%%%%%%%%%%%%%%%%%%%%%%%%%%%%%%%%%%%%%%%%%%%%%%%%%%%%%%%%%%%%%%%%%%%%%%%%%%%%%%%%%%%%%%%%%%%%%%%%%%%%
\subsection{\textbf{The Polar Dual of a subvariety}\label{sect-polar_dual-subvariety}}
%%%%%%%%%%%%%%%%%%%%%%%%%%%%%%%%%%%%%%%%%%%%%%%%%%%%%%%%%%%%%%%%%%%%%%%%%%%%%%%%%%%%%%%%%%%%%%%%%%%%%%%%%
%%%%%%%%%%%%%%%%%%%%%%%%%%%%%%%%%%%%%%%%%%%%%%%%%%%%%%%%%%%%%%%%%%%%%%%%%%%%%%%%%%%%%%%%%%%%%%%%%%%%%%%%%
We shall use the \textit{polar duality} in Euclidean space $\R^\ell$ 
with respect to the unit hypersphere of $\R^\ell$, \,$\sph^{\ell-1}:=\{x\in\R^\ell:\langle x,x\rangle=1\}$. 
\medskip
%, where an orthonormal base is chosen. 
%The unit sphere is given by the equation $x_1^2+\ldots+x_n^2=1$. 

%%%%%%%%%%%%%%%%%%%%%%%%%%%%%%%%%%%%%%%%%%%%%%%%%%%%%%%%%%%%%%%%%%%%%%%%%%%%%%%%%%%%%%%%%%%%%%%%%%%%%%%%%%%%%%%%%%%
%%%%%%%%%%%%%%%%%%%%%%%%%%%%%%%%%%%%%%%%%%%%  DEFINITION  %%%%%%%%%%%%%%%%%%%%%%%%%%%%%%%%%%%%%%%%%%%%%%%%%%%%%%%%%
\noindent
\textbf{\textit{\small Poles and Polars}}. 
Notice that the equation of a hyperplane in $\R^\ell$ can be written as 
$\langle a, x \rangle=1$ for some $a=(a_1,\ldots,a_\ell)\neq(0,\ldots,0)$ iff that hyperplane does not 
contain the origin. 
Therefore, there is a bijection between the points $a\in\R^\ell$, distinct from the origin, 
and the hyperplanes in $\R^\ell$ not containing the origin: 
\begin{equation}\label{eq:pole-polar}
 \mbox{point } a\in\R^\ell\smallsetminus\{0\}\quad \longleftrightarrow \quad \mbox{hyperplane $H_a$ of equation } \langle a,x\rangle=1\,. 
\end{equation}

The hyperplane $H_a$ is called the \textit{polar} of the point $a\in\R^\ell$ and the point $a$ is the \textit{pole} 
of the hyperplane $H_a$. This correspondence is called \textit{polar duality} with respect to $\sph^{\ell-1}$. 
Thus the ``slopes'' of the hyperplane $H_a$ are the coordinates of its pole $a$. 

%%%%%%%%%%%%%%%
\begin{example}
In Euclidean space $\R^4$, the polar of the point $a=(2,5,4,3)$ is the hyperplane given by the equation 
$2x_1+5x_2+4x_3+3x_4=1$. 
\end{example}
%%%%%%%%%%%%%

In order to take into account both, the hyperplanes through the origin and the origin itself, 
this correspondence pole\,$\leftrightarrow$\,polar is extended to the projective completion of $\R^\ell$, 
$\RP^\ell=\R^\ell\sqcup\RP_\infty^{\ell-1}$. The hyperplane at infinity $\RP_\infty^{\ell-1}$ is the polar of the centre 
of $\sph^{\ell-1}$ and the points of $\RP_\infty^{\ell-1}$ are the poles 
of the hyperplanes through the centre of $\sph^{\ell-1}$.  %the hyperplane at infinity. 
\begin{remark}\label{rem:orthogonality-polar}
By relation \eqref{eq:pole-polar}, for any $q\in\R^\ell\sqcup\RP_\infty^{\ell-1}$ the line joining the origin and $q$ 
is orthogonal to the polar hyperplane of $q$. In particular, a hyperplane $H_q$ in $\R^\ell\sqcup\RP_\infty^{\ell-1}$ 
contains the origin of $\R^\ell$ iff its pole $q$ is at infinity in the direction orthogonal to $H_q$.  
\end{remark}

\noindent
\textit{\textbf{{\small Polar Duality}}}.
This one-to-one relation between hyperplanes of $\RP^\ell$ and points of the same 
space, by means of a quadric $\mathcal{S}$ (here $\sph^{\ell-1}$), is an isomorphism called {\em polar duality}
\[\Psi_{\mathcal{S}} : (\RP^\ell)^\vee \stackrel{\simeq}{\longrightarrow}\RP^\ell\,.\]

%%%%%%%%%%%%%%%%%%%%%%%%%%%%%%%%%%%%%%%%%%%%%%%%%%%%%%%%%%%%%%%%%%%%%%%%%%%%%%%%%%%%%%%%%%%%%%%%%%%%%%%%%%%%%%%%%%%
%%%%%%%%%%%%%%%%%%%%%%%%%%%%%%%%%%%%%%%%%%%%  DEFINITION  %%%%%%%%%%%%%%%%%%%%%%%%%%%%%%%%%%%%%%%%%%%%%%%%%%%%%%%%%
\noindent
\textbf{\textit{\small Polar Dual Variety}}. The \textit{polar dual} $M^\vee$ of a subvariety $M$ in $\R^\ell$ 
(with respect to the unit sphere $\sph^{\ell-1}$ of $\R^\ell$) is the subvariety of $\R^\ell\sqcup\RP^{\ell-1}_\infty$ 
which consists of the poles of all hyperplanes tangent to $M$. 

%%%%%%%%%%%%%%%%%%%%%%%%%%%%%%%%%%%%%%%%%%%%%%%%%%%%%%%%%%%%%%%%%%%%%%%%%%%%%%%%%%%%%%%%%%%%%%%%%%%%%%%%%%%%%%%%%%%%
\subsection{\textbf{Duality Between the Local Caustic $\mathcal{C}_p$ and the Indicatrix Ellipse $\E_p$}}
%%%%%%%%%%%%%%%%%%%%%%%%%%%%%%%%%%%%%%%%%%%%%%%%%%%%%%%%%%%%%%%%%%%%%%%%%%%%%%%%%%%%%%%%%%%%%%%%%%%%%%%%%%%%%%%%%%%%
For each point $p$ of a surface $M$ in $\R^{2+\ell\geq 3}$ we consider the polar duality in the normal space 
$N_pM$ (of dimension $\ell$) with respect to its unit hypersphere $\sph^{\ell-1}\subset N_pM$. 

%%%%%%%%%%%%%%%%%%%%%%%%%%%%%%%%%%%%%%%%%%%%%%%%%%%%%%%%%%%%%%%%%%%%%%%%%%%%%%%%%%%%%%%%%%%%%%%%%%%%%%%%%%%%%%%%%%%%
%%%%%%%%%%%%%%%%%%%%%%%%%%%%%%%%%%%   MAIN DUALITY THEOREM   %%%%%%%%%%%%%%%%%%%%%%%%%%%%%%%%%%%%%%%%%%%%%%%%%%%%%%%
%%%%%%%%%%%%%%%%%%%%%%%%%%%%%%%%%%%%%%%%%%%%%%%%%%%%%%%%%%%%%%%%%%%%%%%%%%%%%%%%%%%%%%%%%%%%%%%%%%%%%%%%%%%%%%%%%%%%
\begin{theorem}[\textbf{\small Duality Theorem}]\label{th:focal_quadric-indicatrix}
Let $p$ be a point of a smooth surface $M$ in $\R^{2+\ell}$. 
\smallskip

\noindent
\textit{\textbf{a}}$)$ The polar dual of the indicatrix ellipse $\mathcal{E}_p$ of $M$ at $p$ is 
the local caustic $\mathcal{C}_p$ in the projective completion of the normal space 
$N_pM\sqcup\RP^{\ell-1}_\infty$, that is $\mathcal{E}_p^\vee=\mathcal{C}_p$.
\smallskip

\noindent
\textit{\textbf{b}}$)$ Every line normal to $M$ at $p$ cuts the local caustic $\mathcal{C}_p\subset N_pM\sqcup\RP^{\ell-1}_\infty$ at 
two real points {\em (focal centres)}, counting multiplicities. 
\smallskip

\noindent
\textit{\textbf{c}}$)$ For each unit vector $\nbf\in N_pM$, the two focal centres at which the line generated by $\nbf$ meets 
$\mathcal{C}_p$ are $\nbf/\mu_1$ and $\nbf/\mu_2$, where $\mu_1$, $\mu_2$ are the eigenvalues of the matrix 
$[d^\perp_{p,\nbf}\Gamma]$, and the corresponding sphere-contact directions are two orthogonal lines on $T_pM$. \\ 
{\rm (So to each normal line there correspond two curvature radii.)}
\end{theorem}

\noindent
\underline{\textit{Proof of} \textit{\textbf{a}}}.
Consider the parametrisation \eqref{eq:parametrisation-E_p} of the indicatrix ellipse $\E_p:\RP^1\to N_pM$, 
\[ 
\E_p(\theta)=\frac{1}{2}(\Abf+\Cbf)+\frac{1}{2}(\Abf-\Cbf)\cos 2\theta+\Bbf\sin 2\theta\,.
\]

Any hyperplane tangent to $\E_p$ at $\E_p(\theta)$ contains the points $\E_p(\theta)$ and $\E_p(\theta)+\frac{1}{2}\E_p'(\theta)$. 
Writing the equation of such hyperplane in the form $\langle X, q \rangle=1$, 
where $q$ is its pole, we get, from these two points, the following two equations (on the pole $q$):
\begin{equation}\label{prod-e}
  \bigg\langle \frac{1}{2}(\Abf+\Cbf)+\frac{1}{2}(\Abf-\Cbf)\cos 2\theta+B\sin 2\theta\,, \,q \bigg\rangle=1\, ,
\end{equation}

\begin{equation}\label{prod-e'}
  \bigg\langle \frac{1}{2}(\Abf+\Cbf)+\frac{1}{2}(\Abf-\Cbf)(\cos 2\theta-\sin 2\theta)+\Bbf(\sin 2\theta+\cos 2\theta)\,, \,q \bigg\rangle=1\, .
\end{equation}

We take equation\,\eqref{prod-e} minus equation\,\eqref{prod-e'} to get 
\begin{equation}\label{e-e'}
  \sin 2\theta\bigg\langle \frac{1}{2}(\Abf-\Cbf)\,, \,q\bigg\rangle-\cos 2\theta\bigg\langle \Bbf\,, \,q \bigg\rangle=0
\end{equation}
and we write equation\,\eqref{prod-e} in the convenient form 
\begin{equation}\label{prod-e-bis}
  \cos 2\theta\bigg\langle \frac{1}{2}(\Abf-\Cbf)\,, \,q\bigg\rangle+
  \sin 2\theta\bigg\langle \Bbf\,, \,q \bigg\rangle=1-\bigg\langle\frac{1}{2}(\Abf+\Cbf)\,, \,q \bigg\rangle\, . 
\end{equation}
  Summing up the squares of equations\,\eqref{e-e'} and \eqref{prod-e-bis} we get 
\[
    \bigg\langle \frac{1}{2}(\Abf-\Cbf)\,, \,q\bigg\rangle^2+\bigg\langle \Bbf\,, \,q \bigg\rangle^2
    =1-2\bigg\langle\frac{1}{2}(\Abf+\Cbf)\,, \,q \bigg\rangle+\bigg\langle\frac{1}{2}(\Abf+\Cbf)\,, \,q \bigg\rangle^2\,, 
\]
  which is equivalent to the equation of the local caustic $\mathcal{C}_p$\,:
\[
  (\,\langle \Abf\,, \,q \rangle-1\,)(\,\langle \Cbf\,, \,q \rangle-1\,)-\langle \Bbf\,, \,q \rangle^2=0\,. \eqno{\square}
\]

\begin{proof}[\underline{Proof of \textit{\textbf{b}} and \textit{\textbf{c}}}]
Consider the equation of the local caustic $\mathcal{C}_p$ in $N_pM$,
\[
\det
\begin{pmatrix}
 \langle \Abf\,, \, q\rangle -1 &  \langle \Bbf\,, \, q\rangle \\ 
  \langle \Bbf\,, \, q\rangle & \langle \Cbf\,, \, q\rangle -1
\end{pmatrix} =0, \qquad q\in N_pM\,. 
\]
Take a unit normal vector $\nbf\in N_pM$ and the line it generates 
$\ell^{\nbf}=\{\a\nbf:\a\in\R\}$. The points of intersection of this line with 
the local caustic $\mathcal{C}_p$ are given by the values of $\a$ satisfying 
\[
\det
\begin{pmatrix}
 \langle \Abf\,, \, \a\nbf \rangle -1 &  \langle \Bbf\,, \, \a\nbf\rangle \\ 
  \langle \Bbf\,, \, \a\nbf\rangle & \langle \Cbf\,, \, \a\nbf\rangle -1
\end{pmatrix} =0\,. 
\]
Writing $\mu=1/\a$ ($\a\neq 0$) and dividing the above equation by $\a^2$ we get 
\begin{equation}\label{equation-eigenvalues}
\det
\begin{pmatrix}
 \langle A\,, \, \nbf \rangle -\mu &  \langle B\,, \, \nbf\rangle \\ 
  \langle B\,, \, \nbf\rangle & \langle C\,, \, \nbf\rangle -\mu
\end{pmatrix} =0\,. 
\end{equation}

The numbers $\mu$ that satisfy \eqref{equation-eigenvalues} are the two real eigenvalues of the symmetric matrix 
{\small $
\begin{pmatrix}
 \langle A\,, \, \nbf \rangle &  \langle B\,, \, \nbf\rangle \\ 
  \langle B\,, \, \nbf\rangle & \langle C\,, \, \nbf\rangle 
\end{pmatrix}$}$\,= \left[d_{p,\nbf}^\perp\G\right]$. 
Thus the line $\ell^{\nbf}$ intersects $\mathcal{C}_p$ at two real points, 
counting multiplicities. Distinct eigenvalues correspond to orthogonal eigendirections. 
\end{proof}

\noindent
\textbf{\small Duality Theorem Generalisation}. An extension of Theorem \ref{th:focal_quadric-indicatrix} for $k$-dimensional 
submanifolds $M$ in $\R^{k+\ell}$, in which the indicatrix ellipse $\E_p$ is replaced with the critical value set of the 
local quadratic map $\f_p:T_pM\to N_pM$, will appear in \protect\cite{Uribe_k-dim_submnfd-Rn}.

%%%%%%%%%%%%%%%%%%%%%%%%%%%%%%%%%%%%%%%%%%%%%%%%%%%%%%%%%%%%%%%%%%%%%%%%%%%%%%%%%%%%%%%%%%%%%%%%%%%%%%%%%%%%%%%%%%%%%%%%%
\subsection{\textbf{Classification of Points of a Surface in $\R^4$ ($\R^5$) by its Local Caustic}}\label{sect-Class-Points-R5}
%%%%%%%%%%%%%%%%%%%%%%%%%%%%%%%%%%%%%%%%%%%%%%%%%%%%%%%%%%%%%%%%%%%%%%%%%%%%%%%%%%%%%%%%%%%%%%%%%%%%%%%%%%%%%%%%%%%%%%%%%
The theorems of this subsection follow naturally from our Duality Theorem \ref{th:focal_quadric-indicatrix}. 
\smallskip

%%%%%%%%%%%%%%%%%%%%%%%%%%%%%%%%%%%%%%%%%%%%%%%%%%%%%%%%%%%%%%%%%%%%%%%%%%%%%%%%%%%%%%%%%%%%%%%%%%%%%%%%%%%%%%%%%%%
%%%%%%%%%%%%%%%%%%%%%%%%%%%%%%%%%%%%%%%%%%%%  DEFINITION  %%%%%%%%%%%%%%%%%%%%%%%%%%%%%%%%%%%%%%%%%%%%%%%%%%%%%%%%%
\noindent
\textbf{\textit{\small Umbilical Focal Centre}}.  %and Umbilical Osculating Hypersphere}}. 
Let $p$ be a point 
of $M$ in $\R^n$.  A focal centre $q$ of $M$ at $p$ is said to be \textit{umbilical} iff the line determined by  
$p$ and $q$ intersects $\mathcal{C}_p$ at $q$ with multiplicity $2$, that is iff $1/\|q-p\|$ is a double eigenvalue of 
the matrix $\left[d^\perp_{p,\nbf}\Gamma\right]=\,${\small $\begin{pmatrix}
 \langle A\,, \, \nbf \rangle &  \langle B\,, \, \nbf\rangle \\ 
  \langle B\,, \, \nbf\rangle & \langle C\,, \, \nbf\rangle 
\end{pmatrix}$, where $\nbf=\frac{q-p}{\|q-p\|}$} or (as it is defined in \protect\cite{Costa_Moraes_R-F}) iff 
the squared distance function $F^q$, given in \eqref{eq:squared-distance}, has a corank $2$ singularity at $p$.

%%%%%%%%%%%%%%%%%%%%%%%%%%%%%%%%%%%%%%%%%%%%%%%%%%%%%%%%%%%%%%%%%%%%%%%%%%%%%%%%%%%%%%%%%%%%%%%%%%%%%%%%%%%%%%%%%%%%
%%%%%%%%%%%%%%%%%%%%%%%%%%%%%%%%%%%%%%%%%%%%    THEOREM    %%%%%%%%%%%%%%%%%%%%%%%%%%%%%%%%%%%%%%%%%%%%%%%%%%%%%%%%%
%%%%%%%%%%%%%%%%%%%%%%%%%%%%%%%%%%%%%%%%%%%%%%%%%%%%%%%%%%%%%%%%%%%%%%%%%%%%%%%%%%%%%%%%%%%%%%%%%%%%%%%%%%%%%%%%%%%%
\begin{theorem}\label{th:local-caustic-R4}
Let $M$ be a smooth surface in $\R^4$ and $\E_p$ be its indicatrix ellipse at $p$. 
\smallskip

\noindent
\textit{\textbf{A}}. 
The point $p$ {\rm (the origin in $N_pM$)} lies on a convex component of $N_pM\smallsetminus \mathcal{C}_p$;
\smallskip

\noindent
\textit{\textbf{B}}. If the ellipse $\E_p$ is non degenerate, then the local caustic $\mathcal{C}_p$ is a non degenerate conic. 
In particular {\rm (see Fig.\,\ref{fig:local-causticA}):}
\smallskip

\noindent
$(e)$ $p$ is elliptic iff $\mathcal{C}_p$ is an ellipse; 
\smallskip

\noindent
$(\mathring{e})$ $p$ elliptic with $\E_p$ centred at $p$ {\rm (i.e. $\|\Hbf\|=0$)} iff $p$ is the 
centre of the ellipse $\mathcal{C}_p$; %{\rm (noted $\|\Rbf\|$)}; 
\smallskip

\noindent
$(h)$ p is hyperbolic iff $\mathcal{C}_p$ is a hyperbola; 
\smallskip

\noindent
$(p)$ $p$ is parabolic iff $\mathcal{C}_p$ is a parabola; 
\smallskip

\noindent
\textit{\textbf{C}}. If $\E_p$ is a segment, then the local caustic $\mathcal{C}_p$ is the union 
of the polar lines of the end-points of $\E_p$. These lines intersect at the pole of 
the line that contains $\E_p$, which is the unique umbilical focal centre {\rm (may be at infinity)}. 
In particular {\rm (see Fig.\,\ref{fig:local-causticB}):}
\smallskip

\noindent
$(su)$ $p$ is (hyperbolic) semi-umbilic iff the intersection point of the lines forming $\mathcal{C}_p$ 
{\rm (the unique umbilical focal centre)} is not at infinity - those lines are not parallel in $N_pM$; 
\smallskip

\noindent
$(ri)$ $p$ is real inflection iff the lines of $\mathcal{C}_p$ are orthogonal to $\E_p$, and $p$ is between them; 
\smallskip

\noindent
$(\mathring{ri})$ $p$ is real inflection and is the centre of $\E_p$ iff $p$ is equidistant to the lines 
of $\mathcal{C}_p$; 
\smallskip

\noindent
$(fi)$ $p$ is flat inflection iff one of the lines forming $\mathcal{C}_p$ is the line at infinity; 
\smallskip

\noindent
$(ii)$ $p$ is imaginary inflection iff $p$ is not between the lines of $\mathcal{C}_p$ $($both orthogonal to $\E_p$$)$;
\medskip

\noindent
\textit{\textbf{D}}. If $\E_p$ reduces to a point $q$, then $\mathcal{C}_p$ is the polar $\ell_q$ of $q$ counted twice. In particular\,$:$ 
\smallskip 

\noindent
$(u)$ $p$ is umbilic ($q\neq p$) iff the double line $\mathcal{C}_p=2\times\ell_q$ is not the line 
at infinity {\rm (Fig.\,\ref{fig:local-causticB})}. 
\smallskip 

\noindent
$(fu)$ $p$ is flat umbilic, i.e. $q=p$ iff \, $\mathcal{C}_p=2\times\RP^1_\infty$\, is the line at infinity counted twice. 
\end{theorem}  
\begin{figure}[ht] %< préférence de placement h , t , b or p >
\centering
\includegraphics[scale=0.15]{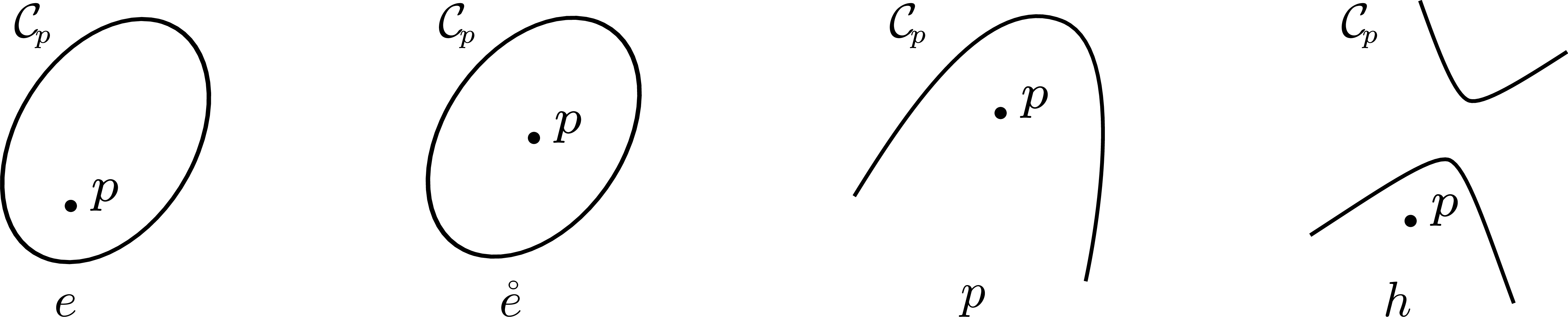}
\caption{\small The local caustic $\mathcal{C}_p$ when the indicatrix ellipse $\E_p$ is non degenerate.}
\label{fig:local-causticA}
\end{figure}
\begin{figure}[ht] %< préférence de placement h , t , b or p >
\centering
\includegraphics[scale=0.15]{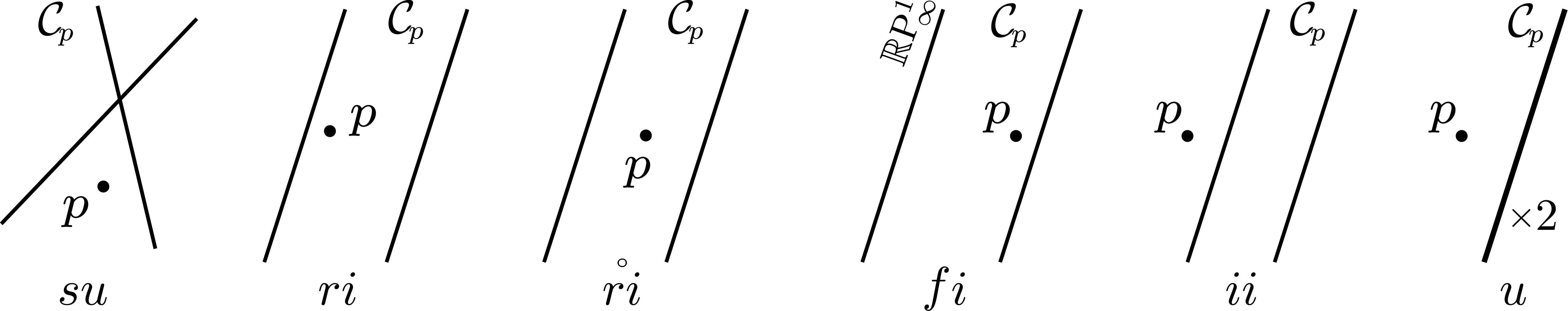}
\caption{\small Local caustic $\mathcal{C}_p$: ($su$) to ($ii$) if $\E_p$ is a segment - ($u$) if $\E_p$ is a point $q\neq p$.}
\label{fig:local-causticB}
\end{figure}

\begin{remark*}
For surfaces in $\R^4$ Theorem\,\ref{th:focal_quadric-indicatrix}-\textit{\textbf{a}} implies that the local caustic coincides with 
the ``characteristic curve'' defined in \protect\cite{Kommerell} (see Kommerell Theorem in \protect\cite{Kommerell,Goncalves}). 
\end{remark*} 

%%%%%%%%%%%%%%%%%%%%%%%%%%%%%%%%%%%%%%%%%%%%%%%%%%%%%%%%%%%%%%%%%%%%%%%%%%%%%%%%%%%%%%%%%%%%%%%%%%%%%%%%%%%%%%%%%%%%%%

For $M$ in $\R^5$ Th.\,4.2 in \protect\cite{Costa_Moraes_R-F} stated the basic properties of the local caustic $\mathcal{C}_p$. 
Our reformulation of that theorem (Theorems \ref{th:C_p=cone} and \ref{prop:caustic-cone}) adds new results, 
shows its duality nature and gives the explicit relation between the points of $\E_p$ and the lines in $\mathcal{C}_p$\,: 

%%%%%%%%%%%%%%%%%%%%%%%%%%%%%%%%%%%%%%%%%%%%%%%%%%%%%%%%%%%%%%%%%%%%%%%%%%%%%%%%%%%%%%%%%%%%%%%%%%%%%%%%%%%%%%%%%%%%%%
%%%%%%%%%%%%%%%%%%%%%%%%%%%%%%%%%%%%%%%%%%%%%    THEOREM    %%%%%%%%%%%%%%%%%%%%%%%%%%%%%%%%%%%%%%%%%%%%%%%%%%%%%%%%%%
%%%%%%%%%%%%%%%%%%%%%%%%%%%%%%%%%%%%%%%%%%%%%%%%%%%%%%%%%%%%%%%%%%%%%%%%%%%%%%%%%%%%%%%%%%%%%%%%%%%%%%%%%%%%%%%%%%%%%%
\begin{theorem}\label{th:C_p=cone}
Given a point $p$ of $M$ in $\R^5$ whose ellipse $\E_p$ is non degenerate, 
the plane that contains $\E_p$ is noted $\Pi_\E$. 
The following holds {\rm (see Fig.\,\ref{fig:Cp-lines-Ep-points}):}
\smallskip

\noindent
$a)$ The poles of the planes tangent to $\E_p$ at $\E_p(\theta)$ form a line 
in $\mathcal{C}_p$. We note it $\l_\theta^\vee$. The sphere-contact direction at $p$  for all points $q$ in $\l_\theta^\vee$ is the line $\ell_\theta$ directed by $\pm\ubf_\theta=\pm(\cos\theta\ebf_1+\sin\theta\ebf_2)$; 
\smallskip

\noindent
$b)$ The local caustic $\mathcal{C}_p$ is a non degenerate quadratic cone, in $N_pM\sqcup\RP^2_\infty$. 
Its vertex $\Rbf$ {\rm (may be at infinity)} is the pole of the plane $\Pi_\E$ and is also the 
unique umbilical focal centre of $M$ at $p$;
\smallskip

\noindent
$c)$ The point $p$ {\rm (the origin in $N_pM$)} lies on a convex component of $N_pM\smallsetminus \mathcal{C}_p$;
\medskip

\noindent
Write $\Pi_{\theta,\theta+\frac{\pi}{2}}$ for the plane that contains the lines $\l_\theta^\vee$ and 
$\l_{\theta+\frac{\pi}{2}}^\vee$ of $\mathcal{C}_p$ {\rm (formed by the poles of the planes tangent to $\E_p$ 
at its opposite points $\E_p(\theta)$ and $\E_p(\theta+\pi/2)$)}. %We have$:$ 
\smallskip

\noindent
$d)$ The plane $\Pi_{\theta,\theta+\frac{\pi}{2}}$ contains $p$ $($the origin of $N_pM$$)$ and contains 
the vertex $\Rbf$;  
\smallskip

\noindent
$e)$ For any unit vector $\nbf\in\Pi_{\theta,\theta+\frac{\pi}{2}}$ the two focal centres of $M$ on the 
line $\{\a\nbf: \a\in\R\}$, noted $q^{\nbf}_{\theta}$ and $q^{\nbf}_{\theta+\frac{\pi}{2}}$, 
have as respective sphere-contact directions $\ell_{\theta}$ and $\ell_{\theta+\pi/2}$, and
%$:$ the line $\ell_{\theta}$ is the sphere-contact direction of all focal centres on the line formed by the 
%poles of the tangent planes to $\E_p$ at $\hat{\E_p}(\theta)$. 
their polars are the planes orthogonal to $\nbf$ which are tangent to $\E_p$, at $\E_p(\theta)$ and 
$\E_p(\theta+\pi/2)$. 
\end{theorem}

\begin{proof}
$a)$ 
The line $\ell_\theta$ directed by $\pm\ubf_\theta=\pm(\cos\theta\ebf_1+\sin\theta\ebf_2)\in T_pM$ is sent  
to the point $\E_p(\theta)$ in $\E_p$. 
The planes tangent to $\E_p$ at $\E_p(\theta)$ contain the line tangent to $\E_p$ at 
$\E_p(\theta)$, noted $\l_\theta$; the poles of these planes form a (projective) line 
$\l_\theta^\vee\subset\mathcal{C}_p$. Thus the sphere-contact direction at $p$ of all these 
poles (focal centres) is $\ell_\theta$. 
\smallskip

\noindent
$b)$ 
By item $a$, $\mathcal{C}_p$ is the union of the (projective) lines $\l_\theta^\vee$, $\theta\in[0,\pi]$. 
%Since the plane $\Pi_\E$ contains all lines tangent to $\E_p$, 
%its pole $\Rbf$ belongs to all lines $\l_\theta^\vee$. Hence 
Thus, $\mathcal{C}_p$ is a cone whose vertex $\Rbf$ is the pole of the plane $\Pi_\E$. 
In general, the dual of an ellipse $\mathcal{E}$ in $\RP^3$ is a cone whose vertex corresponds 
to the plane that contains $\mathcal{E}$. 
\smallskip

\noindent
$c)$ 
The point $p$ lies on a convex component of $N_pM\smallsetminus \mathcal{C}_p$ 
by Theorem\,\ref{th:focal_quadric-indicatrix}\textit{\textbf{b}}. 
\smallskip 

\noindent
$d)$ Given two lines $\ell_\theta$, $\ell_{\theta+\frac{\pi}{2}}$ in $T_pM$, their images 
by the map $\E_p$ are opposite points in $\E_p$ - their respective tangent lines 
$\l_\theta$, $\l_{\theta+\pi/2}$ are parallel, implying that the plane $\Pi_{\theta,\theta+\frac{\pi}{2}}$ contains $p$.
The plane $\Pi_{\theta,\theta+\frac{\pi}{2}}$ contains $\Rbf$ because the dual lines $\l_\theta^\vee$ and $\l_{\theta+\frac{\pi}{2}}^\vee$ (formed by the poles of 
the planes that contain the tangent lines $\l_\theta$ and $\l_{\theta+\frac{\pi}{2}}$) contain $\Rbf$. 
\medskip

\noindent
$e)$ By item $a$, given a vector $\nbf\in\Pi_{\theta,\theta+\frac{\pi}{2}}$, the points at which 
the line $\{\a\nbf:\a\in\R\}$ cuts $\mathcal{C}_p$, $q^{\nbf}_{\theta}$ and $q^{\nbf}_{\theta+\frac{\pi}{2}}$,  
lie on the lines $\l_\theta^\vee$ and $\l_{\theta+\frac{\pi}{2}}^\vee$ of  $\mathcal{C}_p$. Thus,  
their respective sphere-contact directions are $\ell_\theta$ and $\ell_{\theta+\frac{\pi}{2}}$. 
Their polars are orthogonal to $\nbf$ because the polar of a point 
(with respect to a sphere) is orthogonal to the vector determined by that point. 
\end{proof}

\begin{figure}[ht] %< préférence de placement h , t , b or p >
\centering
\includegraphics[scale=0.13]{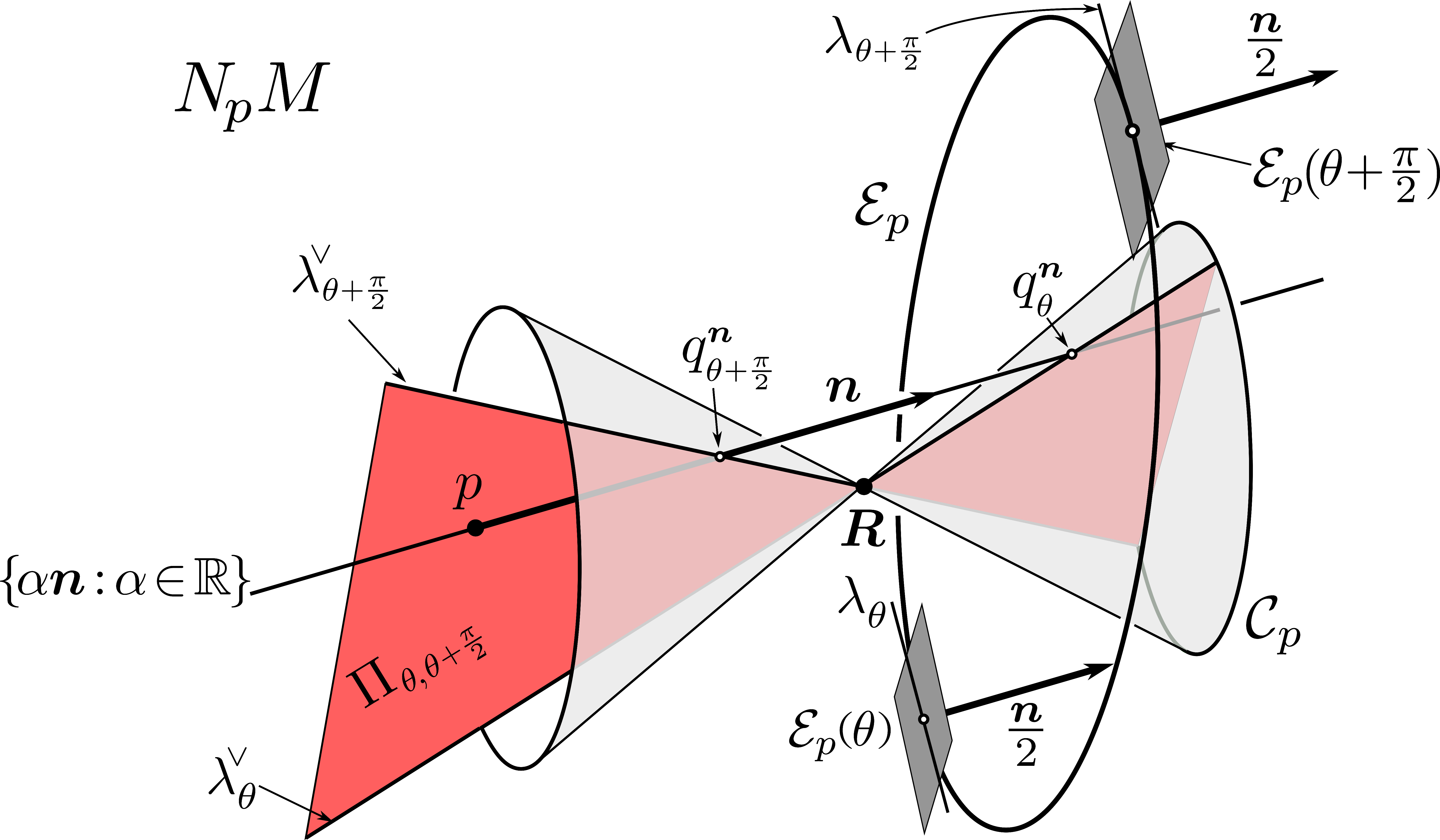}
\caption{\small The focal centres at which the line $\{\a\nbf:\a\in\R\}$ cuts $\mathcal{C}_p$ are the poles of 
the tangent planes to $\E_p$ orthogonal to $\nbf$, and belong to the lines $\l^\vee_\theta$, $\l^\vee_{\theta+\pi/2}$ 
dual to the tangents $\l_\theta$, $\l_{\theta+\pi/2}$.}
\label{fig:Cp-lines-Ep-points}
\end{figure}

%%%%%%%%%%%%%%%%%%%%%%%%%%%%%%%%%%%%%%%%%%%%%%%%%%%%%%%%%%%%%%%%%%%%%%%%%%%%%%%%%%%%%%%%%%%%%%%%%%%%%%%%%%%%%%%%%%%%
%%%%%%%%%%%%%%%%%%%%%%%%%%%%%%%%%%%%%%%%%%%%  PROPOSITION  %%%%%%%%%%%%%%%%%%%%%%%%%%%%%%%%%%%%%%%%%%%%%%%%%%%%%%%%%
%%%%%%%%%%%%%%%%%%%%%%%%%%%%%%%%%%%%%%%%%%%%%%%%%%%%%%%%%%%%%%%%%%%%%%%%%%%%%%%%%%%%%%%%%%%%%%%%%%%%%%%%%%%%%%%%%%%%
\begin{proposition}\label{prop:R=N/2tau}
Let $p$ be a point of a surface $M$ in $\R^5$ at which $\tau\neq 0$.
%and consider the vector $\nubf:=\dfrac{(\Abf-\Cbf)}{2}\times \Bbf$ in the normal space $N_pM$. 
The umbilical focal centre $\Rbf$ {\rm (the vertex of the cone $\mathcal{C}_p$)} has the explicit expression\,$:$
\begin{equation}\label{eq:expression R}
\Rbf=\frac{1}{2\tau}(\Abf-\Cbf)\times\Bbf\,.  
\end{equation}
\end{proposition} 
%%%%%% P R O O F >>>>>>>>>>>>>>>>>>>>>>>>>>>>>>>>>>>>>>>>>>>>>>>>>>>>>>>>>>>>>>>>>>>>>>>>>>>>>>>>>>>>>>>>>>>>>>>>>>
\begin{proof}
The equation of the polar plane of $\Rbf$ (the plane that contains the ellipse $\E_p$) is 
\begin{equation}\label{equation:plane-ellipse}
  \langle v\,,\, q-\Hbf\rangle=0, 
\end{equation}
where $v$ is any vector orthogonal to that plane and $\Hbf=(\Abf+\Cbf)/2$. 
Equation\,\eqref{equation:plane-ellipse} is equivalent to 
\begin{equation*}
  \langle\, v/\langle v,\Hbf\rangle\,,\, q\,\rangle=1\,.
\end{equation*}
Therefore, the pole of that plane with respect to the unit sphere is $\Rbf=v/\langle v,\Hbf\rangle$. 

Taking $v=\Nbf:=(\Abf-\Cbf)\times\Bbf$ as normal vector, we get from \eqref{eq:def-torsion} the equalities
\[
\left\langle v\,,\, \Hbf \right\rangle=\left\langle\, (\Abf-\Cbf)\times\Bbf\,,\, \Hbf\,\right\rangle=2\tau\,. 
\]
We conclude that \, $\Rbf\,=\,\Nbf/2\tau$. 
\end{proof}
%%%%%%<<<<<<<<<<<<<<<<<<<<<<<<<<<<<<<<<<<<<<<<<<<<<<<<<<<<<<<<<<<<<<<<<<<<<<<<<<<<<<<<<<<<<<<<<<<<<<<<< P R O O F %%

\noindent
\textit{\textbf{\small Pseudo-Elliptic, -Hyperbolic, -Parabolic Points in $\R^5$}}. 
A point $p\in M$ in $\R^5$ where $\E_p$ is non-degenerate is said to be \textit{pseudo-elliptic} 
(\textit{-hyperbolic}, \textit{-parabolic}) if the plane $\Pi_\E\subset N_pM$ 
which contains $\E_p$ cuts $\mathcal{C}_p$ along an ellipse (resp. hyperbola, parabola). If the plane 
$\Pi_\E$ contains $p$, then $p$ is said to be \textit{flat-elliptic} 
(resp. \textit{-hyperbolic}, \textit{-parabolic}). 

\begin{note*}
\textit{The vertex $\Rbf$ of $\mathcal{C}_p$ lies on the plane $\Pi_\E$ iff $\|\Rbf\|=1$} (because 
$\langle q, \Rbf\rangle=1$ is the equation of $\Pi_\E$). In that a case, the ellipse $\mathcal{C}_p\cap\Pi_\E$ 
(resp. hyperbola, parabola) degenerates into a point (resp. two transverse lines, two coinciding lines).
\end{note*}

\noindent
\textit{\textbf{We note $\Sigma_p^*$ the cone {\rm(with vertex $p$)} in $N_pM$ based on the ellipse}} $\E_p$. 
Its equation in $N_pM$ is $\mathcal{G}_p^*(q)=0$ (this follows from Theorem\,\ref{th:equationsSigma_pSigma_p^*} 
proved below). 
%($M\subset\R^5$).

%%%%%%%%%%%%%%%%%%%%%%%%%%%%%%%%%%%%%%%%%%%%%%%%%%%%%%%%%%%%%%%%%%%%%%%%%%%%%%%%%%%%%%%%%%%%%%%%%%%%%%%%%%%%%%%%%%%%
%%%%%%%%%%%%%%%%%%%%%%%%%%%%%%%%%%%%%%%%%%%%    THEOREM    %%%%%%%%%%%%%%%%%%%%%%%%%%%%%%%%%%%%%%%%%%%%%%%%%%%%%%%%%
%%%%%%%%%%%%%%%%%%%%%%%%%%%%%%%%%%%%%%%%%%%%%%%%%%%%%%%%%%%%%%%%%%%%%%%%%%%%%%%%%%%%%%%%%%%%%%%%%%%%%%%%%%%%%%%%%%%%
\begin{theorem}\label{prop:caustic-cone}
Let $M$ be a surface in $\R^5$ and $\E_p$ be its indicatrix ellipse at $p$. 
\smallskip

\noindent
\textit{\textbf{A}}. If $\E_p$ is non-degenerate, then 
in the projective completion $N_pM\sqcup\RP^2_\infty$ we have  {\rm (Fig.\,\ref{fig:cone-quadrics}):} 
\smallskip

%%%%%%%%%%%%%%%%%%%%%%%%%%%%%%%%
\noindent
\textit{\textbf{a}}$_\psi$. The plane $\Pi_\E$ does not contain $p$ iff the vertex $\Rbf$ of the cone $\mathcal{C}_p$ 
is a point of $N_pM$ {\rm (i.e., $\Rbf$ is not at infinity)}. Moreover, both for $\tau>0$ and for $\tau<0$, we have three cases$:$
\smallskip

\noindent
$(\psi e)$ $p$ is pseudo-elliptic \,iff\, $\mathcal{G}^*_p(\Rbf)>0$ \,iff\, $\Rbf$ lies inside of the cone $\Sigma_p^*$; 
\smallskip

\noindent
$(\psi h)$ $p$ is pseudo-hyperbolic \,iff\, $\mathcal{G}^*_p(\Rbf)<0$ \,iff\, $\Rbf$ lies outside of the cone $\Sigma_p^*$; 
\smallskip

\noindent
$(\psi p)$ $p$ is pseudo-parabolic \,iff\, $\mathcal{G}^*_p(\Rbf)=0$ \,iff\, $\Rbf$ lies on the cone $\Sigma_p^*$. 
\smallskip

%%%%%%%%%%%%%%%%%%%%%%%%%%%%%%%%
\noindent
\textit{\textbf{a}}$_\f$. The plane $\Pi_\E$ contains $p$ iff $\Rbf$ lies in $\RP^2_\infty$ $($representing 
the direction orthogonal to $\Pi_\E$$)$ iff $\mathcal{C}_p$ is a cylinder in $N_pM$ orthogonal 
to $\Pi_\E$. Here, $\tau=0$ and moreover 
\smallskip

\noindent
$(\f e)$ $p$ is flat-elliptic \,iff\, $p$ lies inside $\E_p$ \,iff\, $\mathcal{C}_p$ is an elliptic cylinder; 
\smallskip 

\noindent
$(\f h)$ $p$ is flat-hyperbolic \,iff\, $p$ lies outside $\E_p$ \,iff\, $\mathcal{C}_p$ is a hyperbolic cylinder; 
\smallskip

\noindent
$(\f p)$ $p$ is flat-parabolic \,iff\, $p$ lies on $\E_p$ \,iff\, $\mathcal{C}_p$ is a parabolic cylinder; 
\medskip

\noindent
\textit{\textbf{B}}. If $\E_p$ is a segment, then the local caustic $\mathcal{C}_p$ is the union of the polar planes 
$($in $N_pM\sqcup\RP^2_\infty$$)$ of the end-points of $\E_p$. Their intersection line consists of umbilical foci: 
the poles of all planes that contain $\E_p$. In $N_pM$, we have four cases {\rm (see Fig.\,\ref{fig:planes-deg_E_p}):} 
\smallskip

\noindent
$(su)$ $p$ is semi-umbilic iff the planes forming $\mathcal{C}_p$ intersect at a line (of umbilical focal centres) 
-- its direction is orthogonal to the plane determined by $p$ and the segment $\E_p$; 
\smallskip

\noindent
$(ri)$ $p$ is real inflection iff the planes forming $\mathcal{C}_p$ are parallel, and $p$ is between them; 
\smallskip

% \noindent
% $(\mathring{ri})$ $p$ is real inflection and $p$ is the centre of $\E_p$ iff $p$ is equidistant to the parallel lines 
% forming $\mathcal{C}_p$; 
% \smallskip

\noindent
$(\f i)$ $p$ is flat inflection iff one of the planes forming $\mathcal{C}_p$ is the plane at infinity; 
\smallskip

\noindent
$(ii)$ $p$ is imaginary inflection iff the planes of $\mathcal{C}_p$ are parallel and $p$ is not between them;
\medskip

\noindent
\textit{\textbf{C}}. If $\E_p$ is just a point $q$, then $\mathcal{C}_p$ is the polar plane 
$\mathcal{P}_q$ of $q$ counted twice. Two cases$:$ 
\smallskip 

\noindent
$(u)$ $p$ is umbilic ($q\neq p$) iff the double plane $\mathcal{C}_p=2\times\mathcal{P}_q$ is not at infinity. 
\smallskip 

\noindent
$(\f u)$ $p$ is flat umbilic, i.e. $q=p$ iff \, $\mathcal{C}_p=2\times\RP^2_\infty$ \, is twice the plane at infinity. 
\end{theorem}
\begin{figure}[ht] %< préférence de placement h , t , b or p >
\centering
\includegraphics[scale=0.13]{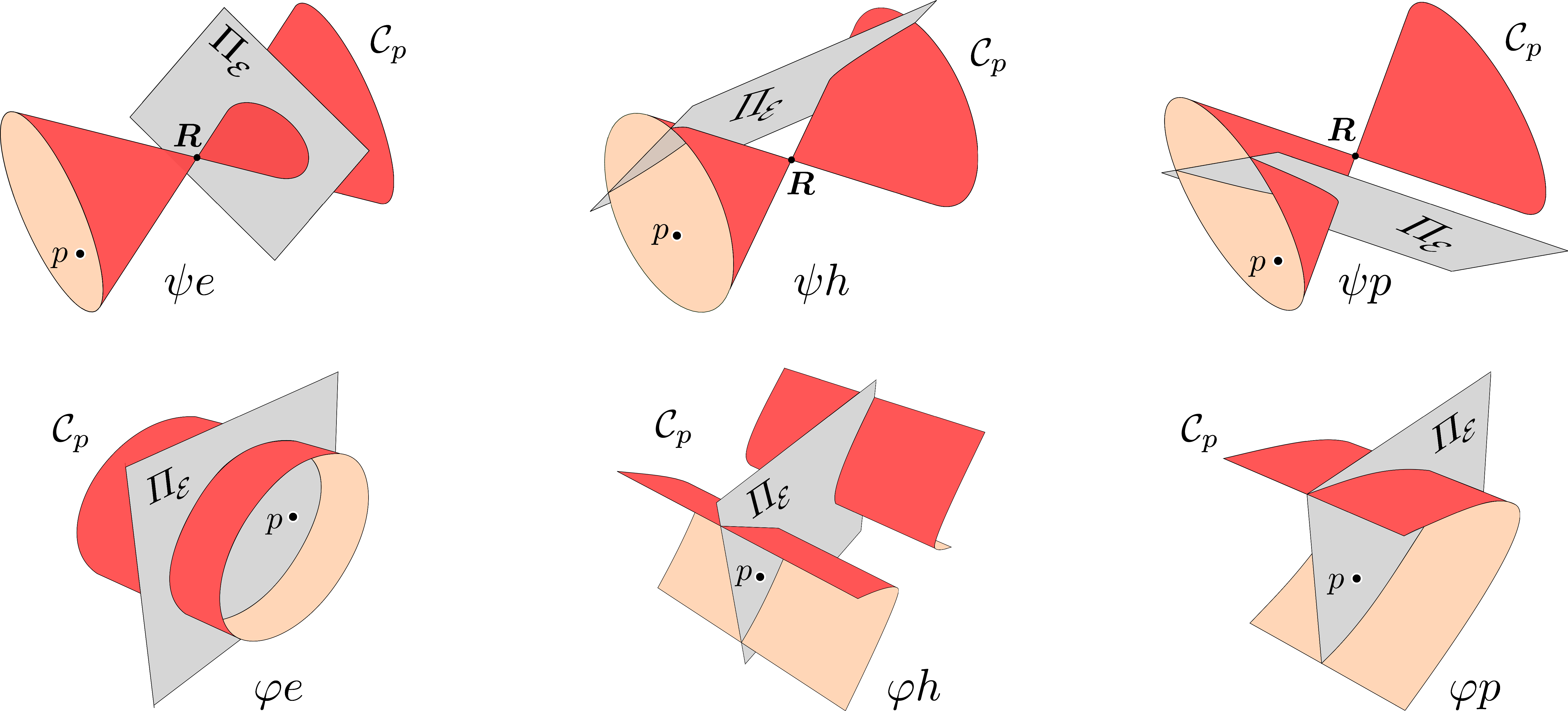}
\caption{\small Local caustic $\mathcal{C}_p$ and plane $\Pi_\E$ when the ellipse $\E_p$ is non degenerate, for $M$ in $\R^5$.}
\label{fig:cone-quadrics}
\end{figure}
% \begin{figure}[ht] %< préférence de placement h , t , b or p >
% \centering
% \includegraphics[scale=0.14]{cone-quadrics}
% \caption{\small The local caustic $\mathcal{C}_p$ when the indicatrix ellipse $\E_p$ is non degenerate, for $M$ in $\R^5$.}
% \label{fig:cone-quadrics}
% \end{figure}
\begin{figure}[ht] %< préférence de placement h , t , b or p >
\centering
\includegraphics[scale=0.14]{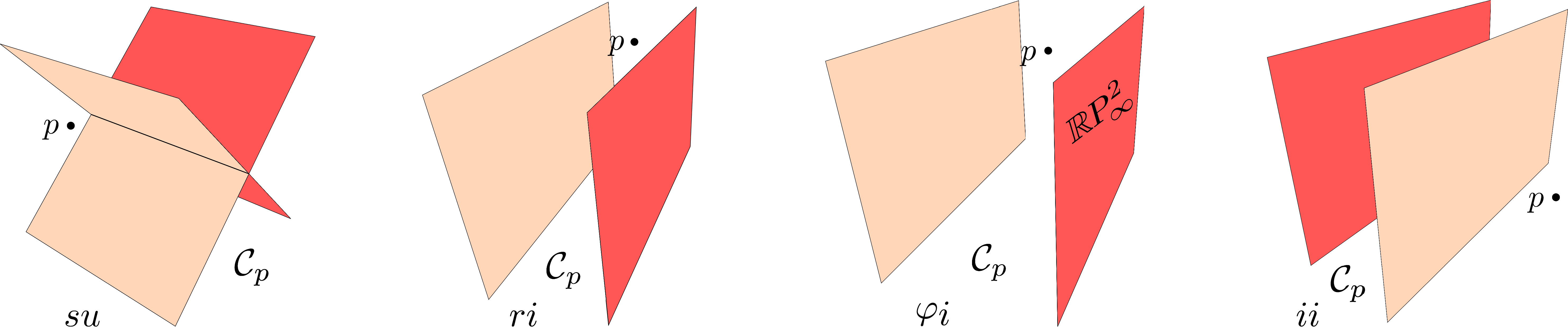}
\caption{\small The local caustic $\mathcal{C}_p$ when the indicatrix ellipse $\E_p$ is a segment, for $M$ in $\R^5$.}
\label{fig:planes-deg_E_p}
\end{figure}
%%%%%% P R O O F >>>>>>>>>>>>>>>>>>>>>>>>>>>>>>>>>>>>>>>>>>>>>>>>>>>>>>>>>>>>>>>>>>>>>>>>>>>>>>>>>>>>>>>>>>>>>>>>>>
\begin{proof}[On the proofs of Theorems {\rm \ref{th:local-caustic-R4}} and {\rm \ref{prop:caustic-cone}}]
Theorems \ref{th:local-caustic-R4} and \ref{prop:caustic-cone} are mostly analogous: 
in Th.\,\ref{th:local-caustic-R4} one takes the poles of the lines tangent to $\E_p\subset N_pM\approx\R^2$ 
and in Th.\,\ref{prop:caustic-cone} the poles of the planes tangent to $\E_p\subset N_pM\approx\R^3$. 
Their proofs, based on Theorem\,\ref{th:focal_quadric-indicatrix} and on the classification of points by 
the indicatrix ellipse, follow the same line of elementary duality ideas, as in the proof of 
Theorem\,\ref{th:C_p=cone}. The only difference is in item \textit{\textbf{a}}$_\psi$: 
\smallskip

\noindent
\underline{\textit{Proof of} \textit{\textbf{a}}$_\psi$}. The plane $\Pi_\E$ -of equation $\langle q,\Rbf\rangle=1$- 
cuts the cone $\mathcal{C}_p$ -of equation $\mathcal{G}_p(q-\Rbf)=0$- 
along an ellipse (hyperbola, parabola)  iff the intersection of the parallel vectorial plane 
$\flecha{\Pi}_\E$ -of equation $\langle q,\Rbf\rangle=0$- with the cone $\Sigma_p$ -of equation $\mathcal{G}_p(q)=0$- 
consists of the origin (resp. two transverse lines, a double line). This happens iff the restriction of the 
quadratic form $\mathcal{G}_p$ to the plane $\flecha{\Pi}_\E$ is definite (resp. undefinite, degenerate). 
In terms of $q_1$, $q_2$, the restriction of $\mathcal{G}_p$ to $\flecha{\Pi}_\E$ is given by 
\[{\mathcal{G}_p}_{|_{\flecha{\Pi}_\E}}(q_1,q_2)=
\left(\mathcal{K}_1\rho_3+\mathcal{K}_3\rho_1^2/\rho_3\right)q_1^2+
2\left(\mathcal{K}_3\rho_1\rho_2/\rho_3\right)q_1q_2+ 
\left(\mathcal{K}_2\rho_3+\mathcal{K}_3\rho_2^2/\rho_3\right)q_2^2\,.\]

Thus, the determinant of ${\mathcal{G}_p}_{|_{\flecha{\Pi}_\E}}$ (which does not depend on $q_1$, $q_2$) is equal to 
\begin{align*}
\hspace{2cm} \det\left({\mathcal{G}_p}_{|_{\flecha{\Pi}_\E}}\right) & =
\left(\mathcal{K}_1\rho_3+\mathcal{K}_3\rho_1^2/\rho_3\right)\left(\mathcal{K}_2\rho_3+\mathcal{K}_3\rho_2^2/\rho_3\right)
-\left(\mathcal{K}_3\rho_1\rho_2/\rho_3\right)^2 \\
 &=\Delta\left(\rho_1^2/\mathcal{K}_1+\rho_2^2/\mathcal{K}_2+\rho_3^2/\mathcal{K}_3\right)
 =\Delta\,\mathcal{G}_p^*(\Rbf)\,.
\end{align*}

Hence the point $p$ is pseudo-elliptic iff $\mathcal{G}_p^*(\Rbf)>0$, 
pseudo-hyperbolic iff $\mathcal{G}_p^*(\Rbf)<0$ and pseudo-parabolic iff $\mathcal{G}_p^*(\Rbf)=0$. \ 
The last equivalences in the subitems of \textit{\textbf{a}}$_\psi$ follow because the equation of 
the cone $\Sigma_p^*$ 
(based on $\E_p$) is $\mathcal{G}_p^*(q)=0$. 
\smallskip

\noindent
\underline{\textit{Proof of} \textit{\textbf{a}}$_\varphi$}. Since at a flat point the plane $\Pi_\E$ 
contains $p$ (the origin of $N_pM$), the vertex of $\mathcal{C}_p$ (which is the pole of $\Pi_\E$) lies at 
infinity in $N_pM\sqcup\RP^2_\infty$. Thus, in $N_pM$, $\mathcal{C}_p$ is a quadratic cylinder, whose type 
is determined by restricting the duality to $\Pi_\E$.  
\end{proof}
%%%%%%<<<<<<<<<<<<<<<<<<<<<<<<<<<<<<<<<<<<<<<<<<<<<<<<<<<<<<<<<<<<<<<<<<<<<<<<<<<<<<<<<<<<<<<<<<<<<<<<< P R O O F %%

Notice that in Th.\,4.2 of \protect\cite{Costa_Moraes_R-F}, the analogue of item \textit{\textbf{a}}$_\f$$\,(\f e)$ 
of our Theorem\,\ref{prop:caustic-cone} states that $\mathcal{C}_p=\emptyset$, 
%while, by Th.\,\ref{th:focal_quadric-indicatrix}, 
but $\mathcal{C}_p$ is an elliptic cylinder -- Fig.\,\ref{fig:cone-quadrics} ($\f e$) -- see the following example.  

%%%%%%%%%%%%%%%%%%%%%%%%%%%%%%%%%%%%%%%%%%%%%   EXAMPLE   %%%%%%%%%%%%%%%%%%%%%%%%%%%%%%%%%%%%%%%%%%%%%%%%%%%%%%%%%
%%%%%%%%%%%%%%%%%%%%%%%%%%%%%%%%%%%%%%%%%%%%%%%%%%%%%%%%%%%%%%%%%%%%%%%%%%%%%%%%%%%%%%%%%%%%%%%%%%%%%%%%%%%%%%%%%%%
\begin{example*}[Local caustic at a flat elliptic point]
Let $M$ be a surface in $\R^5$ whose local quadratic map at $p$ (the origin) is given by the following three quadratic forms 
\[Q_1=\textstyle\frac{1}{2}(\a s^2+2\beta st-\a t^2)\,,\quad Q_2=\frac{1}{2}(as^2+2bst-at^2)\,,\quad Q_3=\frac{1}{2}(es^2+2fst-et^2)\,.\]
We have that the mean curvature vector $\Hbf$ is the zero vector because 
\[\Abf=\begin{pmatrix}
 \a \\ 
  a \\
  e
\end{pmatrix}, \ 
\Bbf=\begin{pmatrix}
  \beta \\ 
   b \\
   f
\end{pmatrix}, \ 
\Cbf=\begin{pmatrix}
 -\a \\ 
  -a \\
  -e
\end{pmatrix}=-\Abf\,.\]  
Thus $p$ is flat elliptic (the plane $\Pi_\E$ contains $p$). The equation of the local caustic, 
\[(\langle \Abf\,,\,q\rangle-1)(\langle \Cbf\,,\,q\rangle-1)-\langle \Bbf\,,\,q\rangle^2=0 \,,\]
reduces, in this case, to 
\[\langle \Abf\,,\,q\rangle^2+\langle \Bbf\,,\,q\rangle^2=1\,,\]
which is the equation of an elliptic cylinder -- Fig.\,\ref{fig:cone-quadrics} ($\f e$). 
\end{example*}

%%%%%%%%%%%%%%%%%%%%%%%%%%%%%%%%%%%%%  DEFINITION Mi  %%%%%%%%%%%%%%%%%%%%%%%%%%%%%%%%%%%%%%%%%%%%%%%%%%%%%%%%%%%%%%
\noindent
\textit{\textbf{$\boldsymbol{M_i}$ subsets}}. 
Given a smooth surface $M$ in $\R^5$, consider the following subsets of $M$: 
\[M_i=\{p\in M:\dim(\mathrm{span}(\textstyle{\frac{1}{2}}(\Abf-\Cbf), \Bbf, \Hbf))=i\}\,.\]

%%%%%%%%%%%%%%%%%%%%%%%%%%%%%%%%%%%%%%%%%%%    REMARK    %%%%%%%%%%%%%%%%%%%%%%%%%%%%%%%%%%%%%%%%%%%%%%%%%%%%%%%%%%%%%%%%
\begin{remark}
If $M\subset \R^5$ is in general position, the vectors $\frac{1}{2}(\Abf-\Cbf)(p)$, $\Bbf(p)$ and $\Hbf(p)$ are coplanar only at the flat points in $M$, which form the \textit{flat curve} $\Phi$, given by the equation 
\[\tau(p):=\det\left[\textstyle{\frac{1}{2}}(\Abf-\Cbf)(p)\,\, \Bbf(p)\,\, \Hbf(p)\right]=0\,.\] 
Moving along the flat curve $\Phi$, two of these three vectors may be 
colinear at isolated points, but there is no point at which the three vectors are colinear.  This leads to the 
\end{remark}

%%%%%%%%%%%%%%%%%%%%%%%%%%%%%%%%%%%%%%%%%%%%%%%%%%%%%%%%%%%%%%%%%%%%%%%%%%%%%%%%%%%%%%%%%%%%%%%%%%%%%%%%%%%%%%%%%%%%
%%%%%%%%%%%%%%%%%%%%%%%%%%%%%%%%%%%%%%%%%%%%  PROPOSITION  %%%%%%%%%%%%%%%%%%%%%%%%%%%%%%%%%%%%%%%%%%%%%%%%%%%%%%%%%
%%%%%%%%%%%%%%%%%%%%%%%%%%%%%%%%%%%%%%%%%%%%%%%%%%%%%%%%%%%%%%%%%%%%%%%%%%%%%%%%%%%%%%%%%%%%%%%%%%%%%%%%%%%%%%%%%%%%
\begin{proposition}[\protect\cite{M-RM-R2003}]
If $M$ is a generic embedded surface in $\R^5$, then $M=M_3\cup M_2$.
\end{proposition}

\noindent
\textbf{\textit{Seven Types of Points of a Generic Surface in $\R^5$}} (in terms of the local quadratic map $\f_p$). 
A surface $M\subset\R^5$ in general position has open domains $\psi E$ of pseudo-elliptic points and $\psi H$ of 
pseudo-hyperbolic points, separated by a curve $\psi P$ of pseudo-parabolic points - their union is $M_3$. 
The flat curve $\Phi=M_2$, on which $\tau=0$, is formed of open intervals 
$\varphi E$ of flat-elliptic points and $\varphi H$ of flat-hyperbolic points, separated by isolated flat-parabolic 
points $\varphi P$ (where $\psi P$ intersects $\Phi$). In the intervals of flat-hyperbolic points, 
one can have isolated semi-umbilic points ($su$)- Fig.\,\ref{fig:points-surf-R5}. 

\begin{figure}[ht] %< préférence de placement h , t , b or p >
\centering
\includegraphics[scale=0.27]{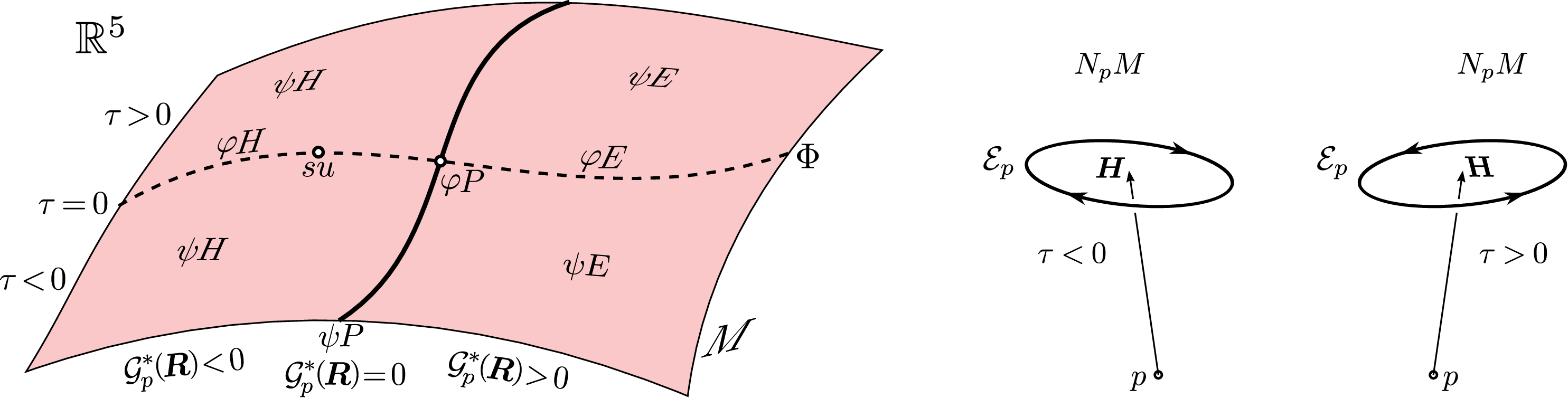}
\caption{\small The seven types of points determined by the local quadratic map $\f_p$ on a generic surface in $\R^5$. 
The sign of $\tau$ indicates the orientation of the parametrisation of $\E_p$ with respect to $\Hbf$, 
according to the left- or right-hand rule. 
Counting the sign of $\tau$ one gets ten types of points.}
\label{fig:points-surf-R5}
\end{figure}   

%%%%%%%%%%%%%%%%%%%%%%%%%%%%%%%%%%%%%%%%%%%%%%%%%%%%%%%%%%%%%%%%%%%%%%%%%%%%%%%%%%%%%%%%%%%%%%%%%%%%%%%%%%%%%%%%%%%%
%%%%%%%%%%%%%%%%%%%%%%%%%%%%%%%%%%%%%%%%%%%%  PROPOSITION  %%%%%%%%%%%%%%%%%%%%%%%%%%%%%%%%%%%%%%%%%%%%%%%%%%%%%%%%%
%%%%%%%%%%%%%%%%%%%%%%%%%%%%%%%%%%%%%%%%%%%%%%%%%%%%%%%%%%%%%%%%%%%%%%%%%%%%%%%%%%%%%%%%%%%%%%%%%%%%%%%%%%%%%%%%%%%%
\begin{proposition}
Let $M$ be a generic embedded surface in $\R^5$. If $p\in M$ is not semi-umbilic, then in the direction 
orthogonal to the plane $\Pi_\E$, $\nbf:=\Rbf/\|\Rbf\|$, the Monge form is umbilical and the Gauss curvature equals 
$\|\Rbf\|^{-2}$,  $K^{\nbf}=1/\|\Rbf\|^2$. Namely, 
\begin{equation}\label{eq:<fi_p,n>}
  \left\langle \, \f_p(s,t)\,,\,\Rbf/\|\Rbf\|\, \right\rangle=\textstyle{\frac{1}{2}}\left(\frac{1}{\|\Rbf\|}\right)(s^2+t^2)\,.
\end{equation}
\end{proposition}

\begin{proof} 
Writing $\f_p$ as $\f_p(s,t)=\frac{1}{2}\left(\frac{1}{2}(\Abf-\Cbf)(s^2-t^2)+2\Bbf st+\Hbf(s^2+t^2)\right)$, 
and using \eqref{eq:expression R} and \eqref{eq:def-torsion}, the quadratic part of $M$ at $p$ in the direction 
orthogonal to $\Pi_\E$ is: 
\begin{align*}
\f_p^{\nbf}(s,t):=\left\langle \f_p(s,t), \frac{(\Abf-\Cbf)\times\Bbf}{\|(\Abf-\Cbf)\times\Bbf\|}\right\rangle & =
{\frac{1}{2}}\left\langle \Hbf,\frac{(\Abf-\Cbf)\times\Bbf}{2\tau\|\Rbf\|}\right\rangle(s^2+t^2) \\
 & =\frac{1}{2}\left(\frac{1}{\|\Rbf\|}\right)(s^2+t^2)\,,
\end{align*}
which shows that the direction orthogonal to $\Pi_\E$ is umbilical and $K^{\nbf}=\|\Rbf\|^{-2}$. 
\end{proof}

\noindent 
\textbf{\textit{Flattening Hyperplane}}. At the flat points of $M$ in $\R^5$, the umbilical focal centre is at infinity, 
being the point of $\RP^2_{\infty}$ which represents the direction orthogonal to $\Pi_\E$ -  in this direction, by \eqref{eq:<fi_p,n>}, the quadratic part of $M$ vanishes and the surface $M$ is flat-umbilic: 
the \textit{flattening hyperplane} $\mathrm{span}(T_pM\cup\Pi_\E)$ has higher contact with $M$ along all tangent directions. 
Moreover, this flatteining hyperplane has over-osculation with $M$ along the 
\textit{asymptotic lines}: on which the cubic part of $M$ (in the direction orthogonal to $\Pi_\E$), $\mathrm{Cub}_p^{\nbf}(s,t)$, vanishes. 
Then, the curve $\Phi$ contains isolated points at which $M$ has two asymptotic lines (one of them double). 
These points separate $\Phi$ into open intervals of two types: the points have three asymptotic lines or the points have one asymptotic line. 
So, the asymptotic lines are not detected by the local quadratic map.  
 
\begin{remark*}
The local caustic $\mathcal{C}_p$ of a surface in $\R^{n>5}$ is also described by  
Theorems\,\ref{th:C_p=cone} and \ref{prop:caustic-cone} (Figs.\,\ref{fig:cone-quadrics} and \ref{fig:planes-deg_E_p}),
 but one has to ``multiply'' orthogonally by the 
complementary normal subspace to the space spanned by $\Abf$, $\Bbf$, $\Cbf$. 
\end{remark*}

\section{Relations between $\E_p$, $\Hbf$, $\mathcal{C}_p$, $\Rbf$ and the Local Invariants}\label{Section-relation-inequalities}
%%%%%%%%%%%%%%%%%%%%%%%%%%%%%%%%%%%%%%%%%%%%%%%%%%%%%%%%%%%%%%%%%%%%%%%%%%%%%%%%%%%%%%%%%%%%%%%%%%%%%%%%%%%%
%%%%%%%%%%%%%%%%%%%%%%%%%%%%%%%%%%%%%%%%%%%%%%%%%%%%%%%%%%%%%%%%%%%%%%%%%%%%%%%%%%%%%%%%%%%%%%%%%%%%%%%%%%%%

%%%%%%%%%%%%%%%%%%%%%%%%%%%%%%%%%%%%%%%%%%%%%%%%%%%%%%%%%%%%%%%%%%%%%%%%%%%%%%%%%%%%%%%%%%%%%%%%%%%%%%%%%%%%%%%%%%%%%%%%%%%
%%%%%%%%%%%%%%%%%%%%%%%%%%%%%%%%%%%%%%%%%%%%%%%%%%%%%%%%%%%%%%%%%%%%%%%%%%%%%%%%%%%%%%%%%%%%%%%%%%%%%%%%%%%%%%%%%%%%%%%%%%%
\subsection{\textbf{The Relation Between $\Hbf$ and $\Rbf$ by the Gauss Quadratic Form}}
%%%%%%%%%%%%%%%%%%%%%%%%%%%%%%%%%%%%%%%%%%%%%%%%%%%%%%%%%%%%%%%%%%%%%%%%%%%%%%%%%%%%%%%%%%%%%%%%%%%%%%%%%%%%%%%%%%%%%%%%%%%
%%%%%%%%%%%%%%%%%%%%%%%%%%%%%%%%%%%%%%%%%%%%%%%%%%%%%%%%%%%%%%%%%%%%%%%%%%%%%%%%%%%%%%%%%%%%%%%%%%%%%%%%%%%%%%%%%%%%%%%%%%%

The centre $\Hbf$ of the indicatrix ellipse $\E_p$ and the centre $\Rbf$ of the local caustic $\C_p$
are related by the symmetric linear map $\flecha{\mathcal{G}_p}:N_pM\to N_pM$ associated to $\mathcal{G}_p$\,:  

%%%%%%%%%%%%%%%%%%%%%%%%%%%%%%%%%%%%%%%%%%%%%%%%%%%%%%%%%%%%%%%%%%%%%%%%%%%%%%%%%%%%%%%%%%%%%%%%%%%%%%%%%%%%%%%%%%%%
%%%%%%%%%%%%%%%%%%%%%%%%%%%%%%%%%%%%%%%%%%%%  PROPOSITION  %%%%%%%%%%%%%%%%%%%%%%%%%%%%%%%%%%%%%%%%%%%%%%%%%%%%%%%%%
%%%%%%%%%%%%%%%%%%%%%%%%%%%%%%%%%%%%%%%%%%%%%%%%%%%%%%%%%%%%%%%%%%%%%%%%%%%%%%%%%%%%%%%%%%%%%%%%%%%%%%%%%%%%%%%%%%%%
\begin{proposition}\label{proposition:H=QR}
Let $p$ be a point of a smooth surface in $\R^4$ $($or in $\R^5$$)$ with $\Delta\neq 0$. 
\smallskip

\noindent
$(a)$ The symmetric linear map $\flecha{\mathcal{G}_p}$ sends the centre $\Rbf$ of the conic $\mathcal{C}_p$ 
(or the vertex $\Rbf$ of the cone $\C_p$) to the mean curvature vector $\Hbf$$:$ 
\begin{equation}\label{eq:H=G_pR}
  \Hbf=\flecha{\mathcal{G}_p}\Rbf\,. 
\end{equation}

\noindent
$(b)$ In the coordinates associated to a principal basis on $N_pM$, the centre $\Rbf$ of the conic $\mathcal{C}_p$ 
$($or the vertex $\Rbf$ of the cone $\mathcal{C}_p$$)$ is given in terms of the principal focal curvatures by 
\begin{equation}\label{eq:relation_for_R}
\Rbf=\left(\frac{h_1}{\mathcal{K}_1}, \frac{h_2}{\mathcal{K}_2}\right)
\quad \mbox{for } M\subset\R^4\,, 
\end{equation}
\begin{equation}\label{eq:relation_for_R-R^5}
\Rbf=\left(\frac{h_1}{\mathcal{K}_1}, \frac{h_2}{\mathcal{K}_2}, \frac{h_3}{\mathcal{K}_3}\right) 
\quad \mbox{for } M\subset\R^5\,. 
\end{equation}
\end{proposition}
%%%%%% P R O O F >>>>>>>>>>>>>>>>>>>>>>>>>>>>>>>>>>>>>>>>>>>>>>>>>>>>>>>>>>>>>>>>>>>>>>>>>>>>>>>>>>>>>>>>>>>>>>>>>>
\begin{proof}
$(a)$ Both, in the $\R^4$ and $\R^5$ cases, the centre $\Rbf$ of the local caustic is the critical point of the quadratic polynomial 
\begin{equation}\label{eq:quadric equation}
  g(q)=(\,\langle \Abf\,, \,q \rangle-1\,)(\,\langle \Cbf\,, \,q \rangle-1\,)-\langle \Bbf\,, \,q \rangle^2\,.
\end{equation}

In order to find this critical point, in the case of surfaces in $\R^4$, we write 
\[
q=\begin{pmatrix}
  u \\ 
  v
\end{pmatrix}, \ 
\Abf=\begin{pmatrix}
 a_1 \\ 
 a_2
\end{pmatrix}, \ 
\Bbf=\begin{pmatrix}
  b_1 \\ 
  b_2
\end{pmatrix}, \ 
\Cbf=\begin{pmatrix}
  c_1 \\ 
  c_2
\end{pmatrix}, \  
h_1=\frac{a_1+c_1}{2} \
\text{ and } \ 
h_2=\frac{a_2+c_2}{2}\,.
\]

Thus the equality $\D_ug(q)=0$ is written as  
\[
a_1(\,\langle \Cbf\,, \,q \rangle-1\,)+c_1(\langle \Abf\,, \,q \rangle-1\,)-2b_1\langle \Bbf\,, \,q \rangle=0\,,
\]
which is equivalent to each of the following equations\,:
\[
a_1\langle \Cbf\,, \,q \rangle+c_1\langle \Abf\,, \,q \rangle-2b_1\langle \Bbf\,, \,q \rangle=a_1+c_1\,,
\]
\[
(a_1c_1-b_1^2)u+\left(\frac{a_1c_2+a_2c_1}{2}-b_1b_2\right)v=h_1\,,
\]
\[
k_1u+k_{12}v=h_1\,, \quad \mbox{( see \eqref{eq:psi-scalar-prod} )}
\]

Similarly, the equality $\D_vg(q)=0$ is equivalent to the equation \, $k_{12}u+k_2v=h_2$.

The last two equations, which determine the critical point 
$\Rbf=
\begin{pmatrix}
  \rho_1 \\ 
  \rho_2
\end{pmatrix}$ of \eqref{eq:quadric equation}, are written as 
\[
\begin{pmatrix}
  k_1 & k_{12}\\ 
  k_{12} & k_2
\end{pmatrix}
\begin{pmatrix}
  \rho_1 \\ 
  \rho_2
\end{pmatrix}
=
\begin{pmatrix}
  h_1 \\ 
  h_2
\end{pmatrix}\,.
\]

In the same way, for a surface in $\R^5$ we get the equality 
\[
\begin{pmatrix}
  k_1    & k_{12} & k_{13} \\ 
  k_{12} & k_2    & k_{23} \\
  k_{13} & k_{23} & k_3 
\end{pmatrix}
\begin{pmatrix}
  \rho_1 \\ 
  \rho_2 \\
  \rho_3
\end{pmatrix}
=
\begin{pmatrix}
  h_1 \\ 
  h_2 \\
  h_3
\end{pmatrix}\,.
\]
Equality \eqref{eq:H=G_pR} is proved. 
\smallskip

\noindent
$(b)$ Taking a principal basis, in both cases ($\R^4$ and $\R^5$), we get a diagonal matrix whose 
entries are the principal focal curvatures (all non zero because $\Delta\neq 0$): 
\[[\mathcal{G}_p]=
\begin{pmatrix}
  \K_1 & 0\\ 
     0 & \K_2
\end{pmatrix}
\quad \mbox{for } M\subset\R^4\,, \quad 
[\mathcal{G}_p]=
\begin{pmatrix}
  \K_1 &  0   & 0 \\ 
   0   & \K_2 & 0 \\
   0   &  0   & \K_3 
\end{pmatrix}
\quad \mbox{for } M\subset\R^5\,, \]
Thus $[\mathcal{G}_p]$ is inversible and then equality \eqref{eq:H=G_pR} implies expressions 
\eqref{eq:relation_for_R} and \eqref{eq:relation_for_R-R^5}. 
\end{proof} 
%%%%%%<<<<<<<<<<<<<<<<<<<<<<<<<<<<<<<<<<<<<<<<<<<<<<<<<<<<<<<<<<<<<<<<<<<<<<<<<<<<<<<<<<<<<<<<<<<<<<<<< P R O O F %%

%%%%%%%%%%%%%%%%%%%%%%%%%%%%%%%%%%%%%%%%%%%%%%%%%%%%%%%%%%%%%%%%%%%%%%%%%%%%%%%%%%%%%%%%%%%%%%%%%%%%%%%%%%%%%%%%%%%%
%%%%%%%%%%%%%%%%%%%%%%%%%%%%%%%%%%%%%%%%%%%%  PROPOSITION  %%%%%%%%%%%%%%%%%%%%%%%%%%%%%%%%%%%%%%%%%%%%%%%%%%%%%%%%%
%%%%%%%%%%%%%%%%%%%%%%%%%%%%%%%%%%%%%%%%%%%%%%%%%%%%%%%%%%%%%%%%%%%%%%%%%%%%%%%%%%%%%%%%%%%%%%%%%%%%%%%%%%%%%%%%%%%%
\begin{proposition}\label{prop:HR=1}
At a point $p$ of a smooth surface $M$ in $\R^5$ with $\Delta\neq 0$, the umbilical focal centre $\Rbf$ 
$($the vertex of $\mathcal{C}_p$$)$ and the mean curvature vector $\Hbf$ satisfy the relation
\begin{equation}\label{eq:RH=1}
  \langle\,\Rbf\,,\,\Hbf\,\rangle=1\,. 
\end{equation}
\end{proposition}
%%%%%% P R O O F >>>>>>>>>>>>>>>>>>>>>>>>>>>>>>>>>>>>>>>>>>>>>>>>>>>>>>>>>>>>>>>>>>>>>>>>>>>>>>>>>>>>>>>>>>>>>>>>>>
\begin{proof}
The equation of the polar plane of the umbilical focal centre $\Rbf$ (the vertex of the cone $\mathcal{C}_p$) 
is $\langle\,\Rbf\,,\,q\,\rangle=1$. By Theorem\,\ref{prop:caustic-cone}, this plane contains 
the indicatrix ellipse and, hence, contains its centre $\Hbf$. 
Thus $\Hbf$ satisfies the relation $\langle\,\Rbf\,,\,\Hbf\,\rangle=1$. 
\end{proof}
%%%%%%<<<<<<<<<<<<<<<<<<<<<<<<<<<<<<<<<<<<<<<<<<<<<<<<<<<<<<<<<<<<<<<<<<<<<<<<<<<<<<<<<<<<<<<<<<<<<<<<< P R O O F %%

For a surface in $\R^4$ relation\,\eqref{eq:RH=1}, $\langle\,\Rbf\,,\,\Hbf\,\rangle=1$, holds only if $N=0$. Namely,  

%%%%%%%%%%%%%%%%%%%%%%%%%%%%%%%%%%%%%%%%%%%%%%%%%%%%%%%%%%%%%%%%%%%%%%%%%%%%%%%%%%%%%%%%%%%%%%%%%%%%%%%%%%%%%%%%%%%%
%%%%%%%%%%%%%%%%%%%%%%%%%%%%%%%%%%%%%%%%%%%%  PROPOSITION  %%%%%%%%%%%%%%%%%%%%%%%%%%%%%%%%%%%%%%%%%%%%%%%%%%%%%%%%%
%%%%%%%%%%%%%%%%%%%%%%%%%%%%%%%%%%%%%%%%%%%%%%%%%%%%%%%%%%%%%%%%%%%%%%%%%%%%%%%%%%%%%%%%%%%%%%%%%%%%%%%%%%%%%%%%%%%%
\begin{proposition}\label{th:<R,H>=1-N^2/4D}
If $\Delta\neq 0$ at a point $p$ of a smooth surface $M$ in $\R^4$, then the 
centre $\Rbf$ of the local caustic $\mathcal{C}_p$ and the mean curvature vector $\Hbf$ satisfy the relation 
\begin{equation}\label{eq:<R,H>=1-N^2/4D}
  \langle\,\Rbf\,,\,\Hbf\,\rangle=1-\frac{N^2}{4\Delta}\,. 
\end{equation}
\end{proposition}
%%%%%% P R O O F >>>>>>>>>>>>>>>>>>>>>>>>>>>>>>>>>>>>>>>>>>>>>>>>>>>>>>>>>>>>>>>>>>>>>>>>>>>>>>>>>>>>>>>>>>>>>>>>>>
\begin{proof}
By Proposition\,\ref{proposition:H=QR}\,$(a)$, we have 
$\langle\,\Rbf\,,\,\Hbf\,\rangle=\langle\,\flecha{\mathcal{G}_p}^*\,\Hbf\,,\,\Hbf\,\rangle=\mathcal{G}_p^*(\Hbf)$. 
Thus 
\[\langle\,\Rbf\,,\,\Hbf\,\rangle=\frac{1}{4\Delta}\left(k_2(a_1+c_1)^2-2k_{12}(a_1+c_1)(a_2+c_2)+k_1(a_2+c_2)^2\right)\,.\]
Opening the parenthesis and rearranging the terms we arrive easily to the equality 
\[\langle\Rbf\,,\Hbf\rangle=
\frac{1}{4\Delta}\left(4\Delta+k_2(a_1^2+2b_1^2+c_1^2)-2k_{12}(a_1a_2+2b_1b_2+c_1c_2)+k_1(a_2^2+2b_2^2+c_2^2)\right).\]
Using $k_1=a_1c_1-b_1^2$, $k_2=a_2c_2-b_2^2$ and $k_{12}=\langle Q_1,Q_2\rangle_\psi$ together with the identity 
\[a_2c_2(a_1^2+c_1^2)+a_1c_1(a_2^2+c_2^2)-(a_1c_2+a_2c_1)(a_1a_2+c_1c_1)=0\,,\] 
we obtain 
\[\langle\Rbf\,,\Hbf\rangle=
\frac{1}{4\Delta}\left(4\Delta-\left(b_2^2(a_1-c_1)^2-2b_1b_2(a_1-c_1)(a_2-c_2)+b_1^2(a_2-c_2)^2\right)\right)\,,\]
that is, \, $\langle\,\Rbf\,,\,\Hbf\,\rangle=1-N^2/4\Delta\,.$
\end{proof}
\subsection{\textbf{The Locus of the Centres $\Hbf$ and $\Rbf$ in Terms of the Local Invariants}}
%%##########################################################################################################
At a point $p$ of a surface in $\R^4$ ($\R^5$) at which $\Delta\neq 0$, with given local invariants 
$\K_1$, $\K_2$, $N$  (resp. $\K_1$, $\K_2$, $\K_3$), Theorem\,\ref{H-positions} (resp. Th.\,\ref{th:hyperbolods-R-and-H}) 
restricts the centres $\Hbf$ and $\Rbf$ to lie in two respective conics (quadrics) $\mathcal{H}$ 
and $\mathcal{R}$ determined by the given invariants.  

%%%%%%%%%%%%%%%%%%%%%%%%%%%%%%%%%%%%%%%%%%%%%%%%%%%%%%%%%%%%%%%%%%%%%%%%%%%%%%%%%%%%%%%%%%%%%%%%%%%%
%%%%%%%%%%%%%%%%%%%%%%%%%%%%%%%%%%%%%%%%  THEOREM  %%%%%%%%%%%%%%%%%%%%%%%%%%%%%%%%%%%%%%%%%%%%%%%%%
%%%%%%%%%%%%%%%%%%%%%%%%%%%%%%%%%%%%%%%%%%%%%%%%%%%%%%%%%%%%%%%%%%%%%%%%%%%%%%%%%%%%%%%%%%%%%%%%%%%%
\begin{theorem}\label{H-positions}
If a smooth surface in $\R^4$ has invariants $N$, $\K_1$, $\K_2$ $(\K_1\K_2\neq 0)$ at $p$, then
the mean curvature vector $\Hbf=(h_1, h_2)$ and the centre $\Rbf=(\rho_1, \rho_2)$ of the local 
caustic $\mathcal{C}_p$ lie on the conic curves of $N_pM$ given by the respective equations
\begin{equation}\label{eq:Gp(H)=1-N2/4D}
 2\mathcal{G}_p^*(q)=1-\frac{N^2}{4\Delta} \qquad \mbox{and} \qquad  2\mathcal{G}_p(q)=1-\frac{N^2}{4\Delta}\,,\qquad\quad
\end{equation}
which in a principal basis have the respective forms 
\begin{equation}\label{equation:MC-vector}
(\mathcal{H}):\,\,\,\frac{q_1^2}{\K_1}+\frac{q_2^2}{\K_2}=1-\frac{N^2}{4\Delta} \qquad \mbox{and} \qquad 
 (\mathcal{R}):\,\,\,\K_1q_1^2+\K_2q_2^2=1-\frac{N^2}{4\Delta}\,.\quad
\end{equation}
They are ellipses for $p$ elliptic and hyperbolas for $p$ hyperbolic {\rm (Fig.\ref{fig:conics-of-H-R})}. 
\end{theorem}
%%%%%% P R O O F >>>>>>>>>>>>>>>>>>>>>>>>>>>>>>>>>>>>>>>>>>>>>>>>>>>>>>>>>>>>>>>>>>>>>>>>>>>>>>>>>>>>>>>>>>>>>>>>>>
\begin{proof} 
By Proposition\,\ref{proposition:H=QR}, $\Hbf=\flecha{\mathcal{G}_p}\Rbf$. 
The fact that $\flecha{\mathcal{G}_p}^*=\flecha{\mathcal{G}_p}^{-1}$, imply the relation $\Rbf=\flecha{\mathcal{G}_p}^*\Hbf$. 
It follows that $\langle \Hbf, \Rbf\rangle=2\mathcal{G}_p^*(\Hbf)$ and $\langle \Hbf, \Rbf\rangle=2\mathcal{G}_p(\Rbf)$, 
which together with the equality $\langle \Hbf, \Rbf\rangle=1-N^2/4\Delta$ (of Proposition\,\ref{th:<R,H>=1-N^2/4D}) lead to 
the relations 
\[
 2\mathcal{G}_p^*(\Hbf)=1-N^2/4\Delta \qquad \mbox{and} \qquad  2\mathcal{G}_p(\Rbf)=1-N^2/4\Delta\,.\qquad\quad
\]

In a principal basis, the quadratic forms $\mathcal{G}_p$ and $\mathcal{G}_p^*$ are diagonal. Being $\K_1$, $\K_2$ 
and $\K_1^{-1}$, $\K_2^{-1}$ their respective sets of eigenvalues, equalities \eqref{eq:Gp(H)=1-N2/4D} are written as 
\[
h_1^2/\K_1+h_2^2/\K_2=1-N^2/4\Delta\,, \qquad\mbox{and}\qquad \K_1\rho_1^2+\K_2\rho_2^2=1-N^2/4\Delta\,. 
\]
The type of the conic - ellipse or hyperbola - is obtained from \eqref{eq:Delta=K1K1-K=K1+K2} and \eqref{eq:Delta>0-elliptic}. 
\end{proof}
%%%%%%<<<<<<<<<<<<<<<<<<<<<<<<<<<<<<<<<<<<<<<<<<<<<<<<<<<<<<<<<<<<<<<<<<<<<<<<<<<<<<<<<<<<<<<<<<<<<<<<< P R O O F %%

\begin{figure}[ht] %< préférence de placement h , t , b or p >
\centering
\includegraphics[scale=0.3]{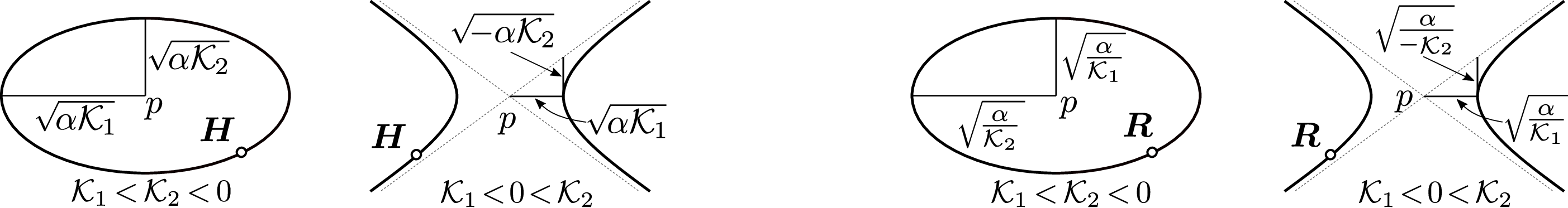}
\caption{\small The locus of $\Hbf$ and $\Rbf$ on the conics $\mathcal{H}$ and $\mathcal{R}$ 
given by \eqref{equation:MC-vector}. 
Here $\a:=1-N^2/4\Delta$.}
\label{fig:conics-of-H-R}
\end{figure}

%%%%%%%%%%%%%%%%%%%%%%%%%%%%%%%%%%%%%%%%%%%%%%%%%%%%%%%%%%%%%%%%%%%%%%%%%%%%%%%%%%%%%%%%%%%%%%%%%%%%
%%%%%%%%%%%%%%%%%%%%%%%%%%%%%%%%%%%%%%%%  THEOREM  %%%%%%%%%%%%%%%%%%%%%%%%%%%%%%%%%%%%%%%%%%%%%%%%%
%%%%%%%%%%%%%%%%%%%%%%%%%%%%%%%%%%%%%%%%%%%%%%%%%%%%%%%%%%%%%%%%%%%%%%%%%%%%%%%%%%%%%%%%%%%%%%%%%%%%
\begin{theorem}\label{th:hyperbolods-R-and-H}
If a surface $M$ in $\R^5$ has invariants $\K_1$, $\K_2$, $\K_3$ $(\K_1\K_2\K_3\neq 0)$ at $p$, then
$\Hbf$ and $\Rbf$ lie on the respective %two-sheet 
hyperboloids $\mathcal{H}$ and $\mathcal{R}$ of $N_pM$ 
{\rm (Fig.\,\ref{fig:hyperboloids-R-H})} of equations $2\mathcal{G}_p^*(q)=1\,$ and $\,2\mathcal{G}_p(q)=1$, 
which in a principal basis have the respective forms 
\begin{equation}\label{eq:hyperboloids-H-R}
(\mathcal{H}):\,\,\, \frac{q_1^2}{\K_1}+\frac{q_2^2}{\K_2}+\frac{q_3^2}{\K_3}=1 \quad \mbox{and}
 \quad (\mathcal{R}):\,\,\, \K_1q_1^2+\K_2q_2^2+\K_3q_3^2=1\,.
\end{equation}
\end{theorem}
%%%%%% P R O O F >>>>>>>>>>>>>>>>>>>>>>>>>>>>>>>>>>>>>>>>>>>>>>>>>>>>>>>>>>>>>>>>>>>>>>>>>>>>>>>>>>>>>>>>>>>>>>>>>>
\begin{proof}
The proof that $\Hbf$ and $\Rbf$ satisfy the respective relations %$2\mathcal{G}_p^*(\Hbf)=1\,$ and $\,2\mathcal{G}_p(\Rbf)=1$, 
\begin{equation}\label{eq:G(H)=1}
2\mathcal{G}^*_p(\Hbf)=1 \qquad \mbox{and} \qquad 2\mathcal{G}_p(\Rbf)=1\,,
\end{equation}
is similar to the proof of relations \eqref{eq:Gp(H)=1-N2/4D}, 
but one uses the equality $\langle \Hbf, \Rbf\rangle=1$. 

In a principal basis, $\mathcal{G}_p^*$ and $\mathcal{G}_p$ are diagonal. Being $\K_1^{-1}$, $\K_2^{-1}$, $\K_3^{-1}$
and $\K_1$, $\K_2$, $\K_3$ their respective sets of eigenvalues, 
equalities \eqref{eq:G(H)=1} are written as 
\begin{equation}\label{eq:H&R_K_1,k_2,k_3}
h_1^2/\K_1+h_2^2/\K_2+h_3^2/\K_3=1 \qquad\mbox{and}\qquad  \K_1\rho_1^2+\K_2\rho_2^2+\K_3\rho_3^2=1\,.
\end{equation}
By Proposition\,\ref{prop:K1<K2<0<K3} below, both $\mathcal{H}$ and $\mathcal{R}$ are two-sheet hyperboloids of $N_pM$. 
\end{proof}
%%%%%%<<<<<<<<<<<<<<<<<<<<<<<<<<<<<<<<<<<<<<<<<<<<<<<<<<<<<<<<<<<<<<<<<<<<<<<<<<<<<<<<<<<<<<<<<<<<<<<<< P R O O F %%
%%<<<<<<<<<<<<<<<<<<<<<<<<<<<<<<<<<<<<<<  F I G U R E  <<<<<<<<<<<<<<<<<<<<<<<<<<<<<<<<<<<<<<<<<<<<<<<<<
\begin{figure}[ht] %< préférence de placement h , t , b or p >
\centering
\includegraphics[scale=0.6]{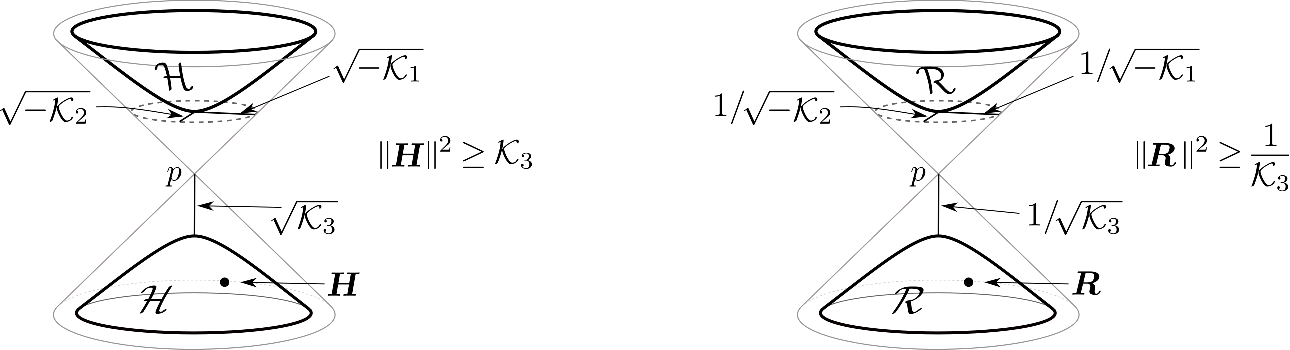}
\caption{\small The locus of $\Hbf$ and $\Rbf$ on the respective hyperboloids $\mathcal{H}$ and $\mathcal{R}$ 
given by \eqref{eq:hyperboloids-H-R}.}
\label{fig:hyperboloids-R-H}
\end{figure}

The hyperboloids of Theorem\,\ref{th:hyperbolods-R-and-H} (Fig.\ref{fig:hyperboloids-R-H}) are polar dual to each other in $N_pM$: 
%%%%%%%%%%%%%%%%%%%%%%%%%%%%%%%%%%%%%%%%%%%%%%%%%%%%%%%%%%%%%%%%%%%%%%%%%%%%%%%%%%%%%%%%%%%%%%%%%%%%
%%%%%%%%%%%%%%%%%%%%%%%%%%%%%%%%%%%%%%%%  THEOREM  %%%%%%%%%%%%%%%%%%%%%%%%%%%%%%%%%%%%%%%%%%%%%%%%%
%%%%%%%%%%%%%%%%%%%%%%%%%%%%%%%%%%%%%%%%%%%%%%%%%%%%%%%%%%%%%%%%%%%%%%%%%%%%%%%%%%%%%%%%%%%%%%%%%%%%
\begin{theorem}\label{th:polar-duality-of-H-and-R} 
Under the conditions of Theorem\,{\rm \ref{th:hyperbolods-R-and-H}}, 
the hyperboloids $\mathcal{H}$ and $\mathcal{R}$ {\rm (where $\Hbf$ and $\Rbf$ lie respectively)} are polar dual 
to each other in $N_pM$. Moreover, $\Rbf$ is the pole of the tangent plane $T_{\Hbf}\mathcal{H}$, 
and $\Hbf$ is the pole of the tangent plane $T_{\Rbf}\mathcal{R}$. 
\end{theorem}
%%%%%% P R O O F >>>>>>>>>>>>>>>>>>>>>>>>>>>>>>>>>>>>>>>>>>>>>>>>>>>>>>>>>>>>>>>>>>>>>>>>>>>>>>>>>>>>>>>>>>>>>>>>>>
\begin{proof}
Take a principal basis of $N_pM$ and take a point $\Hbf=(h_1,h_2,h_3)$ of $\mathcal{H}$. 
We shall prove that $\Rbf=(h_1/\K_1, h_2/\K_2, h_3/\K_3)$ is the pole 
of the tangent plane to $\mathcal{H}$ at $\Hbf$. 

The equation of $\mathcal{H}$ is written as $F(q)=1$, where $F(q):=q_1^2/\K_1+q_2^2/\K_2+q_3^2/\K_3$,  
and the equation of the tangent plane to $\mathcal{H}$ at $\Hbf$ is $d_{\Hbf}F(q-\Hbf)=0$, that is 
\[\frac{h_1}{\K_1}(q_1-h_1)+\frac{h_2}{\K_2}(q_2-h_2)+\frac{h_3}{\K_3}(q_3-h_3)=0\,.\]

Using that $F(\Hbf)=1$, the last equation can be rewritten as 
\[\frac{h_1}{\K_1}q_1+\frac{h_2}{\K_2}q_2+\frac{h_3}{\K_3}q_3=1\,.\]

This means that the pole of $T_{\Hbf}\mathcal{H}$ is the point $(h_1/\K_1, h_2/\K_2, h_3/\K_3)=\Rbf$.

One proves similarly that $\Hbf$ is the pole of the tangent plane to $\mathcal{R}$ at $\Rbf$. 
\end{proof}
%%%%%%<<<<<<<<<<<<<<<<<<<<<<<<<<<<<<<<<<<<<<<<<<<<<<<<<<<<<<<<<<<<<<<<<<<<<<<<<<<<<<<<<<<<<<<<<<<<<<<<< P R O O F %%

%%%%%%%%%%%%%%%%%%%%%%%%%%%%%%%%%%%%%%%%%%%%%%%%%%%%%%%%%%%%%%%%%%%%%%%%%%%%%%%%%%%%%%%%%%%%%%%%%%%%%%%%%%%%%%%%%%%%%%%%%
\subsection{\textbf{The Local Caustic $\mathcal{C}_p$ as Level Set of the Gauss Quadratic Form $\mathcal{G}_p$}}
%%%%%%%%%%%%%%%%%%%%%%%%%%%%%%%%%%%%%%%%%%%%%%%%%%%%%%%%%%%%%%%%%%%%%%%%%%%%%%%%%%%%%%%%%%%%%%%%%%%%%%%%%%%%%%%%%%%%%%%%%
We give the equation of the local caustic in terms of the Gauss quadratic form $\mathcal{G}_p$.

%%%%%%%%%%%%%%%%%%%%%%%%%%%%%%%%%%%%%%%%%%%%%%%%%%%%%%%%%%%%%%%%%%%%%%%%%%%%%%%%%%%%%%%%%%%%%%%%%%%%%%%%%%%%%%%%%%%%
%%%%%%%%%%%%%%%%%%%%%%%%%%%%%%%%%%%%%%%%%%%%    THEOREM    %%%%%%%%%%%%%%%%%%%%%%%%%%%%%%%%%%%%%%%%%%%%%%%%%%%%%%%%%
%%%%%%%%%%%%%%%%%%%%%%%%%%%%%%%%%%%%%%%%%%%%%%%%%%%%%%%%%%%%%%%%%%%%%%%%%%%%%%%%%%%%%%%%%%%%%%%%%%%%%%%%%%%%%%%%%%%%
\begin{theorem}\label{th:equation_local_caustic}
Let $\Rbf$ be the centre of the local caustic $\mathcal{C}_p$ at a point $p$ of a surface in $\R^4$ $($or in $\R^5$$)$ 
with $\Delta\neq 0$. The conic $\mathcal{C}_p$ 
$($or the cone $\mathcal{C}_p$$)$ is given as a level set of $\mathcal{G}_p:$
\begin{equation}\label{eq:Gp(q-R)=-N2/4D}
  2\mathcal{G}_p(q-\Rbf)=-\frac{N^2}{4\Delta}   \qquad \mbox{ equation of the conic $\mathcal{C}_p$ for $M$ in $\R^4$}\,,
\end{equation}
\begin{equation}
  \mathcal{G}_p(q-\Rbf)=0  \qquad\qquad \mbox{\,equation of the cone $\mathcal{C}_p$ for $M$ in $\R^5$}\,.
\end{equation}
\end{theorem}
%%%%%% P R O O F >>>>>>>>>>>>>>>>>>>>>>>>>>>>>>>>>>>>>>>>>>>>>>>>>>>>>>>>>>>>>>>>>>>>>>>>>>>>>>>>>>>>>>>>>>>>>>>>>>
\begin{proof}
Opening the parenthesis of equation \eqref{eq:focal-quadric}, using expression \eqref{eq:CQF(q)} for $\mathcal{G}_p(q)$ 
and that $(\Abf+\Cbf)/2=\Hbf$, the equation of the local caustic, for $M$ in $\R^4$ or in $\R^5$, is written as 
\[2\mathcal{G}_p(q)-2\langle\Hbf,q\rangle+1=0\,.\]
Using that $\Hbf=\overrightarrow{\mathcal{G}_p}\Rbf$ (Prop.\,\ref{proposition:H=QR}) and that the linear operator $\overrightarrow{\mathcal{G}_p}$ 
is symmetric, we get 
\begin{equation}\label{eq:equation-caustic-inter}
2\mathcal{G}_p(q)-2\left\langle\overrightarrow{\mathcal{G}_p}\,q,\Rbf\right\rangle+1=0\,.
\end{equation}

For $M$ in $\R^4$ we replace in \eqref{eq:equation-caustic-inter} the equality $1=2\mathcal{G}_p(\Rbf)+N^2/4\Delta$, 
from \eqref{eq:Gp(H)=1-N2/4D}, to get 
\[2\left(\mathcal{G}_p(q)-2\frac{1}{2}\left\langle\overrightarrow{\mathcal{G}_p}\,q,\Rbf\right\rangle+\mathcal{G}_p(\Rbf)\right)+\frac{N^2}{4\Delta}=0\,,\] 
that is, the equation of the local caustic at $p$ is $2\mathcal{G}_p(q-\Rbf)=-N^2/4\Delta$.
\medskip

For $M$ in $\R^5$ we use in \eqref{eq:equation-caustic-inter} that $1=2\mathcal{G}_p(\Rbf)$, from \eqref{eq:G(H)=1}, to get
\,$\mathcal{G}_p(q-\Rbf)=0$. 
\end{proof}
%%%%%%<<<<<<<<<<<<<<<<<<<<<<<<<<<<<<<<<<<<<<<<<<<<<<<<<<<<<<<<<<<<<<<<<<<<<<<<<<<<<<<<<<<<<<<<<<<<<<<<< P R O O F %%

%%%%%%%%%%%%%%%%%%%%%%%%%%%%%%%%%%%%%%%%%%%%%%%%%%%%%%%%%%%%%%%%%%%%%%%%%%%%%%%%%%%%%%%%%%%%%%%%%%%%%%%%%%%%%%%%%%%%
%%%%%%%%%%%%%%%%%%%%%%%%%%%%%%%%%%%%%%%%%%%%  PROPOSITION  %%%%%%%%%%%%%%%%%%%%%%%%%%%%%%%%%%%%%%%%%%%%%%%%%%%%%%%%%
%%%%%%%%%%%%%%%%%%%%%%%%%%%%%%%%%%%%%%%%%%%%%%%%%%%%%%%%%%%%%%%%%%%%%%%%%%%%%%%%%%%%%%%%%%%%%%%%%%%%%%%%%%%%%%%%%%%%
% \begin{proposition}\label{prop:semi-axes-Cp}
% At a point $p$ of a surface in $\R^4$ with $\Delta\neq 0$, the principal semi-axes of the conic $\mathcal{C}_p$ 
% {\rm (the local caustic)} are\, $\frac{|N|}{2\left|\K_1\sqrt{|\K_2|}\right|}$ and $\frac{|N|}{2\left|\K_2\sqrt{|\K_1|}\right|}$.
% \end{proposition}
% %%%%%% P R O O F >>>>>>>>>>>>>>>>>>>>>>>>>>>>>>>>>>>>>>>>>>>>>>>>>>>>>>>>>>>>>>>>>>>>>>>>>>>>>>>>>>>>>>>>>>>>>>>>>>
% \begin{proof}
% We write \eqref{eq:Gp(q-R)=-N2/4D} in a principal basis: $\K_1(q_1-\rho_1)^2+\K_2(q_2-\rho_2)^2=-N^2/4\Delta$. 
% \end{proof}
%%%%%%<<<<<<<<<<<<<<<<<<<<<<<<<<<<<<<<<<<<<<<<<<<<<<<<<<<<<<<<<<<<<<<<<<<<<<<<<<<<<<<<<<<<<<<<<<<<<<<<< P R O O F %%

%%%%%%%%%%%%%%%%%%%%%%%%%%%%%%%%%%%%%%%%%%%%%%%%%%%%%%%%%%%%%%%%%%%%%%%%%%%%%%%%%%%%%%%%%%%%%%%%%%%%%%%%%%%%%%%%%%%%%%%%%
\subsection{\textbf{Local Caustic related to Degenerate and Binormal Vectors}}
%%%%%%%%%%%%%%%%%%%%%%%%%%%%%%%%%%%%%%%%%%%%%%%%%%%%%%%%%%%%%%%%%%%%%%%%%%%%%%%%%%%%%%%%%%%%%%%%%%%%%%%%%%%%%%%%%%%%%%%%%

%%%%%%%%%%%%%%%%%%%%%%%%%%%%%%%%%%%%%%%%%%%%%%%%%%%%%%%%%%%%%%%%%%%%%%%%%%%%%%%%%%%%%%%%%%%%%%%%%%%%%%%%%%%%%%%%%%%
%%%%%%%%%%%%%%%%%%%%%%%%%%%%%%%%%%%%%%%%%%%%  DEFINITION  %%%%%%%%%%%%%%%%%%%%%%%%%%%%%%%%%%%%%%%%%%%%%%%%%%%%%%%%%
\noindent
\textbf{\textit{\small Degenerate Vectors, Binormal Vectors and Asymptotic Directions}}. 
The eigenvalue $\mu=0$ of $\left[d^\perp_{p,\nbf}\G\right]$ corresponds to a focal centre at $\RP^{\ell-1}_\infty$ in the 
direction of $\nbf$ and to a higher contact of $M$ with the hyperplane normal to $\nbf$ at $p$. 
In this case, the multiples of $\nbf$ ($\a\nbf$ with $\a\neq 0$)  are called \textit{degenerate normal vectors}. % (a hypersphere of infinite radius). 

\begin{proposition}\label{prop:Cone_degenerate_normals}
	Given $p\in M$ in $\R^{2+\ell}$, the set $\Sigma_p$ of degenerate normal vectors of $M$ at $p$ is 
	the cone in $N_pM$ along which the Gauss quadratic form $\mathcal{G}_p$ vanishes:
	\begin{equation}
	\quad \mathcal{G}_p(q)=0 \qquad  \mbox{\rm (equation of $\Sigma_p$)}
	\end{equation}
\end{proposition}

\begin{proof}
	The product of the eigenvalues of $\left[d^\perp_{p,\nbf}\G\right]$ is equal to $K^{\nbf}=2\,\mathcal{G}_p(\nbf)$. 
\end{proof}

%%%%%%%%%%%%%%%%%%%%%%%%%%%%%%%%%%%%%%%%%%%%%%%%%%%%%%%%%%%%%%%%%%%%%%%%%%%%%%%%%%%%%%%%%%%%%%%%%%%%%%%%%%%%%%%%%%%%
%%%%%%%%%%%%%%%%%%%%%%%%%%%%%%%%%%%%%%%%%%%%   COROLLARY   %%%%%%%%%%%%%%%%%%%%%%%%%%%%%%%%%%%%%%%%%%%%%%%%%%%%%%%%%
%%%%%%%%%%%%%%%%%%%%%%%%%%%%%%%%%%%%%%%%%%%%%%%%%%%%%%%%%%%%%%%%%%%%%%%%%%%%%%%%%%%%%%%%%%%%%%%%%%%%%%%%%%%%%%%%%%%%
\begin{corollary}
	Given $p\in M$ in $\R^5$, the local caustic $\mathcal{C}_p$ is the cone parallel to the cone $\Sigma_p$ of 
	degenerate normal vectors, by translation along the vector $\Rbf:$ $\mathcal{C}_p=\Sigma_p+\Rbf$. 
\end{corollary}

\begin{proof}
	It follows from the equation of $\mathcal{C}_p$, \,$\mathcal{G}_p(q-\Rbf)=0$,\, in Theorem\,\ref{th:equation_local_caustic}.  
\end{proof}

If $p\in M\subset\R^4$ has a focal centre at $\RP^1_\infty$, the contact direction in $T_pM$ is said to be \textit{asymptotic} 
and the degenerate vectors are also called \textit{binormal vectors}. 

If $p\in M_3\subset\R^5$ the focal centres at $\RP^2_\infty$ form an ellipse ($\mathcal{C}_p\cap\RP^2_\infty$). 
Thus there is ``an ellipse of tangent hyperplanes'' with greater contact than usual - each direction in $T_pM$ 
is the contact direction of $M$ with one such hyperplane. Moving along the ellipse $\mathcal{C}_p\cap\RP^2_\infty$ 
one can find isolated focal centres whose associated tangent hyperplane is over-osculating - their normal 
(degenerate) vectors are called \textit{binormal} and the higher contact direction is called 
 \textit{asymptotic}. The local quadratic map $\f_p$ cannot detect the asymptotic directions at $p$ - 
 one needs the terms of degree $3$. %, $4$, etc., depending on the considered order of contact. 
In \protect\cite{M-RM-R2003}, it was shown that \textit{if $p\in M_3$, there is at least $1$ and at most $5$ 
 asymptotic directions}. The local configurations of their integral curves, called \textit{asymptotic curves}, 
 were studied in \protect\cite{RM-R-T}. 
\medskip

%%%%%%%%%%%%%%%%%%%%%%%%%%%%%%%%%%%%%%%%%%%%%%%%%%%%%%%%%%%%%%%%%%%%%%%%%%%%%%%%%%%%%%%%%%%%%%%%%%%%%%%%%%%%%%%%%%%%
%%%%%%%%%%%%%%%%%%%%%%%%%%%%%%%%%%%%%%%%%%%%  PROPOSITION  %%%%%%%%%%%%%%%%%%%%%%%%%%%%%%%%%%%%%%%%%%%%%%%%%%%%%%%%%
%%%%%%%%%%%%%%%%%%%%%%%%%%%%%%%%%%%%%%%%%%%%%%%%%%%%%%%%%%%%%%%%%%%%%%%%%%%%%%%%%%%%%%%%%%%%%%%%%%%%%%%%%%%%%%%%%%%%
% \begin{theorem}
% Every line normal to $M$ at $p$ cuts the local caustic $\mathcal{C}_p$ at two points 
% whose distances to $p$ are the inverses of the eigenvalues of the matrix $[d_{p,\nbf}^\perp\G]$ and whose  
% respective principal directions are two orthogonal lines on the tangent plane $T_pM$. \\
% {\rm (So to each normal line there correspond two curvature radii with their respective orthogonal 
% ``principal'' directions on the tangent plane.)}
% \end{theorem}

\noindent
\textbf{\textit{\small Caustic of the Gauss Map}}. 
In the setting of Lagrangian singularities, 
the Gauss map of a submanifold $M$ in $\R^n$ defines a Lagrangian map from the manifold of all oriented 
lines normal to $M$, into the unit sphere $\sph^{n-1}$ - the image of an oriented line is 
its unit orienting vector. 
For a surface $M$ in $\R^4$ the caustic of its Gauss map, called \textit{binormal surface}, is the surface in $\sph^3$ formed by the 
unit binormal vectors of $M$. Its local singularities are those of caustics in $3$-space: cuspidal edges, swallowtails and 
$D_4$ singularities. They were described by Deibrelbis in \cite{Dreibelbis-BiNSurf} using the family of ``hight functions''. 
For $M$ in $\R^5$ the caustic of its Gauss map is the $3$-variety in $\sph^4$  
consisting of the unit degenerate normal vectors of $M$. 
For a locally parametrised surface $M$, $\g:\R^2\to M\subset\R^n$ its family of ``hight functions'' 
$\{h_v(x)=\langle v,\g(x)\rangle:v\in\sph^{n-1}\}$, widely used in \protect\cite{IRFRT}, is a generating family 
of the (Lagrangian) Gauss map of $M$ - whence its efficacity! 
\medskip

\noindent
\textbf{\small Asymptotes of the local caustic and asymptotic vectors in $T_pM$}.
At a hyperbolic point $p\in M$ in $\R^4$ the local caustic $\mathcal{C}_p$ is, by Th.\,\ref{th:focal_quadric-indicatrix}, 
a hyperbola (see Th.\,\ref{th:local-caustic-R4}). The two points of $\mathcal{C}_p$ at $\RP^1_\infty$, 
say $q_a$, $q_b$, are focal centres associated to two asymptotic vectors $\ubf_a$, $\ubf_b$. Thus, their polar 
lines $\ell_{q_a}$, $\ell_{q_b}$ contain $p$ and, by Th.\,\ref{th:focal_quadric-indicatrix}, are tangent to $\E_p$ 
at the two points $a=\E_p(\pm\ubf_a)$, $b=\E_p(\pm\ubf_b)$. 
We get the following proposition (Fig.\ref{fig:asymptotic-tangency-pts}):

\begin{figure}[ht] %< préférence de placement h , t , b or p >
\centering
\includegraphics[scale=0.45]{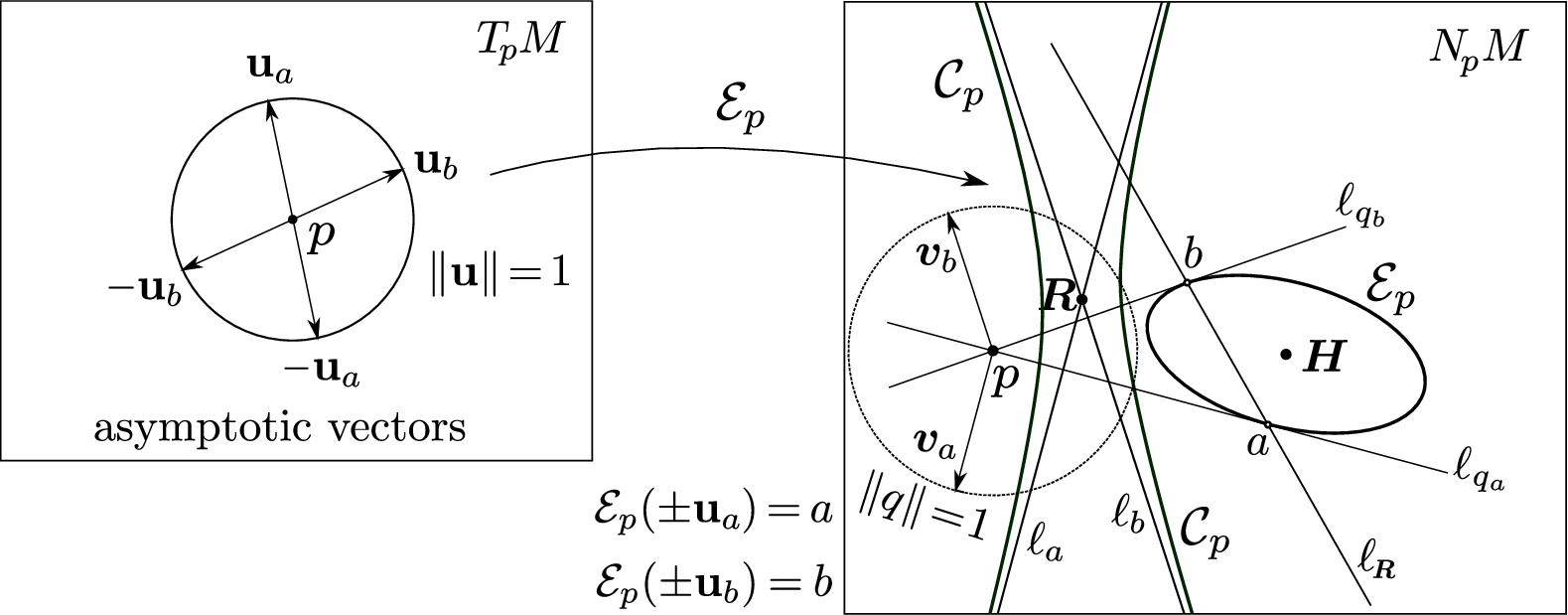}
\caption{\small The asymptotes $\ell_a$, $\ell_b$ of the local caustic $\mathcal{C}_p$ directed by the binormal vectors $\boldsymbol{v}_a$, $\boldsymbol{v}_b$.}
\label{fig:asymptotic-tangency-pts}
\end{figure}

\begin{proposition}
At a hyperbolic point $p\in M$ in $\R^4$, the points $a$, $b$ in $\E_p$ such that the tangent line to 
$\E_p$ contains $p$ are the poles of the asymptotes of the hyperbola $\mathcal{C}_p$ {\rm (noted $\ell_a$, $\ell_b$)}. 
The centre $\Rbf$ of $\mathcal{C}_p$ is the pole of the line that joins $a$ and $b$ {\rm (noted $\ell_{\Rbf}$)}. 
\end{proposition}

%%%%%%%%%%%%%%%%%%%%%%%%%%%%%%%%%%%%%%%%%%%%%%%%%%%%%%%%%%%%%%%%%%%%%%%%%%%%%%%%%%%%%%%%%%%%%%%%%%%%%%%%%%%%%%%%%%%%%%%%%
%%%%%%%%%%%%%%%%%%%%%%%%%%%%%%%%%%%%%%%%%%%%%%%%%%%%%%%%%%%%%%%%%%%%%%%%%%%%%%%%%%%%%%%%%%%%%%%%%%%%%%%%%%%%%%%%%%%%%%%%%
\section{Pseudo-Euclidean Geometry of Quadratic Forms on $\R^2$ and Local Invariants of Surfaces in 
$\R^4$ and in $\R^5$}\label{Section-Pseudo-Euclidean}
%%%%%%%%%%%%%%%%%%%%%%%%%%%%%%%%%%%%%%%%%%%%%%%%%%%%%%%%%%%%%%%%%%%%%%%%%%%%%%%%%%%%%%%%%%%%%%%%%%%%%%%%%%%%%%%%%%%%%%%%%
%%%%%%%%%%%%%%%%%%%%%%%%%%%%%%%%%%%%%%%%%%%%%%%%%%%%%%%%%%%%%%%%%%%%%%%%%%%%%%%%%%%%%%%%%%%%%%%%%%%%%%%%%%%%%%%%%%%%%%%%%

The local geometry and local invariants of a surface in $\R^4$ or in $\R^5$ are determined by the geometric features 
of the parallelogram (or parallelepiped) formed by $Q_1$, $Q_2$ (or $Q_1$, $Q_2$, $Q_3$) in the pseudo-Euclidean 
vector space of quadratic forms. 

%%%%%%%%%%%%%%%%%%%%%%%%%%%%%%%%%%%%%%%%%%%%%%%%%%%%%%%%%%%%%%%%%%%%%%%%%%%%%%%%%%%%%%%%%%%%%%%%%%%%%%%%%%%%%%%%
\subsection{\textbf{Pseudo-Euclidean Space of Quadratic Forms on $\R^2$}}\label{sect-pseudi-euclidian_space-quad}
%%%%%%%%%%%%%%%%%%%%%%%%%%%%%%%%%%%%%%%%%%%%%%%%%%%%%%%%%%%%%%%%%%%%%%%%%%%%%%%%%%%%%%%%%%%%%%%%%%%%%%%%%%%%%%%%

The vector space of quadratic forms on the plane $\R^2$,  
$\{(a,b,c)\}=\R^3$, has a natural pseudo-Euclidean structure given 
by the equation $ac-b^2=0$ of the cone of degenerate forms. Namely, the pseudo-scalar product 
of two forms $Q_1, Q_2$ is defined by 
\begin{equation}
  \langle Q_1,Q_2 \rangle_\psi=\frac{a_1c_2+a_2c_1}{2}-b_1b_2\,.
\end{equation}

\noindent
\textit{\textbf{\small Pseudo-Orthogonality}}. 
Two quadratic forms $Q_1$, $Q_2$ on $\R^2$ are said to be 
\textit{pseudo-orthogonal} if their pseudo-scalar product is zero, %and only if 
$\langle Q_1,Q_2 \rangle_\psi=0$. 
\medskip

In this pseudo Euclidean space, the square of the pseudo-norm of a quadratic form $Q$, 
$\|Q\|_\psi^2:=\langle Q,Q \rangle_\psi$,  %=ac-b^2$, 
can take any real value. Namely, we have the evident 

\begin{proposition}\label{prop:sign-||Q||}
Given a non zero quadratic form $Q$ on $\R^2$, the following holds$:$
\medskip

\noindent
\textit{$\|Q\|_\psi^2$ is positive iff $Q$ is in the interior of $\mathcal{C}$} 
(i.e. in a convex component of $\R^3\smallsetminus\mathcal{C}$); 
\medskip

\noindent
\textit{$\|Q\|_\psi^2$ is negative iff $Q$ is in the exterior of $\mathcal{C}$} 
(i.e. in the non convex part of $\R^3\smallsetminus\mathcal{C}$); 
\medskip

\noindent
\textit{$\|Q\|_\psi^2$ is zero iff $Q$ lies on $\mathcal{C}$.}
\end{proposition}

\noindent
\textit{\textbf{\small Poisson bracket}}. 
The \textit{Poisson bracket} of two functions on the  (symplectic) plane is defined as 
\begin{equation}\label{eq:Poisson-btacket}
  \{f,g\}=\textstyle{\frac{1}{2}}(f_sg_t-f_tg_s)\,.
\end{equation}

\noindent
\textit{\textbf{\small Lie Algebra $\mathfrak{sl}(2,\R)$}}. 
The Poisson bracket of two quadratic forms is again a quadratic form. 
With this bracket the vector space $\R^3$ of quadratic forms is a Lie algebra 
isomorphic to $\mathfrak{sl}(2,\R)$: the \textit{Lie bracket} of two vectors $Q_1, Q_2$ 
is the Poisson bracket of their associated quadratic functions $Q_1$, $Q_2$. 

\noindent
We use the same notation for a quadratic form and for its representative vector in $\R^3$. 
\medskip

Computing the Poisson bracket of two quadratic forms $Q_1=\frac{1}{2}(a_1s^2+2b_1st+c_1t^2)$ and 
$Q_2=\frac{1}{2}(a_2s^2+2b_2st+c_2t^2)$, using \eqref{eq:Poisson-btacket}, one gets
\begin{equation}\label{eq:Poisson-bracket-form}
  \{Q_1,Q_2\}=\frac{1}{2}\left((a_1b_2-a_2b_1)s^2+2\left(\frac{a_1c_2-a_2c_1}{2}\right)st+(b_1c_2-b_2c_1)t^2\right)\,.
\end{equation}

%%%%%%%%%%%%%%%%%%%%%%%%%%%%%%%%%%%%%%%%%%%%%%%%%%%%%%%%%%%%%%%%%%%%%%%%%%%%%%%%%%%%%%%%%%%%%%%%%%%%%%%%%%%%%%%%
\subsection{\textbf{Poisson Bracket {\sc\textbf{vs}}  Vector Product}}\label{subsect-PoissonBracket-VectorProduct}
%%%%%%%%%%%%%%%%%%%%%%%%%%%%%%%%%%%%%%%%%%%%%%%%%%%%%%%%%%%%%%%%%%%%%%%%%%%%%%%%%%%%%%%%%%%%%%%%%%%%%%%%%%%%%%%%
In Euclidean vector space $\R^3$, the vector product $v_1\times v_2$ is orthogonal to $v_1$ and $v_2$, and 
its squared norm is equal to the squared area of the parallelogram formed by $v_1$ and $v_2$, and is also the determinant of the Gram matrix of $v_1$ and $v_2$ for the inner product $\langle\,,\,\rangle$: 
$\|v_1\times v_2\|^2=\langle v_1,v_1\rangle\langle v_2,v_2\rangle-\langle v_1,v_2\rangle^2$. Moreover, the squared volume 
of the parallelepiped formed by three vectors $v_1,v_2,v_3$, is the determinant of the Gram matrix of $v_1$, $v_2$, $v_3$ (for $\langle\,,\,\rangle$):   
$\langle v_1, v_2\times v_3\rangle^2=\det (\langle v_i,v_j\rangle)$ 
\smallskip

%%%%%%%%%%%%%%%%%%%%%%%%%%%%%%%%%%% THEOREM BRACKET-ORTHOGONALITY %%%%%%%%%%%%%%%%%%%%%%%%%%%%%%%%%%%%%%%%%%%

The pseudo-Euclidean structure on the space $\R^3$ of quadratic forms 
is intimately related to its Lie algebra structure. 
%The following proposition shows that 
The Poisson bracket in the pseudo-Euclidean space plays  
a r\^ole analogue to that of the vector product in Euclidean space. Namely,  
%%%%%%%%%%%%%%%%%%%%%%%%%%%%%%%%%%%%%%%%%%%%%%%%%%%%%%%%%%%%%%%%%%%%%%%%%%%%%%%%%%%%%%%%%%%%%%%%%%%%%%%%%%%%%%%%%%%%
%%%%%%%%%%%%%%%%%%%%%%%%%%%%%%%       PROPOSITION       %%%%%%%%%%%%%%%%%%%%%%%%%%%%%%%%%%%%%%%%%%%%%%%%%%%%%%%%%%%%
%%%%%%%%%%%%%%%%%%%%%%%%%%%%%%%%%%%%%%%%%%%%%%%%%%%%%%%%%%%%%%%%%%%%%%%%%%%%%%%%%%%%%%%%%%%%%%%%%%%%%%%%%%%%%%%%%%%%
\begin{proposition}\label{th:bracket-orthogonality}
$(a)$ {\rm \protect\cite{Arnold_Lobachevsky_altitudes} }
The Poisson bracket $\{Q_1, Q_2\}$ of two quadratic forms is a quadratic form which is 
  pseudo-orthogonal to the plane generated by $Q_1$ and $Q_2:$
\begin{equation}
  \langle\,\{ Q_1, Q_2\}, Q_1\,\rangle_\psi=0 \ \ \ \mbox{ and } \ \ \
  \langle\,\{Q_1,Q_2\}, Q_2\,\rangle_\psi=0\,.
\end{equation}

\noindent
$(b)$ The squared pseudo-norm of the Poisson bracket of two quadratic forms $Q_1$, $Q_2$ equals  
the squared pseudo-area of the parallelogram formed by $Q_1$ and $Q_2$, which is given by the determinant of 
the Gram matrix of $Q_1$, $Q_2$ for $\langle \,\, ,\,\rangle_\psi:$  
\[\|\{Q_1, Q_2\}\|_{\psi}^2=
\|Q_1\|_{\psi}^2\|Q_2\|_{\psi}^2-\langle Q_1,Q_2\rangle_{\psi}^2=
\det\begin{pmatrix}
  \langle Q_1,Q_1\rangle_\psi & \langle Q_1,Q_2\rangle_\psi \\ 
  \langle Q_2,Q_1\rangle_\psi & \langle Q_2,Q_2\rangle_\psi 
\end{pmatrix}\,.\] 

\noindent
$(c)$ The squared pseudo-volume of the parallelepiped formed by three quadratic forms $Q_1,Q_2,Q_3$, is the determinant of 
the Gram matrix of $Q_1$, $Q_2$, $Q_3$ for $\langle\,\,,\,\rangle_\psi:$
\[\langle\, Q_1, \{Q_2, Q_3\}\,\rangle_\psi^2=\det\begin{pmatrix}
  \langle Q_1,Q_1\rangle_\psi & \langle Q_1,Q_2\rangle_\psi & \langle Q_1,Q_3\rangle_\psi \\ 
  \langle Q_2,Q_1\rangle_\psi & \langle Q_2,Q_2\rangle_\psi & \langle Q_2,Q_3\rangle_\psi \\
  \langle Q_3,Q_1\rangle_\psi & \langle Q_3,Q_2\rangle_\psi & \langle Q_3,Q_3\rangle_\psi 
\end{pmatrix}\,.\] 
\end{proposition}
%%%%%% P R O O F >>>>>>>>>>>>>>>>>>>>>>>>>>>>>>>>>>>>>>>>>>>>>>>>>>>>>>>>>>>>>>>>>>>>>>>>>>>>>>>>>>>>>>>>>>>>>>>>>>
\begin{proof}
$(a)$ From \eqref{eq:Poisson-bracket-form} we get 
\[\langle\, Q_1, \{Q_1, Q_2\}\,\rangle_\psi=\frac{a_1(b_1c_2-b_2c_1)+(a_1b_2-a_2b_1)c_1}{2}-b_1\left(\frac{a_1c_2-a_2c_1}{2}\right)=0\,.\]
$(b)$ Using \eqref{eq:Poisson-bracket-form}, the direct computation of the squared pseudo-norm gives 
\begin{align*}
\|\{Q_1,Q_2\}\|_{\psi}^2 &= (a_1b_2-a_2b_1)(b_1c_2-b_2c_1)-(a_1c_2-a_2c_1)^2/4\\
                  &= (a_1c_2+a_2c_1)b_1b_2)b_1b_2-a_1c_1b_2^2-a_2c_2b_1^2-(a_1c_2+a_2c_1)^2/4+a_1c_1a_2c_2\,.
\end{align*}
Adding $b_1^2b_2^2-b_1^2b_2^2$ to the last expression and rearranging the terms one easily gets
\begin{align*}
\hspace{2cm}\|\{Q_1,Q_2\}\|_{\psi}^2  & =(a_1c_1-b_1^2)(a_2c_2-b_2^2)-\left( \frac{(a_1c_2+a_2c_1)}{2}-b_1b_2\right)^2\\
                &= \|Q_1\|_{\psi}^2\|Q_2\|_{\psi}^2-\langle Q_1,Q_2\rangle_{\psi}^2\,. %\hspace{5.3cm} \square
\end{align*}
$(c)$ Our proof is a direct and simple (but long and uninteresting) calculation. 
\end{proof} 
%%%%%%<<<<<<<<<<<<<<<<<<<<<<<<<<<<<<<<<<<<<<<<<<<<<<<<<<<<<<<<<<<<<<<<<<<<<<<<<<<<<<<<<<<<<<<<<<<<<<<<< P R O O F %%

%%%%%%%%%%%%%%%%%%%%%%%%%%%%%%%%%%%%%%%%%%%%%%%%%%%%%%%%%%%%%%%%%%%%%%%%%%%%%%%%%%%%%%%
\subsection{\textbf{From pseudo-orthogonality in $\R^3$ to polar duality in $\RP^2$}}
%%%%%%%%%%%%%%%%%%%%%%%%%%%%%%%%%%%%%%%%%%%%%%%%%%%%%%%%%%%%%%%%%%%%%%%%%%%%%%%%%%%%%%%
\noindent
\textit{\textbf{\small Pseudo-orthogonal plane}}. 
Given a vectorial line $\lbf$ in the space of quadratic forms and a vector $Q_0=(a_0,b_0,c_0)$ generating 
$\lbf$, the kernel of the linear form $\langle Q_0 \,,\,\cdot\,\rangle_\psi:\R^3\to\R$ is the plane which is 
\textit{pseudo-orthogonal} to the line $\lbf$ (Fig.\,\ref{fig:conjugacy-by-cone}): 
\[H_{\lbf}=\{Q\in\R^3:\langle\,Q_0\,,\,Q\,\rangle_\psi=0\}\,.\]
%%>>>>>>>>>>>>>>>>>>>>>>>>>>>>>>>>>>>>>>>>>>>>>>>>>>>>>>>>>>>>>>>>>>>>>>>>>>>>>>>>>>>>>>>>>>>>>>>>>>>>>>
%%<<<<<<<<<<<<<<<<<<<<<<<<<<<<<<<<<<<<<<  F I G U R E  <<<<<<<<<<<<<<<<<<<<<<<<<<<<<<<<<<<<<<<<<<<<<<<<<
\begin{figure}[ht]
\centerline{\psfig{figure=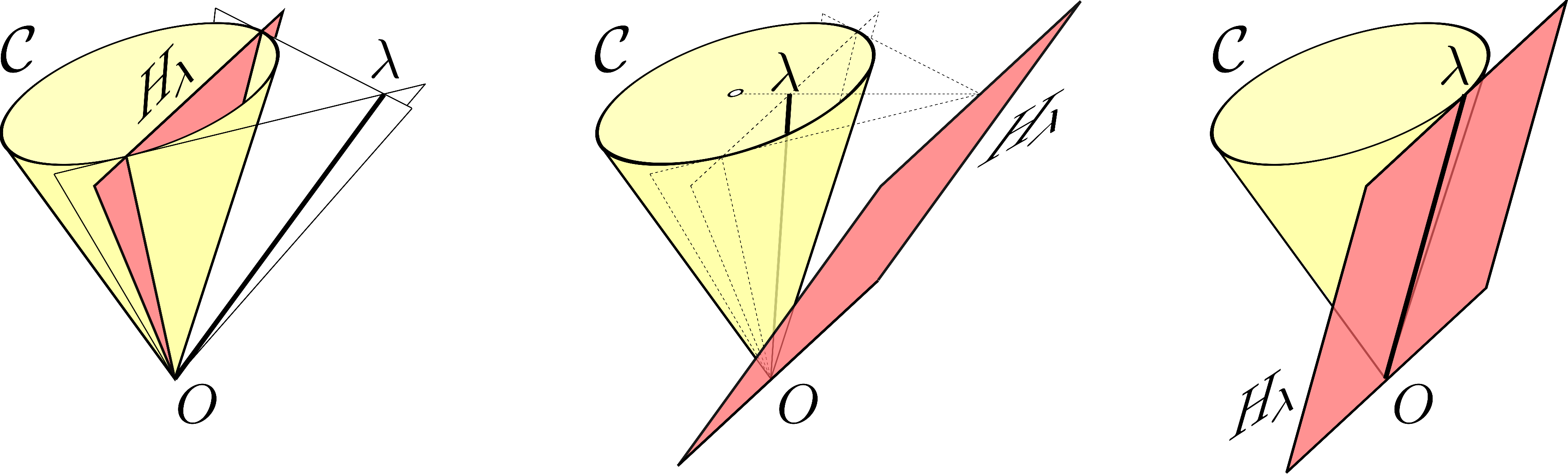,height=3.6cm}}
\caption{\small A vectorial line $\l$, its pseudo-orthogonal plane $H_\l$ and the cone $\langle Q,Q\rangle_\psi=0$.}
\label{fig:conjugacy-by-cone}
\end{figure}

% {\footnotesize
% \noindent
% {\bf Geometric construction of $H_\l$}. 
% Each plane $\Pi$ in $\R^3$ containing $\l$ intersects the cone $\mathcal{C}$ along two
% lines $\mu$, $\nu$, which may be real, coincident or complex conjugate. 
% In the three cases, there is a real line $\l'$ in $\Pi$ (called {\em harmonic conjugate} of $\l$) such
% that the cross ratio $[\mu, \nu, \l, \l']$ equals $-1$. {\em The plane $H_\l$ is the union of 
% all such harmonic conjugate lines $\l'$}, taken over all planes $\Pi$ containing $\l$ 
% (Fig.\ref{fig:conjugacy-by-cone}). 
% %In particular, if the line $\l$ lies on the cone $\hat{\mathcal{C}}$ then 
% %$H_\l$ is the plane tangent to $\hat{\mathcal{C}}$ along $\l$ \,--\, . 
% }\smallskip

The projective plane $\RP^2=(\R^3\smallsetminus 0)/ (\R\smallsetminus 0)$ is the space of
vectorial lines in $\R^3$. Its points are non-zero vectors in $\R^3$ considered up to a
non-zero constant factor - so we write $\RP^2=\P(\R^3)$. 
Under projectivisation, the vectorial planes in $\R^3$ become lines in $\RP^2$, 
the vectorial lines in $\R^3$ become points and our cone $\mathcal{C}$ becomes a
conic $\hat{\mathcal{C}}$ in $\RP^2$.

Projectivising the Lie algebra of quadratic forms on $\R^2$, one considers the
points of the projective plane $\RP^2=\P(\mathfrak{sl}(2,\R))$ as \guillemotleft forms\guillemotright 
(\protect\cite{Arnold_Lobachevsky_altitudes}): 
a \guillemotleft \textit{form}\guillemotright\, $\,\hat{Q}=[a:b:c]\,$ is a quadratic form 
$Q=\frac{1}{2}(as^2+2bst+ct^2)$ considered up to a nonzero constant factor. 

The conic $\hat{\mathcal{C}}$, obtained from the discriminant cone $\mathcal{C}$,
defines a polar duality (\S\,\ref{sect-polar_dual-subvariety}).
\medskip

%%%%%%%%%%%%%%%%%%%%%%%%%%%%%%%% DEFINITION %%%%%%%%%%%%%%%%%%%%%%%%%%%%%%%%%%%%%%%%%%%%%%%%%%%%%%%%%%
\noindent 
\textit{\textbf{\small Polar Duality}}. We can prove (using the cross-ratio invariance) 
that \textit{a point $P$ is the pole of a line $\ell_P$ in $\RP^2$ with respect to the conic 
$\hat{\mathcal{C}}=\P(\,{\mathcal{C}}\,)$ $($and  $\ell_P$ is the {\em polar} of $P$$)$ iff 
their respective vectorial line and vectorial plane in $\R^3$ are pseudo-orthogonal} (Fig.\,\ref{fig:circle-pole-polar}). 
%%>>>>>>>>>>>>>>>>>>>>>>>>>>>>>>>>>>>>>>>>>>>>>>>>>>>>>>>>>>>>>>>>>>>>>>>>>>>>>>>>>>>>>>>>>>>>>>>>>>>>>>
%%<<<<<<<<<<<<<<<<<<<<<<<<<<<<<<<<<<<<<<  F I G U R E  <<<<<<<<<<<<<<<<<<<<<<<<<<<<<<<<<<<<<<<<<<<<<<<<<
\begin{figure}[ht]
\centerline{\psfig{figure=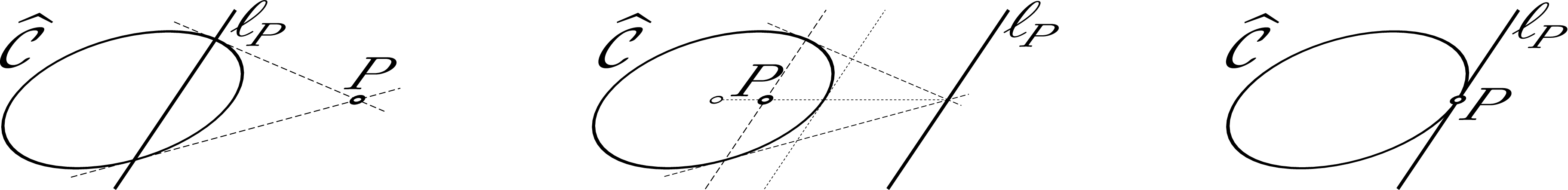,height=1.5cm}}
\caption{\small Duality \,pole$\,\leftrightarrow\,$polar\, ($P\,\leftrightarrow\,\ell_P$) 
defined by the conic $\hat{\mathcal{C}}=\P(\,{\mathcal{C}}\,)$. (Compare to Fig.\,\ref{fig:conjugacy-by-cone}.)}
\label{fig:circle-pole-polar}
\end{figure}

\noindent
\textit{\textbf{\small Conjugate Point}}. Given a point $P$ in $\RP^2$, we say that a point $P'$ {\em is 
conjugate to} $\,P\,$ with respect to $\hat{\mathcal{C}}$ iff $\,P'\,$ belongs to the polar line of $P$ (that is, if $P'\in\ell_P$). 
\smallskip

We directly get the 
%%%%%%%%%%%%%%%%%%%%%%%%%%%%%%%%%%%%%%%%%%%%%%%%%%%%%%%%%%%%%%%%%%%%%%%%%%%%%%%%%%%%%%%%%%%%%%%%%%%%%%%%%%%%%%%%%%%%
%%%%%%%%%%%%%%%%%%%%%%%%%%%%%%%       PROPOSITION       %%%%%%%%%%%%%%%%%%%%%%%%%%%%%%%%%%%%%%%%%%%%%%%%%%%%%%%%%%%%
%%%%%%%%%%%%%%%%%%%%%%%%%%%%%%%%%%%%%%%%%%%%%%%%%%%%%%%%%%%%%%%%%%%%%%%%%%%%%%%%%%%%%%%%%%%%%%%%%%%%%%%%%%%%%%%%%%%%
\begin{proposition}\label{conjugacy_psi-orthogonality}
Two \guillemotleft forms\guillemotright\, $\hat{F},\hat{H}$ in $\RP^2$ are conjugate 
iff any two representatives $F$, $H$ {\rm (in the space of quadratic forms)} are 
pseudo-orthogonal, $\langle F,H \rangle_\psi=0$.
\end{proposition}
The following two statements are classical theorems in the p\^ole-polar theory. 

%%%%%%%%%%%%%%%%%%%%%%%%%%%%%%%%%%%%%%%%%%%%%%%%%%%%%%%%%%%%%%%%%%%%%%%%%%%%%%%%%%%%%%%%%%%%%%%%%%%%%%%%%%
%%%%%%%%%%%%%%%%%%%%%%%%%%%%%%%%%%%%%  FACT  %%%%%%%%%%%%%%%%%%%%%%%%%%%%%%%%%%%%%%%%%%%%%%%%%%%%%%%%%%%%%
%%%%%%%%%%%%%%%%%%%%%%%%%%%%%%%%%%%%%%%%%%%%%%%%%%%%%%%%%%%%%%%%%%%%%%%%%%%%%%%%%%%%%%%%%%%%%%%%%%%%%%%%%%
\begin{fact}\label{fact:conj-points}
If two distinct \guillemotleft forms\guillemotright\, $\hat{Q},\hat{Q}'$ in $\RP^2$ are conjugate, 
then at least one of them lies in the non convex component of $\RP^2\smallsetminus\hat{\mathcal{C}}$.
{\rm (Outside $\hat{\mathcal{C}}$.)} 
\end{fact}

%%%%%%%%%%%%%%%%%%%%%%%%%%%%%%%%%%%%%%%%%%%%%%%%%%%%%%%%%%%%%%%%%%%%%%%%%%%%%%%%%%%%%%%%%%%%%%%%%%%%%%%%%%
%%%%%%%%%%%%%%%%%%%%%%%%%%%%%%%%%%%%%  FACT  %%%%%%%%%%%%%%%%%%%%%%%%%%%%%%%%%%%%%%%%%%%%%%%%%%%%%%%%%%%%%
%%%%%%%%%%%%%%%%%%%%%%%%%%%%%%%%%%%%%%%%%%%%%%%%%%%%%%%%%%%%%%%%%%%%%%%%%%%%%%%%%%%%%%%%%%%%%%%%%%%%%%%%%%
\begin{fact}\label{fact:selfconj-triangle}
If three \guillemotleft forms\guillemotright\, $\hat{Q_1},\hat{Q_2},\hat{Q_3}$ in $\RP^2$ form 
a self-conjugate triangle {\rm (the line through any two of them is the polar of the third)}, 
then two forms lie in the non convex component of $\RP^2\smallsetminus\hat{\mathcal{C}}$ 
{\rm (outside $\hat{\mathcal{C}}$)} and the third form lies in the convex one {\rm (inside $\hat{\mathcal{C}}$)}. 
\end{fact}

%%%%%%%%%%%%%%%%%%%%%%%%%%%%%%%%%%%%%%%%%%%%%%%%%%%%%%%%%%%%%%%%%%%%%%%%%%%%%%%%%%%%%%%%%%%%%%
%\subsubsection{\textbf{Polar Duality in $\RP^2$ via the \guillemotleft Projectivised Bracket\guillemotright}}
%%%%%%%%%%%%%%%%%%%%%%%%%%%%%%%%%%%%%%%%%%%%%%%%%%%%%%%%%%%%%%%%%%%%%%%%%%%%%%%%%%%%%%%%%%%%%%

\begin{remark*}
Taking the Poisson bracket up to a constant factor, we get a 
\guillemotleft bracket\guillemotright\, operation in the projective plane of \guillemotleft forms\guillemotright. 
\textit{The  \guillemotleft bracket\guillemotright\, of two points $\hat{F}$, $\hat{H}$ of the projective plane 
$\P(\mathfrak{sl}(2,\R))$, $\{\hat{F},\hat{H}\}:=\P(\{F, H\})$, does not depend 
on the choice of the representatives $F$, $H$, and is the pole with respect to $\hat{\mathcal{C}}$ of the line 
through $\hat{F}$ and $\hat{H}$} \protect\cite{Arnold_Lobachevsky_altitudes}. 
\end{remark*}

{\footnotesize
\noindent 
\textbf{Note}. The representation of quadratic forms on $T_pM$, up to a non zero constant factor, as points 
of the projective plane has been used in \protect\cite{BT2005, Arnold_Lobachevsky_altitudes}, and the Lie algebra of quadratic forms 
$\mathfrak{sl}(2,\R)$ was used in \protect\cite{Arnold_Lobachevsky_altitudes}. 
}

\subsection{\textbf{Local Invariants of Surfaces in $\R^4$ $(\R^5)$ via the Pseudo-Euclidean geometry of Quadratic Forms}}
%%%%%%%%%%%%%%%%%%%%%%%%%%%%%%%%%%%%%%%%%%%%%%%%%%%%%%%%%%%%%%%%%%%%%%%%%%%%%%%%%%%%%%%%%%%%%%%%%%%%%%%%%%%%%%%%%%%%%%%
The following theorems describe the local invariants of a surface $M$ in $\R^4$ (or $\R^5$) at a point $p$ in terms of the 
pseudo-Euclidean properties of the quadratic forms $Q_i$ of $\f_p$.

%%%%%%%%%%%%%%%%%%%%%%%%%%%%%%%%%%%%%%%%%%%%%%%%%%%%%%%%%%%%%%%%%%%%%%%%%%%%%%%%%%%%%%%%%%%%%%%%%%%%
%%%%%%%%%%%%%%%%%%%%%%%%%%%%%%%%%%%%%%%%  THEOREM  %%%%%%%%%%%%%%%%%%%%%%%%%%%%%%%%%%%%%%%%%%%%%%%%%
%%%%%%%%%%%%%%%%%%%%%%%%%%%%%%%%%%%%%%%%%%%%%%%%%%%%%%%%%%%%%%%%%%%%%%%%%%%%%%%%%%%%%%%%%%%%%%%%%%%%
\begin{theorem} 
For every $p\in M$ in $\R^4$ $($or in $\R^5$$)$ the following items are equivalent$:$ 
\smallskip

\noindent
$(a)$ The local quadratic map of $M$ at $p$, $\f_p:T_pM\to N_pM$, 
is given in terms of a principal basis of $N_pM$ {\rm (i.e. $\nbf_1$, $\nbf_2$ or $\nbf_1$, $\nbf_2$, $\nbf_3$ 
are  unit eigenvectors of $\mathcal{G}_p$)}; 
\medskip

\noindent
$(b)$ The quadratic forms $Q_1$, $Q_2$ $($or $Q_1$, $Q_2$, $Q_3$$)$ that take part in that 
local quadratic map $\f_p$ are pseudo-orthogonal to each other. 
\medskip

\noindent
$(c)$ The projectivised \guillemotleft forms\guillemotright\ $\hat{Q}_1$, $\hat{Q}_2$ in $\RP^2$ are conjugate to each other
$($or $\hat{Q}_1$, $\hat{Q}_2$, $\hat{Q}_3$ form a selfconjugate triangle$)$. 
\medskip

\noindent
$(d)$ The squared pseudo-norms of the quadratic forms $Q_i$ are the principal focal 
curvatures of $M$ at $p:$
\[\K_1=\|Q_1\|_\psi^2\,, \, \ \K_2=\|Q_2\|_\psi^2  
\quad \left(\mbox{and also } \K_3=\|Q_3\|_\psi^2 \mbox{ for } M\in\R^5\right)\,.\] 
\end{theorem}

\begin{proof}
The non diagonal entries of the matrix $[\mathcal{G}_p]$ are the pseudo-scalar products of the quadratic forms: 
$k_{ij}=\langle Q_i,Q_j \rangle_{\psi}$. Thus the Gauss quadratic form $\mathcal{G}_p$ is diagonal 
iff the quadratic forms $Q_1$, $Q_2$ $($or $Q_1$, $Q_2$, $Q_3$$)$ are pseudo-orthogonal, and this holds iff 
the diagonal entries $\langle Q_i,Q_i \rangle_{\psi}$ are the eigenvalues, $\mathcal{K}_i$, of $\mathcal{G}_p$. 
The equivalence of $(b)$ and $(c)$ follows from Proposition\,\ref{conjugacy_psi-orthogonality}. 
\end{proof}

%%%%%%%%%%%%%%%%%%%%%%%%%%%%%%%%%%%%%%%%%%%%%%%%%%%%%%%%%%%%%%%%%%%%%%%%%%%%%%%%%%%%%%%%%%%%%%%%%%%%%%%%%%%%%%%%%%%%
\subsubsection{\textbf{Local Invariants of Surfaces in $\R^4$ and Pseudo-Euclidean Geometry}}
%%%%%%%%%%%%%%%%%%%%%%%%%%%%%%%%%%%%%%%%%%%%%%%%%%%%%%%%%%%%%%%%%%%%%%%%%%%%%%%%%%%%%%%%%%%%%%%%%%%%%%%%%%%%%%%%%%%%

The pseudo-unit vector $W:=(1,0,1)$, associated to the quadratic form $W=(s^2+t^2)/2$, generates and orients the
axis of the cone $\mathcal{C}$. The trace of the quadratic form represented by the vector $Q=(a,b,c)$ 
is given by 
\begin{equation}\label{eq:Tr=2< ,W>}
  \Tr Q =2\,\langle Q, W\rangle_\psi=a+c\,.
\end{equation}
%%%%%%%%%%%%%%%%%%%%%%%%%%%%%%%%%%%%%%%%%%%%%%%%%%%%%%%%%%%%%%%%%%%%%%%%%%%%%%%%%%%%%%%%%%%%%%%%%%
%%%%%%%%%%%%%%%%%%%%%%%%%%%%%%%%%%%%%%%% THEOREM %%%%%%%%%%%%%%%%%%%%%%%%%%%%%%%%%%%%%%%%%%%%%%%%%
%%%%%%%%%%%%%%%%%%%%%%%%%%%%%%%%%%%%%%%%%%%%%%%%%%%%%%%%%%%%%%%%%%%%%%%%%%%%%%%%%%%%%%%%%%%%%%%%%%
\begin{theorem}\label{th:features-pallelogram}
If $Q_1$, $Q_2$ are the quadratic forms taking part in the local quadratic map $\f_p:T_pM\to N_pM$ 
at $p\in M$ in $\R^4$ $($in an orthonormal basis $\nbf_1$, $\nbf_2$ of $N_pM$$)$, then 
\smallskip

\noindent
$(a)$ The Gaussian curvature of $M$ at $p$ is the sum of the squared norms 
of the quadratic forms $Q_i$ {\rm (Fig.\ref{fig:parallelogramme-Q_1Q_2}):}   
\[K=\|Q_1\|_\psi^2+\|Q_2\|_\psi^2\,;\] 

\noindent
$(b)$ The determinant invariant $\Delta$ of $M$ at $p$ is equal to the squared pseudo-area of the parallelogram 
formed by $Q_1$ and $Q_2$, and it is also equal to the squared pseudo-norm of the Poisson bracket 
of $Q_1$ and $Q_2$ {\rm (Fig.\ref{fig:parallelogramme-Q_1Q_2}):} 
\[\Delta=\|Q_1\|_{\psi}^2\|Q_2\|_{\psi}^2-\langle Q_1,Q_2\rangle_{\psi}^2=\|\{Q_1,Q_2\}\|_\psi^2\,;\]

\noindent 
$(c)$ The normal curvature of $M$ equals the trace of the Poisson bracket of $Q_1$ and $Q_2$: 
\[N=\Tr \{Q_1,Q_2\}=2\,\langle\, \{Q_1,Q_2\}\,,\, W\, \rangle_\psi\,;\] 

\noindent 
$(d)$ The pseudo-orthogonal projections of $Q_1$, $Q_2$ on the oriented axis of the cone $\mathcal{C}$ 
are the components of the mean curvature vector\,$:$
\[\Hbf=\left(\langle Q_1, W\rangle_\psi, \langle Q_2, W\rangle_\psi\right)\,.\]  
\begin{figure}[ht] %< préférence de placement h , t , b or p >
\centering
\includegraphics[scale=0.145]{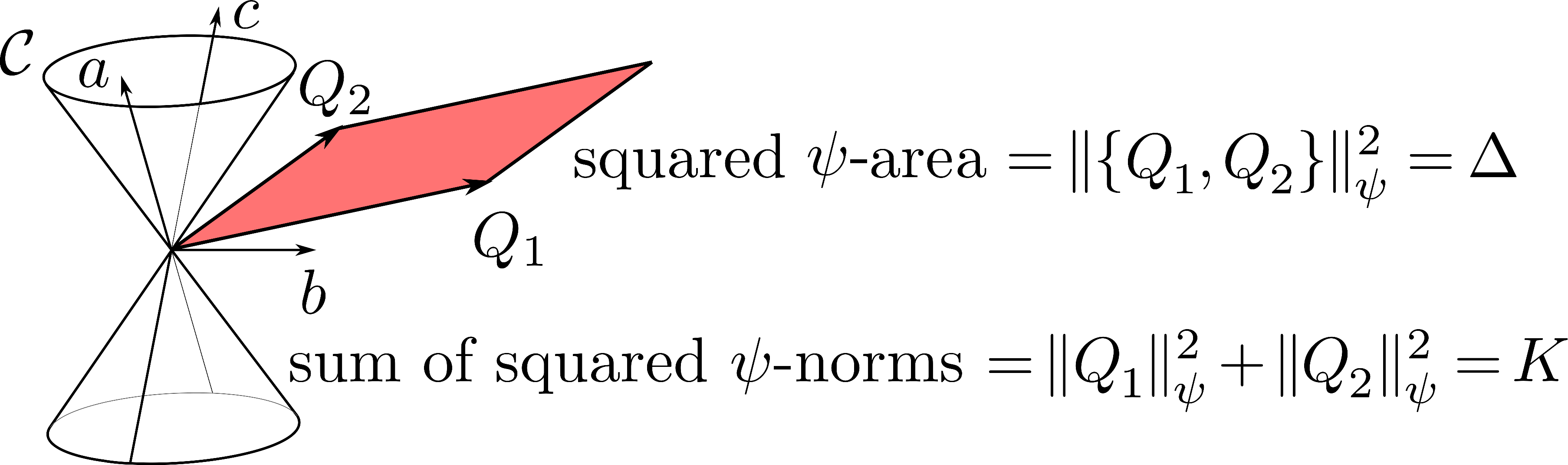}
\caption{\small Invariants $\Delta$ and $K$ for $M$
in $\R^4$ as features of the parallelogram formed by $Q_1$, $Q_2$.}
\label{fig:parallelogramme-Q_1Q_2}
\end{figure}
\end{theorem}

\begin{proof} $(a)$ We have seen $K=k_1+k_2$ for $M$ in $\R^4$, 
where $k_i=a_ic_i-b_i^2=\langle Q_i,Q_i\rangle_\psi$. %(cf. \protect\cite{Little}). 
\smallskip

\noindent
$(b)$ We already know (by Proposition\,\ref{prop:inv-char_polynomial}) that $\Delta=k_1k_2-k_{12}^2$, where 
$k_i=\langle Q_i,Q_i\rangle_\psi$ and $k_{12}=\langle Q_1,Q_2\rangle_\psi$. 
Thus, Theorem\,\ref{th:features-pallelogram}$(b)$ follows from Proposition\,\ref{th:bracket-orthogonality}$(b)$. 
\smallskip

\noindent
$(c)$ From \eqref{eq:Poisson-bracket-form} one directly computes $\Tr\{Q_1,Q_2\}=(a_1-c_1)b_2-(a_2-c_2)b_1=N$. 
\smallskip

\noindent
$(d)$ It is a short and direct calculation. 
\end{proof}

%%%%%%%%%%%%%%%%%%%%%%%%%%%%%%%%%%%%%%%%%%%%%%%%%%%%%%%%%%%%%%%%%%%%%%%%%%%%%%%%%%%%%%%%%%%%%%%%%%
%%%%%%%%%%%%%%%%%%%%%%%%%%%%%%%%%%%%%%%% THEOREM %%%%%%%%%%%%%%%%%%%%%%%%%%%%%%%%%%%%%%%%%%%%%%%%%
%%%%%%%%%%%%%%%%%%%%%%%%%%%%%%%%%%%%%%%%%%%%%%%%%%%%%%%%%%%%%%%%%%%%%%%%%%%%%%%%%%%%%%%%%%%%%%%%%%
\begin{theorem}\label{th:points-by-Q1Q2}
Consider a smooth surface $M$ in $\R^4$. Let $Q_1$, $Q_2$ be the quadratic forms of the local quadratic map 
of $M$ at a point $p$ $($in an orthonormal basis $\nbf_1$, $\nbf_2$ of $N_pM$$)$. 
\smallskip

\noindent
\textit{\textbf{A}}. If $Q_1$, $Q_2$ are not proportional and $\Pi$ is the plane they span {\rm (see Fig.\,\ref{fig:Positions-Plane12})}, then 
\smallskip

\noindent
$(e)$ $\Delta>0$  $($$p$ is elliptic$)$ iff the plane $\Pi$ intersects the cone $\mathcal{C}$ only at the origin; 
\smallskip

\noindent
$(\mathring{e})$ $4\Delta=N^2$ iff $\|\Hbf\|=0$ iff \,$\Pi$ is pseudo-orthogonal to the axis of the cone $\mathcal{C}$. 
\smallskip

\noindent
$(p)$ $\Delta=0$  $($$p$ is parabolic$)$ iff the plane $\Pi$ is tangent to the cone $\mathcal{C}$; 
\smallskip

\noindent
$(h)$ $\Delta<0$  $($$p$ is hyperbolic$)$ iff the plane $\Pi$ intersects the cone $\mathcal{C}$ along two lines; 
\smallskip

\noindent
$(h_{su})$ $N=0$  $($$p$ is semi umbilic$)$ iff \,$\Pi$ contains the axis of the cone $\mathcal{C}$; 
\medskip

\noindent
\textit{\textbf{B}}. If $Q_1$, $Q_2$ span only a line $\ell$, then $\Delta=0$ and $N=0$. And moreover {\rm (see Fig.\,\ref{fig:Positions-line})}
\smallskip

\noindent
$(ri)$ $K<0$, $\|\Hbf\|\neq0$ $($$p$ is a real inflection$)$ iff the line $\ell$ lies outside the cone $\mathcal{C}$;
\smallskip

\noindent
$(\mathring{ri})$ $K<0$, $\|\Hbf\|=0$ $($$p$ is a real inflection$)$ iff \,$\ell$ is pseudo-orthogonal to the axis of $\mathcal{C}$;
\smallskip

\noindent
$(fi)$ $K=0$, $\|\Hbf\|\neq 0$  $($$p$ is a flat inflection$)$ iff the line $\ell$ lies on the cone $\mathcal{C}$;
\smallskip

\noindent
$(ii)$ $\|\Hbf\|^2>K>0$  $($$p$ is an imaginary inflection$)$ iff the line $\ell$ lies inside the cone $\mathcal{C}$;
\smallskip

\noindent
$(u)$ $\|\Hbf\|^2=K> 0$  $($$p$ is an umbilic$)$ iff the line $\ell$ coincides with the axis of the cone $\mathcal{C}$. 
\medskip

\noindent
\textit{\textbf{C}}. If $Q_1,Q_2$ vanish identically, then $\Delta=N=K=\|\Hbf\|=0$ and $p$ is a flat umbilic. 
\end{theorem}
\begin{figure}[ht] %< préférence de placement h , t , b or p >
\centering
\includegraphics[scale=0.14]{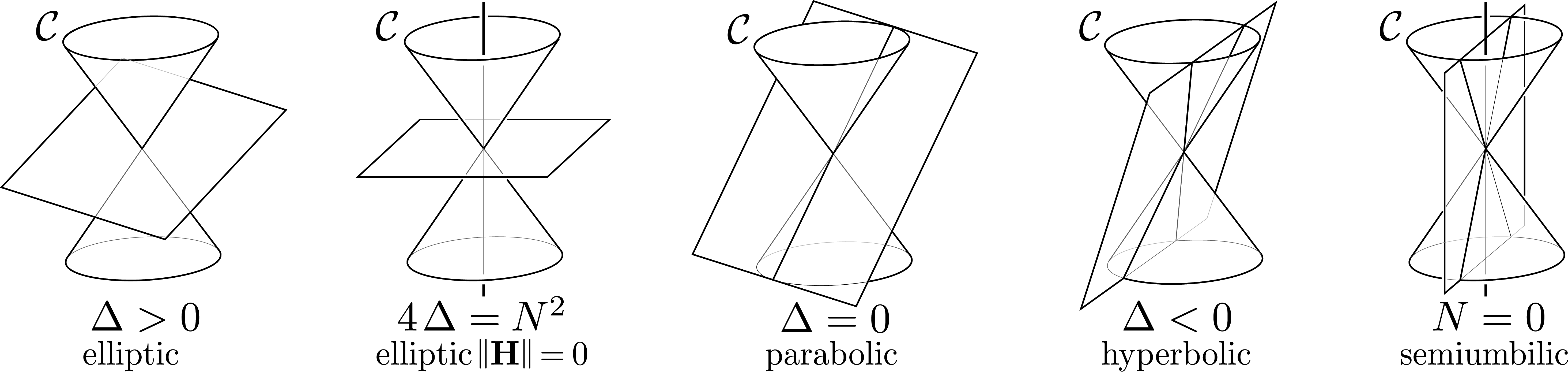}
\caption{\small The positions of the plane $\Pi(Q_1Q_2)$ and the cone $\mathcal{C}$ related to $\Delta$ and $N$.}
\label{fig:Positions-Plane12}
\end{figure}
\begin{figure}[ht] %< préférence de placement h , t , b or p >
\centering
\includegraphics[scale=0.14]{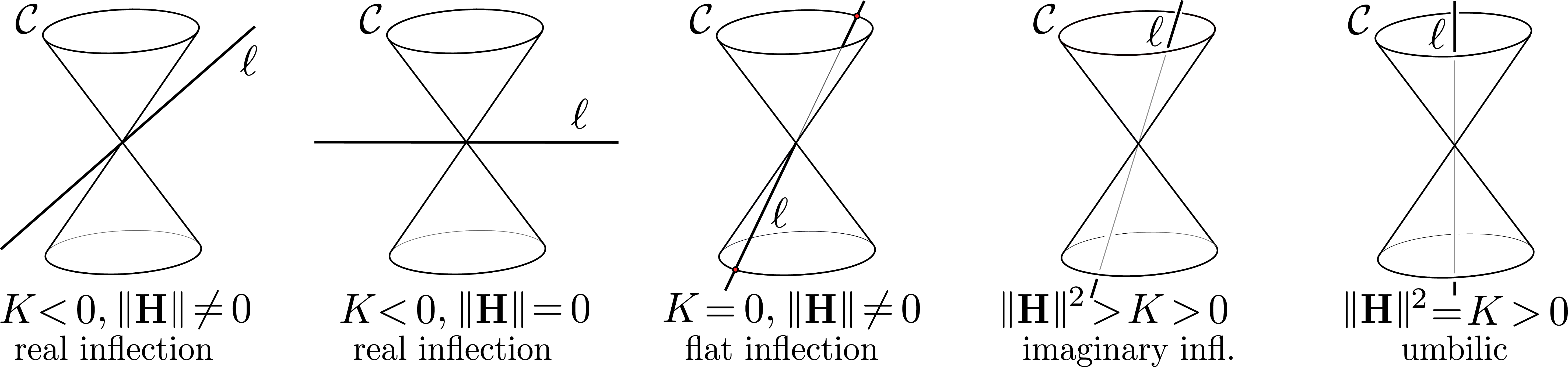}
\caption{\small The positions of the line $\ell(Q_1Q_2)$ and the cone $\mathcal{C}$ at special points where $\Delta=N=0$.}
\label{fig:Positions-line}
\end{figure}

\begin{proof}[\underline{Proof of \textit{\textbf{A}}}]
We have that the bracket $\{Q_1,Q_2\}$ is pseudo-orthogonal to the plane $\Pi$ (Proposition\,\ref{th:bracket-orthogonality}) 
and that $\Delta=\|\{Q_1,Q_2\}\|_{\psi}^2$ (Theorem\,\ref{th:features-pallelogram}). Hence the plane $\Pi$ intersects 
$\mathcal{C}$ along two lines (or is tangent to $\mathcal{C}$, or intersects $\mathcal{C}$ only at the origin) iff the vector 
$\{Q_1,Q_2\}$ lies outside (resp. lies on or lies inside) $\mathcal{C}$, and this happens, by Proposition\,\ref{prop:sign-||Q||}, 
iff $\Delta<0$ (resp. $\Delta=0$, $\Delta>0$). Items $e,p,h$ are proved. 
\smallskip

\noindent
Item $h_{su}$: By Theorem\,\ref{th:features-pallelogram}, $N=0$ iff $\Tr \{Q_1,Q_2\}=0$. 
So the vector $\{Q_1,Q_2\}$ is of the form $(\a,\beta, -\a)$. 
Then the plane $\Pi$, pseudo-orthogonal to $(\a,\beta, -\a)$, consists of the vectors $(a,b,c)$ satisfying 
the equation $\langle (\a,\beta, -\a), (a,b,c,)\rangle_\psi=0$, that is $\a(a-c)+2\beta b=0$. Therefore, $\Pi$ 
contains the axis of the cone $\mathcal{C}$ --given by the equations $a=c$, $b=0$.  
\smallskip

\noindent
Item $\mathring{e}$: Consider the unit vector $W=(1,0,1)$ which generates the axis of $\mathcal{C}$. The plane $\Pi$ is 
pseudo-orthogonal to $W$ iff both $Q_1$ and $Q_2$ are pseudo-orthogonal to $W$, that is, iff $a_1+c_1=0$ and $a_2+c_2=0$, 
which is equivalent to $\|\Hbf\|=0$ (and $\|\Rbf\|=0$ because $\Delta\neq 0$). This is equivalent, 
by equation\,\eqref{eq:<R,H>=1-N^2/4D}, to the equality $N^2=4\Delta$. 
\medskip

\noindent
\underline{\textit{Proof of \textit{\textbf{B}}}}. The proportionality of $Q_1$ and $Q_2$ implies that the Poisson bracket $\{Q_1,Q_2\}$ vanish. 
Then, by Theorem\,\ref{th:features-pallelogram},  $N=\Tr\{Q_1,Q_2\}=0$ and $\Delta=\|\{Q_1,Q_2\}\|_{\psi}^2=0$. 
\smallskip

In all items of $B$, the sign of $K$ is obtained from Theorem\,\ref{th:features-pallelogram} 
and Proposition\,\ref{prop:sign-||Q||}. 

Concerning $\|\Hbf\|$, we have seen above that $\|\Hbf\|=0$ iff the vectors $Q_1$, $Q_2$ are pseudo-orthogonal 
to $W=(1,0,1)$.  
Items $ri$, $\mathring{ri}$ and $fi$ are completely proved. 
\smallskip

\noindent
Item $ii$: If $Q_2=\l Q_1$ and $a_1\neq c_1$, we have $(a_1-c_1)^2+(a_2-c_2)^2>0$,  
which is equivalent to $(a_1+c_1)^2+(a_2+c_2)^2> 4a_1c_1+4a_2c_2$. This last inequality implies 
\[\left(\frac{a_1+c_1}{2}\right)^2+\left(\frac{a_2+c_2}{2}\right)^2> a_1c_1-b_1^2+a_2c_2-b_2^2\,,\]
that is, $\|\Hbf\|^2>K>0$.  Item $ii$ is proved. \  It remains to prove Item $u$:
\smallskip

\noindent
The equality $\|\Hbf\|^2=K$ is equivalent to $\left((a_1-c_1)/2\right)^2+\left((a_2-c_2)/2\right)^2=-b_1^2-b_2^2$, 
which holds iff $b_1=b_2=0$ and $a_1=c_1$ and $a_2=c_2$, that is, iff both $Q_1$ and $Q_2$ are proportional 
to the vector $(1,0,1)$ which generates the axis of the cone $\mathcal{C}$.   
\end{proof}

%%%%%%%%%%%%%%%%%%%%%%%%%%%%%%%%%%%%%%%%%%%%%%%%%%%%%%%%%%%%%%%%%%%%%%%%%%%%%%%%%%%%%%%%%%%%%%%%%%%%
%%%%%%%%%%%%%%%%%%%%%%%%%%%%%%%%%%%%%%%%  THEOREM  %%%%%%%%%%%%%%%%%%%%%%%%%%%%%%%%%%%%%%%%%%%%%%%%%
%%%%%%%%%%%%%%%%%%%%%%%%%%%%%%%%%%%%%%%%%%%%%%%%%%%%%%%%%%%%%%%%%%%%%%%%%%%%%%%%%%%%%%%%%%%%%%%%%%%%
\begin{proposition}\label{prop:at-least-one-negative}
Let $p$ be a point of a smooth surface in $\R^4$ at which $\Delta\neq 0$. 
At least one of the two principal focal curvatures $\K_1$, $\K_2$ at $p$ is negative. 
\end{proposition}
\begin{proof}
It follows from Proposition\,\ref{prop:sign-||Q||} and Fact\,\ref{fact:conj-points}.   
\end{proof}

%%%%%%%%%%%%%%%%%%%%%%%%%%%%%%%%%%%%%%%%%%%%%%%%%%%%%%%%%%%%%%%%%%%%%%%%%%%%%%%%%%%%%%%%%%%%%%%%%%%%
%%%%%%%%%%%%%%%%%%%%%%%%%%%%%%%%%%%%%%  COROLLARY  %%%%%%%%%%%%%%%%%%%%%%%%%%%%%%%%%%%%%%%%%%%%%%%%%
%%%%%%%%%%%%%%%%%%%%%%%%%%%%%%%%%%%%%%%%%%%%%%%%%%%%%%%%%%%%%%%%%%%%%%%%%%%%%%%%%%%%%%%%%%%%%%%%%%%%
\begin{corollary}\label{cor:at-least-one-negative}
For $p\in M$ in $\R^4$, with $\Delta\neq 0$, at least one invariant $K$ or $\Delta$ is negative.  
\end{corollary}

%%%%%%%%%%%%%%%%%%%%%%%%%%%%%%%%%%%%%%%%%%%%%%%%%%%%%%%%%%%%%%%%%%%%%%%%%%%%%%%%%%%%%%%%%%%%%%%%%%%%
%%%%%%%%%%%%%%%%%%%%%%%%%%%%%%%%%%%%%%  COROLLARY  %%%%%%%%%%%%%%%%%%%%%%%%%%%%%%%%%%%%%%%%%%%%%%%%%
%%%%%%%%%%%%%%%%%%%%%%%%%%%%%%%%%%%%%%%%%%%%%%%%%%%%%%%%%%%%%%%%%%%%%%%%%%%%%%%%%%%%%%%%%%%%%%%%%%%%
% \begin{corollary}
% At no point $p$ of a smooth surface $M$ in $\R^4$ the Gauss quadratic form $\mathcal{G}_p$ 
% can be positive definite. 
% \end{corollary}
%
% Indeed, $\mathcal{G}_p$ is negative definite if $\Delta>0$, degenerate if $\Delta=0$ and indefinite if $\Delta<0$. 
%
%%%%%%%%%%%%%%%%%%%%%%%%%%%%%%%%%%%%%%%%%%%%%%%%%%%%%%%%%%%%%%%%%%%%%%%%%%%%%%%%%%%%%%%%%%%%%%%%%%%%%%%%%%%%%%%%%%%%
\subsubsection{\textbf{Local Invariants of Surfaces in $\R^5$ and Pseudo-Euclidean Geometry}}
%%%%%%%%%%%%%%%%%%%%%%%%%%%%%%%%%%%%%%%%%%%%%%%%%%%%%%%%%%%%%%%%%%%%%%%%%%%%%%%%%%%%%%%%%%%%%%%%%%%%%%%%%%%%%%%%%%%%

%%%%%%%%%%%%%%%%%%%%%%%%%%%%%%%%%%%%%%%%%%%%%%%%%%%%%%%%%%%%%%%%%%%%%%%%%%%%%%%%%%%%%%%%%%%%%%%%%%%%
%%%%%%%%%%%%%%%%%%%%%%%%%%%%%%%%%%%%%%%%  THEOREM  %%%%%%%%%%%%%%%%%%%%%%%%%%%%%%%%%%%%%%%%%%%%%%%%%
%%%%%%%%%%%%%%%%%%%%%%%%%%%%%%%%%%%%%%%%%%%%%%%%%%%%%%%%%%%%%%%%%%%%%%%%%%%%%%%%%%%%%%%%%%%%%%%%%%%%
\begin{theorem}[Fig.\,\ref{fig:Parallelopode-features}]\label{th:features-pallelotope}
If $Q_1$, $Q_2$, $Q_3$ are the quadratic forms of the local quadratic map of a 
surface $M$ in $\R^5$ at $p$ $($in an orthonormal basis $\nbf_1$, $\nbf_2$, $\nbf_3$ of $N_pM$$)$, then 
\smallskip

\noindent
$(a)$ The Gaussian curvature of $M$ at $p$ is the sum of the squared pseudo-norms 
of the vectors $Q_i$ {\rm (Fig.\ref{fig:Parallelopode-features}):}   
\[K=\|Q_1\|_\psi^2+\|Q_2\|_\psi^2+\|Q_3\|_\psi^2\,.\] 

\noindent
$(b)$ The invariant $\mathcal{A}$ is the sum of the squared pseudo-areas of the parallelogram faces 
$Q_1$-$Q_2$, $Q_2$-$Q_3$, $Q_3$-$Q_1$ of the parallelepiped formed by $Q_1$, $Q_2$, $Q_3$.
These squared pseudo-areas are equal to the squared pseudo-norms of the Poisson brackets of the 
respective pairs of vectors {\rm (Fig.\ref{fig:Parallelopode-features}):}
\[\mathcal{A}=\|\{Q_1,Q_2\}\|_\psi^2+\|\{Q_2,Q_3\}\|_\psi^2+\|\{Q_3,Q_1\}\|_\psi^2\,.\]

\noindent
$(c)$ The torsion of $M$ at $p$ is equal to the oriented pseudo-volume of the parallelepiped 
formed by $Q_1$, $Q_2$, $Q_3$. Moreover, the invariant $\Delta$ of $M$ at $p$ is equal to the 
squared pseudo-volume of that parallelepiped {\rm (Fig.\ref{fig:Parallelopode-features}):} 
\begin{align}
\tau & =\langle Q_1,\{Q_2,Q_3\}\rangle_\psi\,, \\
\Delta & =\det(\langle Q_i,Q_j\rangle_\psi)=\langle Q_1,\{Q_2,Q_3\}\rangle_{\psi}^2=\tau^2\,.
\end{align}

\noindent 
$(d)$ The components of the mean curvature vector $\Hbf$ are the pseudo-orthogonal projections of $Q_1$, $Q_2$, $Q_3$ 
on the oriented axis of the cone $\mathcal{C}$\,$:$
\[\Hbf=\left(\langle Q_1, W\rangle_\psi, \langle Q_2, W\rangle_\psi, \langle Q_3, W\rangle_\psi\right)\,.\]  
$(e)$ The components of the normal vector $\Nbf=(\Abf-\Cbf)\times\Bbf$ are the traces  
of the Poisson brackets $\{Q_2,Q_3\}$, $\{Q_3,Q_1\}$, $\{Q_1,Q_2\}$,
\begin{equation}\label{eq:N=(tr,tr,tr)}
  \Nbf=\left(\Tr\{Q_2,Q_3\}, \Tr\{Q_3,Q_1\}, \Tr\{Q_1,Q_2\}\right)\,.
\end{equation}
Moreover, if $\tau\neq 0$ the components of the umbilical focal centre $\Rbf$ are the pseudo-orthogonal projections of 
the brackets $\{Q_2,Q_3\}$, $\{Q_3,Q_1\}$, $\{Q_1,Q_2\}$, divided by $\tau$, on the oriented axis of the cone $\mathcal{C}$\,$:$
\begin{equation}\label{eq:R=(<,>,<,>,<,>)}
  \Rbf=\left(\left\langle \frac{\{Q_2,Q_3\}}{\tau}, W\right\rangle_\psi, \left\langle \frac{\{Q_3,Q_1\}}{\tau}, W\right\rangle_\psi, 
\left\langle \frac{\{Q_1,Q_2\}}{\tau}, W\right\rangle_\psi\right)\,.
\end{equation}
\begin{figure}[ht] %< préférence de placement h , t , b or p >
\centering
\includegraphics[scale=0.145]{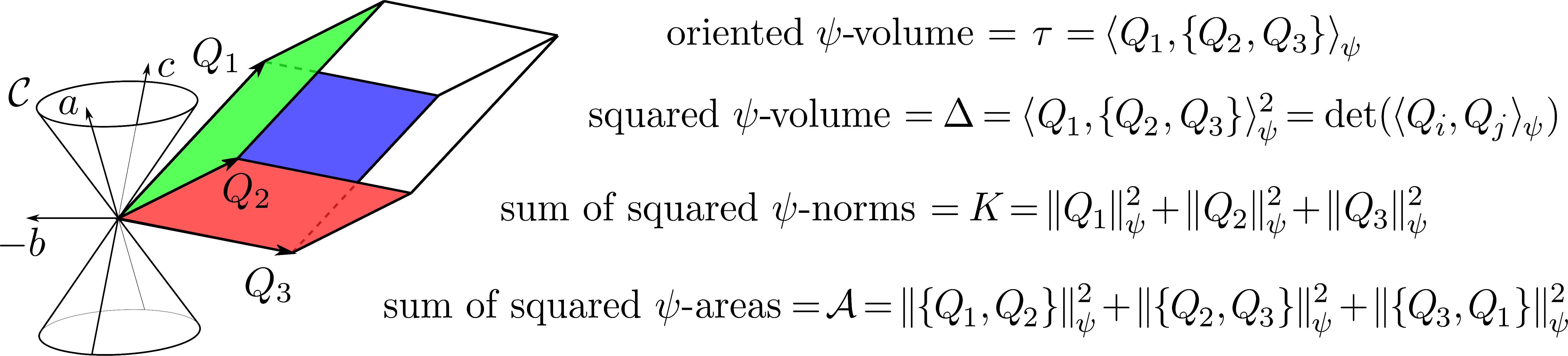}
\caption{\small The local invariants $K$, $\mathcal{A}$, $\tau$, $\Delta$ as features of the 
parallelepiped of $Q_1$-$Q_2$-$Q_3$.}
\label{fig:Parallelopode-features}
\end{figure}
\end{theorem}
\begin{proof}
$(a)$ It is well known that for a surface $M$ in $\R^5$ we have $K=k_1+k_2+k_3$, where 
$k_i=a_ic_i-b_i^2=\langle Q_i,Q_i\rangle_\psi$\, (cf. \protect\cite{Little, IRFRT}).  %and Section\,\ref{section:Gaussian curvature}). 
\smallskip

\noindent
$(b)$ By Proposition\,\ref{th:bracket-orthogonality}\,$(b)$,  the three right-hand side terms 
in the expression of the invariant $\mathcal{A}=(k_1k_2-k_{12}^2)+(k_2k_3-k_{23}^2)+(k_3k_1-k_{31}^2)$ are the 
squared pseudo-areas of the parallelograms formed by the respective pairs of vectors $(Q_1,Q_2)$, $(Q_2,Q_3)$, 
$(Q_3,Q_1)$, and they are also equal to the squared pseudo-norms of the Poisson brackets of those pairs of vectors.  
\smallskip

\noindent
$(c)$ First, one directly checks that 
\begin{align*}
\langle Q_1,\{Q_2,Q_3\}\rangle_\psi & =\textstyle\frac{1}{2}\left((a_3c_2-a_2c_3)b_1+(a_1c_3-a_3c_1)b_2+(a_2c_1-a_1c_2)b_3\right) \\
 & =\textstyle\frac{1}{2}\left\langle\Hbf\times(\Abf-\Cbf)\,,\,\Bbf\right\rangle =\left\langle\frac{1}{2}(\Abf-\Cbf)\times\Bbf\,,\,\Hbf\right\rangle=:\tau\,.
\end{align*}
Next, we have $\Delta=\det[\mathcal{G}_p]$, and that $[\mathcal{G}_p]$ is the Gram matrix of the vectors 
$Q_1$, $Q_2$, $Q_3$ (for $\langle\,\cdot\,,\cdot\,\rangle_\psi$). Therefore Theorem\,\ref{th:features-pallelotope}\,$(c)$ 
follows from Proposition\,\ref{th:bracket-orthogonality}\,$(c)$. 
\smallskip

\noindent
$(d)$ It is a short and direct calculation. 
\smallskip

\noindent
$(e)$ One directly verifies \eqref{eq:N=(tr,tr,tr)}. To prove \eqref{eq:R=(<,>,<,>,<,>)} one 
applies the relation $\Tr Q=2\langle Q,W\rangle$ to the brackets in \eqref{eq:N=(tr,tr,tr)}, and  
uses that $\Nbf=2\tau\Rbf$ (Proposition\,\ref{prop:R=N/2tau}). 
\end{proof}

For $M$ in $\R^5$, we note $N$ the oriented area of the parallelogramme formed by $\Abf-\Cbf$ and $\Bbf$ 
(sign $N$= sign $\tau$). Thus $N=\mathrm{sign}\,\tau\|(\Abf-\Cbf)\times\Bbf\|$. If $\tau\neq 0$, then $N=2\tau\|\Rbf\|$.

%%%%%%%%%%%%%%%%%%%%%%%%%%%%%%%%%%%%%%%%%%%%%%%%%%%%%%%%%%%%%%%%%%%%%%%%%%%%%%%%%%%%%%%%%%%%%%%%%%%%
%%%%%%%%%%%%%%%%%%%%%%%%%%%%%%%%%%%%%%%%  THEOREM  %%%%%%%%%%%%%%%%%%%%%%%%%%%%%%%%%%%%%%%%%%%%%%%%%
%%%%%%%%%%%%%%%%%%%%%%%%%%%%%%%%%%%%%%%%%%%%%%%%%%%%%%%%%%%%%%%%%%%%%%%%%%%%%%%%%%%%%%%%%%%%%%%%%%%%
\begin{theorem}
Let $M$ be a smooth surface in $\R^5$, being $Q_1$, $Q_2$, $Q_3$ the components of its local 
quadratic map at $p$ {\rm (in an orthonormal basis $\nbf_1$, $\nbf_2$, $\nbf_3$ of $N_pM$)}. 
\smallskip
  
\noindent
\textbf{\textit{\small A}}. If $Q_1$, $Q_2$, $Q_3$ are not coplanar and $\nbf_1$, $\nbf_2$, $\nbf_3$ 
is a principal basis of $N_pM$, then 
\smallskip

\noindent
$(\psi e)$ $p$ is pseudo-elliptic \,iff \,  
$\dfrac{\langle Q_1,W\rangle_{\psi}^2}{\langle Q_1, Q_1\rangle_{\psi}^3}+
\dfrac{\langle Q_2,W\rangle_{\psi}^2}{\langle Q_2, Q_2\rangle_{\psi}^3}+
\dfrac{\langle Q_3,W\rangle_{\psi}^2}{\langle Q_3, Q_3\rangle_{\psi}^3}>0$; 
\medskip

\noindent
$(\psi h)$ $p$ is pseudo-hyperbolic \,iff \,  
$\dfrac{\langle Q_1,W\rangle_{\psi}^2}{\langle Q_1, Q_1\rangle_{\psi}^3}+
\dfrac{\langle Q_2,W\rangle_{\psi}^2}{\langle Q_2, Q_2\rangle_{\psi}^3}+
\dfrac{\langle Q_3,W\rangle_{\psi}^2}{\langle Q_3, Q_3\rangle_{\psi}^3}<0$; 
\medskip

\noindent
$(\psi p)$ $p$ is pseudo-parabolic \,iff \,  
$\dfrac{\langle Q_1,W\rangle_{\psi}^2}{\langle Q_1, Q_1\rangle_{\psi}^3}+
\dfrac{\langle Q_2,W\rangle_{\psi}^2}{\langle Q_2, Q_2\rangle_{\psi}^3}+
\dfrac{\langle Q_3,W\rangle_{\psi}^2}{\langle Q_3, Q_3\rangle_{\psi}^3}=0$; 
\medskip

% \noindent
% $(K)$ The Gauss curvature $K$ is equal to the squared pseudo-norm of the diagonal vector $Q:=Q_1+Q_2+Q_3:$ 
% $K=\langle\, Q,Q\,\rangle_\psi$; 
% \medskip
%
% \noindent
% $(\mathcal{A})$ The vector $Q^*:=Q_1^*+Q_2^*+Q_3^*$ provides $\mathcal{A}$, 
% namely $\mathcal{A}=\langle \tau Q^*,\tau Q^*\rangle_\psi$. 
% \medskip
%
\noindent
\textbf{\textit{\small B}}. If $Q_1$, $Q_2$, $Q_3$ lie on a plane $\Pi$, being not colinear, then $\tau=0$. 
Moreover {\rm (Fig.\,\ref{fig:Positions-Plane12})}, 
\smallskip

\noindent
$(\varphi e)$ $p$ is flat-elliptic \,iff\, $\mathcal{A}>0$ \,iff\, $\Pi$ intersects the cone $\mathcal{C}$ only at the origin; 
\smallskip 

\noindent
$(\varphi\mathring{e})$ $p$ satisfies $4\mathcal{A}=N^2>0$ \,iff\, $\|\Hbf\|=0$\,iff\, $\Pi$ 
is pseudo-orthogonal to the axis of $\mathcal{C}$; 
\smallskip 

\noindent
$(\varphi h)$ $p$ is flat-hyperbolic \,iff\, $\mathcal{A}<0$ \,iff\, $\Pi$ intersects the cone $\mathcal{C}$ along two lines; 
\smallskip

\noindent
$(\varphi p)$ $p$ is flat-parabolic \,iff\, $\mathcal{A}=0$ \,iff\, the plane $\Pi$ is tangent to the cone $\mathcal{C}$; 
\smallskip

\noindent
$(su)$ $p$ is semi-umbilic \,iff\, $N=0$ \,iff \, $\Pi$ contains the axis of the cone $\mathcal{C}$. 
\end{theorem}

\begin{proof}[\underline{Proof of \textit{\textbf{A}}}]
From Theorem\,\ref{th:features-pallelotope}-$d$, the equalities $\langle Q_i,Q_i\rangle=\mathcal{K}_i$ and 
relation\,\eqref{eq:relation_for_R-R^5} in Theorem\,\ref{proposition:H=QR}, we get the equalities 
$\langle Q_i,W\rangle_\psi=\rho_i\mathcal{K}_i$, which imply that 
\[\dfrac{\langle Q_1,W\rangle_{\psi}^2}{\langle Q_1, Q_1\rangle_{\psi}^3}+
\dfrac{\langle Q_2,W\rangle_{\psi}^2}{\langle Q_2, Q_2\rangle_{\psi}^3}+
\dfrac{\langle Q_3,W\rangle_{\psi}^2}{\langle Q_3, Q_3\rangle_{\psi}^3}=\mathcal{G}_p^*(\Rbf)\,.\]
Now, item \textbf{\textit{\small A}} follows from the three subitems of 
Theorem\,\ref{prop:caustic-cone}-\textit{\textbf{a}}$_\psi$. 
\medskip

\noindent
\underline{\textit{Proof of \textit{\textbf{B}}}}. One mostly repetes the proof of item \textbf{\textit{\small A}} 
of Theorem\,\ref{th:points-by-Q1Q2}, but the role of the invariant $\Delta$ is played by the invariant $\mathcal{A}$ 
-- in Fig.\,\ref{fig:Positions-Plane12}, replace $\Delta$ with $\mathcal{A}$.  
\end{proof}

%%%%%%%%%%%%%%%%%%%%%%%%%%%%%%%%%%%%%%%%%%%%%%%%%%%%%%%%%%%%%%%%%%%%%%%%%%%%%%%%%%%%%%%%%%%%%%%%%%%%
%%%%%%%%%%%%%%%%%%%%%%%%%%%%%%%%%%%%%%%%  THEOREM  %%%%%%%%%%%%%%%%%%%%%%%%%%%%%%%%%%%%%%%%%%%%%%%%%
%%%%%%%%%%%%%%%%%%%%%%%%%%%%%%%%%%%%%%%%%%%%%%%%%%%%%%%%%%%%%%%%%%%%%%%%%%%%%%%%%%%%%%%%%%%%%%%%%%%%
\begin{proposition}\label{prop:K1<K2<0<K3}
At a point $p$ of a smooth surface in $\R^5$ with $\Delta\neq 0$, two principal 
focal curvatures are negative and one is positive, say $\K_1\leq\K_2<0<\K_3$.  
\end{proposition} 
\begin{proof}
The statement follows from Proposition\,\ref{prop:sign-||Q||} and Fact\,\ref{fact:selfconj-triangle}. 
\end{proof}

% \begin{corollary} 
% $(a)$ At each $p\in M$ in $\R^5$ the determinant invariant satisfies $\Delta\geq 0$. 
% \smallskip
%
% \noindent
% $(b)$ At each point $p$ of $M$ in $\R^5$ the Gauss quadratic form $\mathcal{G}_p$ is not definite. 
% \end{corollary}
%
%
% We naturally get the dual definition. 
% \smallskip
%
% \noindent
% \textit{\textbf{\small Conjugate Line}}. Given a line $\ell$ in $\RP^2$, we say that a line $\ell'$ 
% \textit{is conjugate to} $\,\ell\,$ for $\hat{\mathcal{C}}$ if $\,\ell'\,$ contains the pole of $\ell$. \
% ($\ell$ and $\ell'$ are also called \textit{orthogonal}.)
%
% {\footnotesize
% \begin{exercise*}[for the reader]
% Prove the following evident facts (and the dual statements)\,:
% \smallskip
%
% \noindent 
% $(i)$ {\em Given point $P$, all its conjugate points belong to a line\,$:$ the polar of $P$}. 
% \smallskip
%
% \noindent 
% $(ii)$ {\em A point $P'$ is conjugate to $P$ if and only if $P$ is conjugate to $P'$}. 
% \smallskip
%
% \noindent 
% $(iii)$ {\em The pole of the line through two given points is the intersection point of 
% the polar lines ofthese two points}. 
% \end{exercise*}}

%%%%%%%%%%%%%%%%%%%%%%%%%%%%%%%%%%%%%%%%%%%%%%%%%%%%%%%%%%%%%%%%%%%%%%%%%%%%%%%%%%%%%%%%%%%%%%%%%
%%%%%%%%%%%%%%%%%%%%%%%%%%%%%%%%%%%%%%%%%%%%%%%%%%%%%%%%%%%%%%%%%%%%%%%%%%%%%%%%%%%%%%%%%%%%%%%%%
\section{Inequalities Between Local Invariants in $\R^4$ and in $\R^5$}\label{section:Inequalities-Local-Invariants}
%%%%%%%%%%%%%%%%%%%%%%%%%%%%%%%%%%%%%%%%%%%%%%%%%%%%%%%%%%%%%%%%%%%%%%%%%%%%%%%%%%%%%%%%%%%%%%%%%
%%%%%%%%%%%%%%%%%%%%%%%%%%%%%%%%%%%%%%%%%%%%%%%%%%%%%%%%%%%%%%%%%%%%%%%%%%%%%%%%%%%%%%%%%%%%%%%%%
%%%%%%%%%%%%%%%%%%%%%%%%%%%%%%%%%%%%%%%%%%%%%%%%%%%%%%%%%%%%%%%%%%%%%%%%%%%%%%%%%%%%%%%%%%%%%%%%%
\subsection{\textbf{Inequalities that Restrict the Local Invariants of Surfaces in $\R^4$}}
%%%%%%%%%%%%%%%%%%%%%%%%%%%%%%%%%%%%%%%%%%%%%%%%%%%%%%%%%%%%%%%%%%%%%%%%%%%%%%%%%%%%%%%%%%%%%%%%%
%We provide additional inequalities between the local invariants of surfaces in $\R^4$. 

%%%%%%%%%%%%%%%%%%%%%%%%%%%%%%%%%%%%%%%%%%%%%%%%%%%%%%%%%%%%%%%%%%%%%%%%%%%%%%%%%%%%%%%%%%%%%%%%%%%%%%%%%%%%%%%%%%%%
%%%%%%%%%%%%%%%%%%%%%%%%%%%%%%%       PROPOSITION       %%%%%%%%%%%%%%%%%%%%%%%%%%%%%%%%%%%%%%%%%%%%%%%%%%%%%%%%%%%%
%%%%%%%%%%%%%%%%%%%%%%%%%%%%%%%%%%%%%%%%%%%%%%%%%%%%%%%%%%%%%%%%%%%%%%%%%%%%%%%%%%%%%%%%%%%%%%%%%%%%%%%%%%%%%%%%%%%%
\begin{proposition}\label{prop:inequalities-H^2}
	$(a)$ If $p\in M\subset\R^4$ is elliptic with $\K_1\leq\K_2<0$, then {\rm (Fig.\,\ref{fig:conics-of-H-R})}
	\begin{equation}\label{eq:inequalities-H}
	0\leq \left(1-\frac{N^2}{4\Delta}\right)\K_2 \leq \|\Hbf\|^2 \leq  \left(1-\frac{N^2}{4\Delta}\right)\K_1\,.
	\end{equation}
	and 
	\begin{equation}\label{eq:inequalities-R}
	0\leq \left(1-\frac{N^2}{4\Delta}\right)\Big/\K_1\leq \|\Rbf\|^2 \leq \left(1-\frac{N^2}{4\Delta}\right)\Big/\K_2\,.
	\end{equation}

	$(b)$ If $p\in M$ is hyperbolic with $\K_1<0<\K_2$, then $\|\Hbf\|$ and $\|\Rbf\|$ satisfy {\rm (Fig.\,\ref{fig:conics-of-H-R}):} 
	\[0<\left(1-\frac{N^2}{4\Delta}\right)\K_1 \leq \|\Hbf\|^2 \quad \mbox{and}\quad 0<\left(1-\frac{N^2}{4\Delta}\right)\Big/\K_1\leq \|\Rbf\|^2
	\]
\end{proposition}
%%%%%% P R O O F >>>>>>>>>>>>>>>>>>>>>>>>>>>>>>>>>>>>>>>>>>>>>>>>>>>>>>>>>>>>>>>>>>>>>>>>>>>>>>>>>>>>>>>>>>>>>>>>>>
\begin{proof}
	$(a)$ The inequality $\mathcal{K}_1\leq\mathcal{K}_2<0$ implies the following two inequalities
	\[\frac{h_1^2+h_2^2}{\mathcal{K}_2}\leq\frac{h_1^2}{\mathcal{K}_1}+\frac{h_2^2}{\mathcal{K}_2}\leq\frac{h_1^2+h_2^2}{\mathcal{K}_1}\,,\] 
	which together with equation \eqref{equation:MC-vector}, satisfied by $\|\Hbf\|$, provide inequalities \eqref{eq:inequalities-H}. 
	
	The proofs of \eqref{eq:inequalities-R} and of item $b$ are similar. 
\end{proof}
%%%%%%<<<<<<<<<<<<<<<<<<<<<<<<<<<<<<<<<<<<<<<<<<<<<<<<<<<<<<<<<<<<<<<<<<<<<<<<<<<<<<<<<<<<<<<<<<<<<<<<< P R O O F %%
%%>>>>>>>>>>>>>>>>>>>>>>>>>>>>>>>>>>>>>>>>>>>>>>>>>>>>>>>>>>>>>>>>>>>>>>>>>>>>>>>>>>>>>>>>>>>>>>>>>>>>>>
%%<<<<<<<<<<<<<<<<<<<<<<<<<<<<<<<<<<<<<<<<<<<<<<<<<<<<<<<<<<<<<<<<<<<<<<<<<<<<<<<<<<<<<<<<<<<<<<<<<<<<<<

%%%%%%%%%%%%%%%%%%%%%%%%%%%%%%%%%%%%%%%%%%%%%%%%%%%%%%%%%%%%%%%%%%%%%%%%%%%%%%%%%%%%%%%%%%%%%%%%%%%%
%%%%%%%%%%%%%%%%%%%%%%%%%%%%%%%%%%%%%%%%  THEOREM  %%%%%%%%%%%%%%%%%%%%%%%%%%%%%%%%%%%%%%%%%%%%%%%%%
%%%%%%%%%%%%%%%%%%%%%%%%%%%%%%%%%%%%%%%%%%%%%%%%%%%%%%%%%%%%%%%%%%%%%%%%%%%%%%%%%%%%%%%%%%%%%%%%%%%%
\begin{theorem}\label{Th:restrictions1}
At every point of a smooth surface in $\R^4$ the following relations hold 

$(a)$ $N^2\geq 4\Delta$ \ \ $($equivalently, $N^2\geq 4\K_1\K_2$$)\,;$  

$(b)$ $K^2\geq 4\Delta\,;$ 

$(c)$ If $\Delta\geq 0$, then $K<0$ \ \  $($equivalently, if $K\geq 0$, then $\Delta<0$$)$. 
\end{theorem}
%%%%%%%%%%%%%%%%%%%%%%%%%%%%%%%%%%%%%%%%%%%%%%%%%%%%%%%%%%%%%%%%%%%%%%%%%%%%%%%%%
%%%%%% P R O O F >>>>>>>>>>>>>>>>>>>>>>>>>>>>>>>>>>>>>>>>>>>>>>>>>>>>>>>>>>>>>>>>>>>>>>>>>>>>>>>>>>>>>>>>>>>>>>>>>>
\begin{proof}
Items $a$ and $b$ follow from the following elementary fact on symmetric matrices\,: 
\smallskip 
 
\noindent
\textit{Any symmetric $2\times 2$-matrix $\mathcal{M}$ satisfy} \ $\Tr^2\mathcal{M}\geq 4\det \mathcal{M}$. \ Let us apply it.  
\smallskip

\noindent
$(a)$ For the matrix $[\{Q_1,Q_2\}]$ we have $\Tr[\{Q_1,Q_2\}]=N$\, and\, $\det[\{Q_1,Q_2\}]=\Delta$; \\
$(b)$ For the matrix $[\mathcal{G}_p]$ we have $\Tr[\mathcal{G}_p]=K$\, and\, $\det[\mathcal{G}_p]=\Delta$. 
\smallskip

\noindent
$(c)$ $\Delta\geq 0$ implies that both $\K_1$ and $\K_2$ are not positive  
because at least one of them is negative (Prop.\,\ref{prop:at-least-one-negative}) and $\K_1\K_2=\Delta\geq 0$. 
Therefore $K=\K_1+\K_2<0$.
\end{proof}
%%%%%%<<<<<<<<<<<<<<<<<<<<<<<<<<<<<<<<<<<<<<<<<<<<<<<<<<<<<<<<<<<<<<<<<<<<<<<<<<<<<<<<<<<<<<<<<<<<<<<<< P R O O F %%

%%%%%%%%%%%%%%%%%%%%%%%%%%%%%%%%%%%%%%%%%%%%%%%%%%%%%%%%%%%%%%%%%%%%%%%%%%%%%%%%%%%%%%%%%%%%%%%%%%%%%%%%%
%%%%%%%%%%%%%%%%%%%%%%%%%%%%%%%%%%%%%%%%  COROLLARY  %%%%%%%%%%%%%%%%%%%%%%%%%%%%%%%%%%%%%%%%%%%%%%%%%%%%
%%%%%%%%%%%%%%%%%%%%%%%%%%%%%%%%%%%%%%%%%%%%%%%%%%%%%%%%%%%%%%%%%%%%%%%%%%%%%%%%%%%%%%%%%%%%%%%%%%%%%%%%%
\begin{corollary*}[of Th.\,\ref{Th:restrictions1}\,$(c)$]
Let $M$ be a smooth surface in $\R^4$. At every elliptic or parabolic point the Gauss curvature $K$ is negative. 
If the curve along which $K=0$ is nonempty, then it is entirely contained in the hyperbolic domain, 
which has open regions where $K<0$ and open regions where $K>0$.
\end{corollary*}

Let $p\in M\subset\R^4$ be a point at which the inequality of Th.\,\ref{Th:restrictions1}($b$) 
becomes equality. 
%%%%%%%%%%%%%%%%%%%%%%%%%%%%%%%%%%%%%%%%%%%%%%%%%%%%%%%%%%%%%%%%%%%%%%%%%%%%%%%%%%%%%%%%%%%%%%%%%%%%%%%%%%%%%%%%%%%%
%%%%%%%%%%%%%%%%%%%%%%%%%%%%%%%       PROPOSITION       %%%%%%%%%%%%%%%%%%%%%%%%%%%%%%%%%%%%%%%%%%%%%%%%%%%%%%%%%%%%
%%%%%%%%%%%%%%%%%%%%%%%%%%%%%%%%%%%%%%%%%%%%%%%%%%%%%%%%%%%%%%%%%%%%%%%%%%%%%%%%%%%%%%%%%%%%%%%%%%%%%%%%%%%%%%%%%%%%
\begin{proposition}
Let $K$, $\Delta$, $N$, $\|\Hbf\|$ the local invariants of $M$ in $\R^4$ at $p$. 
The following statements are equivalent: 
\medskip

\noindent 
$(a)$ The equality $K^2=4\Delta$ holds\,;  
\smallskip

\noindent
$(b)$ The local caustic at $p$, $\mathcal{C}_p$, is a circle\, \,{\rm (i.e. $\K_1=\K_2$)}; 
\smallskip

\noindent
$(c)$  The centre $\Rbf$ of $\mathcal{C}_p$, the centre $\Hbf$ of $\mathcal{E}_p$ and $p$ 
are colinear, with $p$ between $\Rbf$ and $\Hbf$; 
\begin{equation}\label{eq:H=(1-N^2/K^2-1)K/2}
\hspace*{-2.5cm}(d) \qquad  \qquad  \|\Hbf\|^2 = \left(\frac{N^2}{K^2}-1\right)\Delta^{\frac{1}{2}}\,  \qquad  \left(\mbox{or } 
\|\Rbf\|^2 = \left(\frac{N^2}{K^2}-1\right)\Delta^{-\frac{1}{2}}\right).
\end{equation}
\end{proposition}
%%%%%% P R O O F >>>>>>>>>>>>>>>>>>>>>>>>>>>>>>>>>>>>>>>>>>>>>>>>>>>>>>>>>>>>>>>>>>>>>>>>>>>>>>>>>>>>>>>>>>>>>>>>>>
\begin{proof}
Take a principal basis $\nbf_1$, $\nbf_2$ of $N_pM$ and let $\mathcal{K}_1$, $\mathcal{K}_2$ 
be the corresponding principal focal curvatures at $p$. Then $K=\mathcal{K}_1+\mathcal{K}_2$ 
and $\Delta=\mathcal{K}_1\mathcal{K}_2$. 
\smallskip

\noindent
$(a)\Leftrightarrow(b)$: The equality $K^2=4\Delta$ is written as 
$(\mathcal{K}_1+\mathcal{K}_2)^2=4\mathcal{K}_1\mathcal{K}_2$, which is equivalent to 
$(\mathcal{K}_1-\mathcal{K}_2)^2=0$, and then equivalent to $\mathcal{K}_1=\mathcal{K}_2$. 
\smallskip

\noindent
$(b)\Leftrightarrow(c)$: The equality $\mathcal{K}_1=\mathcal{K}_2$ holds iff the linear 
operator $\flecha{\mathcal{G}_p}$ is a homothety, that is, $\flecha{\mathcal{G}_p}=\mathcal{K}_1\mathrm{Id}_2$. 
Since $\Hbf=\flecha{\mathcal{G}_p}\Rbf$, we have $\Hbf=\mathcal{K}_1\Rbf$, which means that $\Hbf$, $\Rbf$ and $p$ 
(the origin of $N_pM$) are colinear. The point $p$ lies between $\Rbf$ and $\Hbf$ because $\mathcal{K}_1<0$.  
\smallskip

\noindent
$(b)\Leftrightarrow(d)$: Inequalities\,\eqref{eq:inequalities-H} imply that $\mathcal{K}_1=\mathcal{K}_2$ iff 
equality \eqref{eq:H=(1-N^2/K^2-1)K/2} holds.
\end{proof}
%%%%%%<<<<<<<<<<<<<<<<<<<<<<<<<<<<<<<<<<<<<<<<<<<<<<<<<<<<<<<<<<<<<<<<<<<<<<<<<<<<<<<<<<<<<<<<<<<<<<<<< P R O O F %%

%%%%%%%%%%%%%%%%%%%%%%%%%%%%%%%%%%%%%%%%%%%%%%%%%%%%%%%%%%%%%%%%%%%%%%%%%%%%%%%%%%%%%%%%%%%%%%%%%%
\subsection{\textbf{Inequalities bounding the local invariants of surfaces in $\R^5$}}\label{subsect-inequalities_R^5}
%%%%%%%%%%%%%%%%%%%%%%%%%%%%%%%%%%%%%%%%%%%%%%%%%%%%%%%%%%%%%%%%%%%%%%%%%%%%%%%%%%%%%%%%%%%%%%%%%%
%We provide additional inequalities between the local invariants of surfaces in $\R^5$.
%%%%%%%%%%%%%%%%%%%%%%%%%%%%%%%%%%%%%%%%%%%%%%%%%%%%%%%%%%%%%%%%%%%%%%%%%%%%%%%%%%%%%%%%%%%%%%%%%%%%%%%%%%%%%%%%%%%%
%%%%%%%%%%%%%%%%%%%%%%%%%%%%%%%       PROPOSITION       %%%%%%%%%%%%%%%%%%%%%%%%%%%%%%%%%%%%%%%%%%%%%%%%%%%%%%%%%%%%
%%%%%%%%%%%%%%%%%%%%%%%%%%%%%%%%%%%%%%%%%%%%%%%%%%%%%%%%%%%%%%%%%%%%%%%%%%%%%%%%%%%%%%%%%%%%%%%%%%%%%%%%%%%%%%%%%%%%
\begin{proposition}\label{prop:K_3<H}
	For any $p\in M\subset\R^5$ with $\Delta\neq 0$, the squared mean curvature $\|\Hbf\|^2$ and the squared 
	umbilical radius $\|\Rbf\|^2$ bound the respective eigenvalues of $\mathcal{G}_p$ and $\mathcal{G}_p^*$. 
	Namely, if $\K_1\leq\K_2<0<\K_3$ we have {\rm (Fig.\,\ref{fig:hyperboloids-R-H})}
	\begin{equation}\label{eq:K3<=H^2}
	\K_3\leq\|\Hbf\|^2 \qquad\mbox{and}\qquad \K_3^{-1}\leq\|\Rbf\|^2\,.
	\end{equation}
\end{proposition}
%%%%%% P R O O F >>>>>>>>>>>>>>>>>>>>>>>>>>>>>>>>>>>>>>>>>>>>>>>>>>>>>>>>>>>>>>>>>>>>>>>>>>>>>>>>>>>>>>>>>>>>>>>>>>
\begin{proof}
	Inequalities \eqref{eq:K3<=H^2} are direct consequence of equations \eqref{eq:H&R_K_1,k_2,k_3} -- see also Fig.\,\ref{fig:hyperboloids-R-H}.
\end{proof}
%%%%%%<<<<<<<<<<<<<<<<<<<<<<<<<<<<<<<<<<<<<<<<<<<<<<<<<<<<<<<<<<<<<<<<<<<<<<<<<<<<<<<<<<<<<<<<<<<<<<<<< P R O O F %%

%%%%%%%%%%%%%%%%%%%%%%%%%%%%%%%%%%%%%%%%%%%%%%%%%%%%%%%%%%%%%%%%%%%%%%%%%%%%%%%%%%%%%%%%%%%%%%%%%%%%
%%%%%%%%%%%%%%%%%%%%%%%%%%%%%%%%%%%%%%%%  THEOREM  %%%%%%%%%%%%%%%%%%%%%%%%%%%%%%%%%%%%%%%%%%%%%%%%%
%%%%%%%%%%%%%%%%%%%%%%%%%%%%%%%%%%%%%%%%%%%%%%%%%%%%%%%%%%%%%%%%%%%%%%%%%%%%%%%%%%%%%%%%%%%%%%%%%%%%
\begin{theorem}\label{Th:restrictionsR^5}
At every point of a smooth surface in $\R^5$ the following relations hold 
\medskip

\centerline{
$(a)$ $K^2\geq 3\mathcal{A}\,;$ \hspace*{2cm} 
$(b)$ $\mathcal{A}^2\geq 3K\Delta\,.$ }
\end{theorem}
%%%%%% P R O O F >>>>>>>>>>>>>>>>>>>>>>>>>>>>>>>>>>>>>>>>>>>>>>>>>>>>>>>>>>>>>>>>>>>>>>>>>>>>>>>>>>>>>>>>>>>>>>>>>>
\begin{proof}
Newton's inequality on elementary symmetric functions $\sigma_1$,\ldots,$\sigma_n$ of $n$ variables states that 
the functions $E_k=\sigma_k/\binom{n}{k}$, with $E_0=1$, satisfy the inequalities 
\begin{equation}\label{eq:Newton-inequalities}
  E_k^2\geq E_{k-1}E_{k+1} \qquad (1\leq k \leq n-1).
\end{equation}
Since the invariants $K$, $\mathcal{A}$ and $\Delta$ are the elementary symmetric functions of the 
principal focal curvatures, $K=\K_1+\K_2+\K_3$, $\mathcal{A}=\K_1\K_2+\K_2\K_3+\K_3\K_1$, $\Delta=\K_1\K_2\K_3$, 
we have $E_1=K/3$, $E_2=\mathcal{A}/3$ and $E_3=\Delta$. It remains to apply \eqref{eq:Newton-inequalities} 
for $k=1$ and $2$. %, Theorem\,\ref{Th:restrictionsR^5} is proved.
\end{proof}
%%%%%%<<<<<<<<<<<<<<<<<<<<<<<<<<<<<<<<<<<<<<<<<<<<<<<<<<<<<<<<<<<<<<<<<<<<<<<<<<<<<<<<<<<<<<<<<<<<<<<<< P R O O F %%
%
% \begin{remark*}
% The inequalities of Theorem\,\ref{Th:restrictionsR^5} become equalities only at flat umbilics. 
% \end{remark*}
%

%%>>>>>>>>>>>>>>>>>>>>>>>>>>>>>>>>>>>>>>>>>>>>>>>>>>>>>>>>>>>>>>>>>>>>>>>>>>>>>>>>>>>>>>>>>>>>>>>>>>>>>>

\noindent
\textbf{\textit{Wintgen Inequality}}. Given $p\in M$ in $\R^n$, we note 
$|N|:=\|\Abf-\Cbf\|\|\Bbf\||\sin(\Abf-\Cbf, \Bbf)|$ the area of the parallelogramme 
formed by $\Abf-\Cbf$ and $\Bbf$ in $N_pM$. 

\begin{proposition}[Wintgen Inequality]
At any point of a surface in $\R^{2+\ell}$ 
\begin{equation}\label{eq:wintgen-ineq}
  \|\Hbf\|^2\geq K+|N|\,.
\end{equation}
\end{proposition}

\begin{proof}
For any two vectors $X$, $Y$ in Euclidean space $\R^\ell$ the 
sum of their squared norms is not less than twice the area of the parallelogram that they form: 
\[\|X\|^2+\|Y\|^2\geq 2\|X\|\|Y\|\geq \|2X\|\|Y\||\sin(X,Y)|\,.\]
Applying this inequality to the vectors $X=(\Abf-\Cbf)/2$ and $Y=\Bbf$, we get 
\[\|(\Abf-\Cbf)/2\|^2\geq \|\Abf-\Cbf\|\|\Bbf\||\sin(\Abf-\Cbf,\Bbf)|-\|\Bbf\|^2=|N|-\|\Bbf\|^2\,.\]
Using that $\Hbf=(\Abf+\Cbf)/2$, and the equalities $\|(\Abf+\Cbf)/2\|^2-\|(\Abf-\Cbf)/2\|^2=\langle \Abf,\Cbf\rangle$ and 
$K=\langle \Abf,\Cbf\rangle-\|\Bbf\|^2$, we get $\|\Hbf\|^2\geq K+|N|$. 
\end{proof}

Inequality \eqref{eq:wintgen-ineq} was found for surfaces in $\R^4$ in \protect\cite{Wintgen} (see also \protect\cite{Goncalves}). 

%%%%%%%%%%%%%%%%%%%%%%%%%%%%%%%%%%%%%%%%%%%%%%%%%%%%%%%%%%%%%%%%%%%%%%%%%%%%%%%%%%%%%%%%%%%%%%%%%
%%%%%%%%%%%%%%%%%%%%%%%%%%%%%%%%%%%%%%%%%%%%%%%%%%%%%%%%%%%%%%%%%%%%%%%%%%%%%%%%%%%%%%%%%%%%%%%%%
\section{Paired Quadratic Map and its Invariants}\label{section:Paired Quadratic Form}
%%%%%%%%%%%%%%%%%%%%%%%%%%%%%%%%%%%%%%%%%%%%%%%%%%%%%%%%%%%%%%%%%%%%%%%%%%%%%%%%%%%%%%%%%%%%%%%%%
%%%%%%%%%%%%%%%%%%%%%%%%%%%%%%%%%%%%%%%%%%%%%%%%%%%%%%%%%%%%%%%%%%%%%%%%%%%%%%%%%%%%%%%%%%%%%%%%%

%%%%%%%%%%%%%%%%%%%%%%%%%%%%%%%%%%%%%%%%%%%%%%%%%%%%%%%%%%%%%%%%%%%%%%%%%%%%%%%%%%%%%%%%%%%%%%%%%
\subsection{\textbf{Legendre Transform and Paired Quadratic Map}}
%%%%%%%%%%%%%%%%%%%%%%%%%%%%%%%%%%%%%%%%%%%%%%%%%%%%%%%%%%%%%%%%%%%%%%%%%%%%%%%%%%%%%%%%%%%%%%%%%

Similarly to the local quadratic map of $M$ at $p$, $\f_p:T_pM\to N_pM$, each homogeneous quadratic 
map $\f:\R^2\to\R^\ell$, 
\begin{equation}\label{eq:quad-map-phi}
\f(s,t)=\textstyle\frac{1}{2}\left(s^2\Abf+2st\Bbf+t^2\Cbf\right)=\left(Q_1,\ldots,Q_\ell\right)\,,
\end{equation} 
with $\Abf$, $\Bbf$, $\Cbf\in\R^\ell$, has its indicatrix ellipse 
$\E: \RP^1\to\R^\ell$: 
\begin{equation}\label{eq:parametrisation-E}
\E(\theta)=\frac{\Abf+\Cbf}{2}+\cos 2\theta\frac{\Abf-\Cbf}{2}+\sin 2\theta \Bbf\,, 
\end{equation}
(doubly covered by the map $\hat{\E}:=2\f$ restricted to the unit circle), 
%\[\hat{\E}(\ubf):=2\f(\ubf) \qquad \mbox{with } \|\ubf\|=1\,,\]
has its ``\textit{caustic}'' $\mathcal{C}$ (or ``\textit{focal}'' $\mathcal{F}$) in $\R^\ell\sqcup\RP^{\ell-1}_\infty$ 
defined in affine coordinates by the equation 
\begin{equation}\label{eq:focal-of-phi}
  (\,\langle \Abf\,, \,q \rangle-1\,)(\,\langle \Cbf\,, \,q \rangle-1\,)-\langle \Bbf\,, \,q \rangle^2=0\,, \quad q\in\R^{\ell}\,, 
\end{equation}
and has its associated ``Gauss'' quadratic form $\mathcal{G}:\R^\ell\to\R$ given by 
\begin{equation}\label{eq:QF-of-phi}
  \mathcal{G}(q)=\frac{1}{2}\det\begin{pmatrix}
 \langle \Abf\,, \, q\rangle & \langle \Bbf\,, \, q\rangle  \\ 
 \langle \Bbf\,, \, q\rangle & \langle \Cbf\,, \, q\rangle  
\end{pmatrix}\,,
\end{equation}
together with its associated symmetric linear map $\flecha{\mathcal{G}}:\R^\ell\to\R^\ell$, whose matrix $[\mathcal{G}]$ is the Gram matrix 
of the vectors $Q_1,\ldots,Q_\ell$ for the pseudo-scalar product $\langle\,,\,\rangle_\psi$ on the $3$-dimensional vector 
space $\mathfrak{sl}(2,\R)$. 
\medskip

\noindent 
\textit{\textbf{\small Legendre Transform of a Quadratic Form $\mathcal{G}$}}. 
The Legendre transform of a function $f:\R^\ell\to\R$ is a (may be multivalued) function 
$f^\vee$ of the variables $p_i:=\D f/\D x_i$, given by $f^\vee(p)=\langle p,x \rangle-f(x)$, where the 
$x_i=x_i(p)$ are determined by the ``slopes'' $p_k$. 

For a quadratic form $\mathcal{G}(q)=\frac{1}{2}\langle q,\flecha{\mathcal{G}}q\rangle$ with 
$\det[\mathcal{G}]\neq 0$ 
we have $p=\flecha{\mathcal{G}}q$, $q=\flecha{\mathcal{G}}^{-1}\,p$ and, then, 
$\mathcal{G}^\vee(p)=\langle p, q \rangle-\mathcal{G}(q)=\frac{1}{2}\left\langle p,\flecha{\mathcal{G}}^{-1}p\right\rangle$. 
Thus $\mathcal{G}^\vee:\R^\ell\to\R$ is a quadratic form whose associated symmetric 
linear map $\flecha{\mathcal{G}}^\vee:\R^\ell\to\R^\ell$ is the inverse of $\flecha{\mathcal{G}}$, 
that is $\flecha{\mathcal{G}}^\vee=\flecha{\mathcal{G}}^{-1}$. Clearly $\,(\mathcal{G}^\vee)^\vee=\mathcal{G}$.\, 
Identifying the $q$- and $p$-spaces we get
\begin{equation}\label{eq:Leg-transf-matrix}
\mathcal{G}^\vee(q)=\frac{1}{2}\left\langle q\,,\, \flecha{\mathcal{G}}^{-1}\,q\right\rangle\,.
\end{equation}  

%%%%%%%%%%%%%%%%%%%%%%%%%%%%%%%% DEFINITION %%%%%%%%%%%%%%%%%%%%%%%%%%%%%%%%%%%%%%%%%%%%%%%%%%%%%%%%%%
\noindent 
\textit{\textbf{\small Paired Quadratic Map}}. Take a quadratic map $\f$ given in \eqref{eq:quad-map-phi} 
with $\det[\mathcal{G}]\neq 0$. We call \textit{paired quadratic map} of $\f$ the map $\f^*:\R^2\to\R^\ell$ given by 
$\f^*=\flecha{\mathcal{G}}^{-1}\comp\f$: 
\[\f^*(s,t)=\textstyle\frac{1}{2}(s^2\Abf^*+2st\Bbf^*+t^2\Cbf^*)\,,\] 
where $\Abf^*:=\flecha{\mathcal{G}}^{-1}(\Abf)$, $\Bbf^*:=\flecha{\mathcal{G}}^{-1}(\Bbf)$, $\Cbf^*:=\flecha{\mathcal{G}}^{-1}(\Cbf)$. 

Thus the components of $\f^*$, \,$Q_1^*$,\ldots, $Q_\ell^*$,\, are given in terms of the matrix $[\mathcal{G}]^{-1}$ by
\[\begin{pmatrix}
        Q_1^* \\ 
        \vdots \\
        Q_\ell^*
\end{pmatrix}=
[\mathcal{G}]^{-1}\begin{pmatrix}
        Q_1 \\ 
        \vdots \\
        Q_\ell
\end{pmatrix}\,,\]
and the paired ``Gauss'' quadratic form $\mathcal{G}^*:\R^\ell\to\R$ is given by 
\begin{equation}\label{eq:pairedQF-of-phi}
  \mathcal{G}^*(q)=\frac{1}{2}\det\begin{pmatrix}
 \langle \Abf^*\,, \, q\rangle & \langle \Bbf^*\,, \, q\rangle  \\ 
 \langle \Bbf^*\,, \, q\rangle & \langle \Cbf^*\,, \, q\rangle  
\end{pmatrix}\,.
\end{equation}

%%%%%%%%%%%%%%%%%%%%%%%%%%%%%%%%%%%%%%%%%%%%%%%%%%%%%%%%%%%%%%%%%%%%%%%%%%%%%%%%%%%%%%%%%%%%%%%%%%%%%%%%%%%%%%%%%%%%
%%%%%%%%%%%%%%%%%%%%%%%%%%%%%%%       PROPOSITION       %%%%%%%%%%%%%%%%%%%%%%%%%%%%%%%%%%%%%%%%%%%%%%%%%%%%%%%%%%%%
%%%%%%%%%%%%%%%%%%%%%%%%%%%%%%%%%%%%%%%%%%%%%%%%%%%%%%%%%%%%%%%%%%%%%%%%%%%%%%%%%%%%%%%%%%%%%%%%%%%%%%%%%%%%%%%%%%%%
\begin{proposition}\label{prop:G^v=G^*}
Let $\f:\R^2\to\R^\ell$ a homogeneous quadratic map and $\mathcal{G}:\R^\ell\to\R$ be its quadratic form 
\eqref{eq:QF-of-phi} with $\det\mathcal{G}\neq 0$. 
The Legendre transform $\mathcal{G}^\vee$ of $\mathcal{G}$ coincides with the paired quadratic form $\mathcal{G}^*$ 
{\rm (associated to the quadratic map $\f^*$):} $\mathcal{G}^\vee=\mathcal{G}^*$. 

Moreover, we have that 
\begin{equation}\label{eq:paired-of-paired}
  \flecha{\mathcal{G}}^*=\flecha{\mathcal{G}}^{-1}\,, \quad (\f^*)^*=\f, \quad \mbox{and} \quad (\mathcal{G}^*)^*=\mathcal{G}\,.
\end{equation}
\end{proposition}

\begin{proof}
% \noindent
% \textit{Proof}. 
Using \eqref{eq:pairedQF-of-phi}, the symmetry of $\flecha{\mathcal{G}}^{-1}$ and \eqref{eq:Leg-transf-matrix},  
for any $q\in\R^\ell$ we get the equalities 
\begin{align*}
\hspace{2cm}\mathcal{G}^*(q) & =\frac{1}{2}\det\begin{pmatrix}
 \langle \Abf^*\,, \, q\rangle & \langle \Bbf^*\,, \, q\rangle  \\ 
 \langle \Bbf^*\,, \, q\rangle & \langle \Cbf^*\,, \, q\rangle  
\end{pmatrix}= \frac{1}{2}\det\begin{pmatrix}
 \langle \flecha{\mathcal{G}}^{-1}\Abf\,, \, q\rangle & \langle \flecha{\mathcal{G}}^{-1}\Bbf\,, \, q\rangle  \\ 
 \langle \flecha{\mathcal{G}}^{-1}\Bbf\,, \, q\rangle & \langle \flecha{\mathcal{G}}^{-1}\Cbf\,, \, q\rangle  
\end{pmatrix} \\
                &=\frac{1}{2}\det\begin{pmatrix}
 \langle \Abf\,, \, \flecha{\mathcal{G}}^{-1}q\rangle & \langle \Bbf\,, \, \flecha{\mathcal{G}}^{-1}q\rangle  \\ 
 \langle \Bbf\,, \, \flecha{\mathcal{G}}^{-1}q\rangle & \langle \Cbf\,, \, \flecha{\mathcal{G}}^{-1}q\rangle  
\end{pmatrix}=\mathcal{G}\left(\flecha{\mathcal{G}}^{-1}q\right) \\ 
  & =\frac{1}{2}\left\langle \flecha{\mathcal{G}}\flecha{\mathcal{G}}^{-1}q\,,\,\flecha{\mathcal{G}}^{-1}q\right\rangle=
  \frac{1}{2}\left\langle q\,,\,\flecha{\mathcal{G}}^{-1}q\right\rangle=\mathcal{G}^\vee(q)\,. %\hspace{3.9cm}{\square}
\end{align*}
The relations \eqref{eq:paired-of-paired} follow directly from the first statement.
\end{proof}

%%%%%%%%%%%%%%%%%%%%%%%%%%%%%%%%%%%%%%%%%%%%%%%%%%%%%%%%%%%%%%%%%%%%%%%%%%%%%%%%%%%%%%%%%%%%%%%%%%%%%%%%%%%%%%%%%%%%
\subsection{\textbf{Paired Indicatrix, Paired Local Caustic and Paired Invariants}}
%%%%%%%%%%%%%%%%%%%%%%%%%%%%%%%%%%%%%%%%%%%%%%%%%%%%%%%%%%%%%%%%%%%%%%%%%%%%%%%%%%%%%%%%%%%%%%%%%%%%%%%%%%%%%%%%%%%%
We are essentially concerned with quadratic maps $\R^2\to\R^\ell$, given in  \eqref{eq:quad-map-phi}, 
for $\ell=2$ and $\ell=3$ because the vectors $\Abf$, $\Bbf$, $\Cbf$ generate a vector space of dimension 
at most $3$, and there is at most three linearly independent vectors $Q_i$ in $\mathfrak{sl}(2,\R)$. 
The cases $\ell=2$ and $3$ correspond respectively to the local quadratic maps of a surface $M$ in $\R^4$ and in $\R^5$. 
\medskip

\noindent
\textit{\textbf{\small Paired Local Caustic}}.
The \textit{paired local caustic} $\mathcal{C}_p^*$ of $M\subset\R^{\ell+2}$ at $p$ is the hypersurface of the 
normal space $N_pM$ which, in affine coordinates, is given by the equation
\[(\,\langle \Abf^*\,, \,q \rangle-1\,)(\,\langle \Cbf^*\,, \,q \rangle-1\,)-\langle \Bbf^*\,, \,q \rangle^2=0\,.\]

\noindent
\textit{\textbf{\small Paired Indicatrix Ellipse}}.
The \textit{paired indicatrix ellipse} $\E_p^*$ of $M\subset\R^{\ell+2}$ at $p$ is the indicatrix ellipse 
of the paired local quadratic map $\f_p^*$. That is, $\E_p^*:\RP^1\to N_pM$, 
\begin{equation}\label{eq:parametrisation-E_p^*}
\E_p^*(\theta)=\frac{\Abf^*+\Cbf^*}{2}+\cos 2\theta\frac{\Abf^*-\Cbf^*}{2}+\sin 2\theta \Bbf^*\,,
\end{equation}
and $\E_p^*$ is doubly covered by restricting $2\f_p^*$ to $\sph^1$: $\,\hat{\E_p^*}(\ubf)=2\,\f_p^*(\ubf)\,$ with $\|\ubf\|=1$.
\medskip

The paired (second order) local invariants of a surface have the following expressions 

%%%%%%%%%%%%%%%%%%%%%%%%%%%%%%%%%%%%%%%%%%%%%%%%%%%%%%%%%%%%%%%%%%%%%%%%%%%%%%%%%%%%%%%%%%%%%%%%%%%%%%%%%%%%%%%%%%%%
%%%%%%%%%%%%%%%%%%%%%%%%%%%%%%%       PROPOSITION       %%%%%%%%%%%%%%%%%%%%%%%%%%%%%%%%%%%%%%%%%%%%%%%%%%%%%%%%%%%%
%%%%%%%%%%%%%%%%%%%%%%%%%%%%%%%%%%%%%%%%%%%%%%%%%%%%%%%%%%%%%%%%%%%%%%%%%%%%%%%%%%%%%%%%%%%%%%%%%%%%%%%%%%%%%%%%%%%%
\begin{proposition}\label{prop:paired-inv-R^4}
For each point $p$ of $M$ in $\R^4$ with $\Delta\neq 0$ we have 
\[\Delta^*=\frac{1}{\Delta}\,, \quad K^*=\frac{K}{\Delta}\,, \quad N^*=\frac{N}{\Delta}\,, 
\quad \mathcal{K}_1^*=\frac{1}{\mathcal{K}_1} \quad \mbox{and} \quad \mathcal{K}_2^*=\frac{1}{\mathcal{K}_2}\,. \]
\end{proposition}
As a direct corollary we get that the quantity $N^2/\Delta$ is self-paired: 
\begin{equation}\label{N^2/D=N*^2/D*}
  \frac{{N^*}^2}{\Delta^*}=\frac{N^2}{\Delta}
\end{equation}

%%%%%%%%%%%%%%%%%%%%%%%%%%%%%%%%%%%%%%%%%%%%%%%%%%%%%%%%%%%%%%%%%%%%%%%%%%%%%%%%%%%%%%%%%%%%%%%%%%%%%%%%%%%%%%%%%%%%
%%%%%%%%%%%%%%%%%%%%%%%%%%%%%%%       PROPOSITION       %%%%%%%%%%%%%%%%%%%%%%%%%%%%%%%%%%%%%%%%%%%%%%%%%%%%%%%%%%%%
%%%%%%%%%%%%%%%%%%%%%%%%%%%%%%%%%%%%%%%%%%%%%%%%%%%%%%%%%%%%%%%%%%%%%%%%%%%%%%%%%%%%%%%%%%%%%%%%%%%%%%%%%%%%%%%%%%%%
\begin{proposition}\label{prop:paired-inv-R^5}
For each point $p$ of $M$ in $\R^5$ with $\Delta\neq 0$ we have 
\begin{equation}\label{eq:relations-paired-invR^5}
\Delta^*=\frac{1}{\Delta}\,, \quad K^*=\frac{\mathcal{A}}{\Delta}\,, \quad \mathcal{A}^*=\frac{K}{\Delta}\,, 
\quad \tau^*=\frac{\tau}{\Delta}=\frac{1}{\tau} \quad \mbox{and} \quad \mathcal{K}_i^*=\frac{1}{\mathcal{K}_i}\,.
\end{equation}
\end{proposition}

\begin{proof}
To prove Propositions \ref{prop:paired-inv-R^4} and \ref{prop:paired-inv-R^5} one uses that 
$\flecha{\mathcal{G}_p}^*=\flecha{\mathcal{G}_p}^{-1}$ which implies that 
$\Delta^*=\det[\mathcal{G}_p^*]=1/\det[\mathcal{G}_p]=1/\Delta$. 
For example, 
\begin{align*}
\tau^* & =\det\left(\frac{\Abf^*-\Cbf^*}{2} \ \ \Bbf^* \ \ \Hbf^*\right)
=\det\left(\flecha{\mathcal{G}_p}^*\frac{\Abf-\Cbf}{2} \ \ \flecha{\mathcal{G}_p}^*\Bbf \ \ \flecha{\mathcal{G}_p}^*\Hbf\right) \\
 & =\det\left(\left[\flecha{\mathcal{G}_p}^*\right]\left(\frac{\Abf-\Cbf}{2} \ \ \Bbf \ \ \Hbf\right)\right)
=\det\left[\flecha{\mathcal{G}_p}^*\right]\det\left(\frac{\Abf-\Cbf}{2} \ \ \Bbf \ \ \Hbf\right) \\ 
 & =\frac{1}{\Delta}\tau=\frac{1}{\tau}
\end{align*}
Also the relation $\flecha{\mathcal{G}_p}^*=\flecha{\mathcal{G}_p}^{-1}$ directly implies that $\K_i^*=1/\K_i$. 
\end{proof}

%%%%%%%%%%%%%%%%%%%%%%%%%%%%%%%%%%%%%%%%%% R E M A R K %%%%%%%%%%%%%%%%%%%%%%%%%%%%%%%%%%%%%%%%%%%%%%%%%%%%%%
\begin{remark*}
\textit{Given $p\in M$ in $\R^5$ with $\Delta\neq 0$ the inequalities $\,K^2>3\mathcal{A}\,$ and $\,\mathcal{A}^2>3K\Delta\,$} 
(of Theorem\,\ref{Th:restrictionsR^5}) \textit{are paired to each other.} 
\end{remark*}
\begin{proof}
Using Proposition \ref{prop:paired-inv-R^5} the inequality $(K^*)^2>3\mathcal{A}^*$ becomes $\mathcal{A}^2>3K\Delta$ and 
the inequality $(\mathcal{A}^*)^2>3K^*\Delta^*$ becomes $K^2>3\mathcal{A}$.
\end{proof} 

%%%%%%%%%%%%%%%%%%%%%%%%%%%%%%%%%%%%%%%%%%%%%%%%%%%%%%%%%%%%%%%%%%%%%%%%%%%%%%%%%%%%%%%%%%%%%%%%%%%%%%%%%%%%%%%%%%%%
\subsection{\textbf{Relations Between $\E_p$, $\mathcal{C}_p$ and Their Paired Versions $\E_p^*$, $\mathcal{C}_p^*$}}
%%%%%%%%%%%%%%%%%%%%%%%%%%%%%%%%%%%%%%%%%%%%%%%%%%%%%%%%%%%%%%%%%%%%%%%%%%%%%%%%%%%%%%%%%%%%%%%%%%%%%%%%%%%%%%%%%%%%

%%%%%%%%%%%%%%%%%%%%%%%%%%%%%%%%%%%%%%%%%%%%%%%%%%%%%%%%%%%%%%%%%%%%%%%%%%%%%%%%%%%%%%%%%%%%%%%%%%%%%%%%%%%%%%%%%%%%
%%%%%%%%%%%%%%%%%%%%%%%%%%%%%%%       PROPOSITION       %%%%%%%%%%%%%%%%%%%%%%%%%%%%%%%%%%%%%%%%%%%%%%%%%%%%%%%%%%%%
%%%%%%%%%%%%%%%%%%%%%%%%%%%%%%%%%%%%%%%%%%%%%%%%%%%%%%%%%%%%%%%%%%%%%%%%%%%%%%%%%%%%%%%%%%%%%%%%%%%%%%%%%%%%%%%%%%%%
\begin{proposition}\label{prop:H-R_paired}
For each point $p$ of a surface in $\R^4$ $($or in $\R^5$$)$ with $\Delta\neq 0$ we have$:$

\noindent
$a)$ The centre $\Hbf$ of the indicatrix
ellipse $\E_p$ coincides with the centre $\Rbf^*$ of the paired local caustic $\mathcal{C}_p^*$ 
and the centre $\Rbf$ of the local caustic $\mathcal{C}_p$ coincides with the centre $\Hbf^*$ of the 
paired indicatrix ellipse $\E_p^*$:
\begin{equation}\label{eq:R*=H-H*=R}
  \Rbf^*=\Hbf \quad \mbox{and}\quad \Hbf^*=\Rbf\,.
\end{equation}
$b)$ The linear map $\flecha{\mathcal{G}_p}$ sends $\E_p^*$ onto $\E_p$, that is 
$\,\flecha{\mathcal{G}_p}\comp\hat{\E_p^*}(\ubf)=\hat{\E_p}(\ubf)$ \,$(\|u\|=1)$. 
\medskip

\noindent
$c)$ The linear map $\flecha{\mathcal{G}_p}$ sends $\mathcal{C}_p$ onto $\mathcal{C}_p^*$, that is 
$\,\flecha{\mathcal{G}_p}(\mathcal{C}_p)=\mathcal{C}_p^*$. 
\end{proposition}
\begin{proof}
$a)$ The relations \eqref{eq:R*=H-H*=R} follow from Proposition \ref{proposition:H=QR} and Proposition\,\ref{prop:G^v=G^*}. 
\smallskip

\noindent
$b)$ To prove that $\flecha{\mathcal{G}_p}$ sends $\E_p^*$ onto $\E_p$, we shall use their double covering maps: 
\begin{align*}
\hspace{0.7cm}\flecha{\mathcal{G}_p}\comp\hat{\E_p^*}(\ubf) &=2\flecha{\mathcal{G}_p}\comp\f_p^*(\ubf)=
2\flecha{\mathcal{G}_p}\flecha{\mathcal{G}_p}^{-1}\comp\f_p(\ubf)=2\f_p(\ubf) \\
             &=\hat{\E_p}(\ubf)\,.
\end{align*}

\noindent
$c)$ We shall use equation \eqref{eq:Gp(q-R)=-N2/4D} for $\mathcal{C}_p$ and $\mathcal{C}p^*$, to prove that $\,q\in\mathcal{C}_p\,$ iff 
$\,\flecha{\mathcal{G}_p}q\in\mathcal{C}_p^*$: 
\begin{align*}
\hspace{2cm}\mathcal{G}_p(q-\Rbf)& =\frac{1}{2}\left\langle \flecha{\mathcal{G}_p}(q-\Rbf)\,,\,q-\Rbf\right\rangle \\
             &=\frac{1}{2}\left\langle \flecha{\mathcal{G}_p}(q-\Rbf)\,,\,\flecha{\mathcal{G}_p}^{-1}\flecha{\mathcal{G}_p}(q-\Rbf)\right\rangle
             =\mathcal{G}_p^*\left(\flecha{\mathcal{G}_p}(q-\Rbf)\right)  \\
                &= \mathcal{G}_p^*\left(\flecha{\mathcal{G}_p}q-\Rbf^*\right)\,.
\end{align*}
Thus, by relation \eqref{N^2/D=N*^2/D*}, $2\mathcal{G}_p(q-\Rbf)=-N^2/4\Delta\,$ iff 
$\,2\mathcal{G}_p^*\left(\flecha{\mathcal{G}_p}q-\Rbf^*\right)=-{N^*}^2/4\Delta^*$. 
\end{proof}

%%%%%%%%%%%%%%%%%%%%%%%%%%%%%%%%%%%%%%%%%%%%%%%%%%%%%%%%%%%%%%%%%%%%%%%%%%%%%%%%%%%%%%%%%%%%%%%%%%%%%%%%%%%%%%%%%%%%
%%%%%%%%%%%%%%%%%%%%%%%%%%%%%%%       PROPOSITION       %%%%%%%%%%%%%%%%%%%%%%%%%%%%%%%%%%%%%%%%%%%%%%%%%%%%%%%%%%%%
%%%%%%%%%%%%%%%%%%%%%%%%%%%%%%%%%%%%%%%%%%%%%%%%%%%%%%%%%%%%%%%%%%%%%%%%%%%%%%%%%%%%%%%%%%%%%%%%%%%%%%%%%%%%%%%%%%%%
\begin{theorem}[Paired Wintgen Inequalities]\label{th:paired-wingten-inequalities}
\textit{Given $p\in M$ in $\R^4$ $($or in $\R^5$$)$ with $\Delta\neq 0$, let $\Rbf$ be the centre $($resp. the vertex$)$ 
of the local caustic $\mathcal{C}_p$. Then,} 
\[(a) \, \|\Rbf\|^2\geq K/\Delta+|N/\Delta| \qquad \quad \mbox{ for $M$ in $\R^4$}\,.\] 
\[(b) \, \|\Rbf\|^2\geq \mathcal{A}/\Delta+2\|\Hbf\|/|\tau| \qquad \mbox{for $M$ in $\R^5$}\,.\] 
\end{theorem}
\begin{proof}
$(a)$ Use Wintgen inequality for $\f_p^*$, $\|\Hbf^*\|^2\geq K^*+|N^*|$, and Prop. \ref{prop:H-R_paired} and \ref{prop:paired-inv-R^4}.

\noindent
$(b)$ Use Wintgen inequality for $\f_p^*$ together with \eqref{eq:R*=H-H*=R}, the paired version
of Proposition\,\ref{prop:R=N/2tau} (i.e. $2|\tau^*|\|\Rbf^*\|=|N^*|$), again \eqref{eq:R*=H-H*=R}  
and \eqref{eq:relations-paired-invR^5}.  
\end{proof}

Notice that the inequality of Theorem\,\ref{th:paired-wingten-inequalities}-$b$ is equivalent to the inequality 
\[N^2\geq4\mathcal{A}+8|\tau|\|\Hbf\|\,.\]

%%%%%%%%%%%%%%%%%%%%%%%%%%%%%%%%%%%%%%%%%%%%%%%%%%%%%%%%%%%%%%%%%%%%%%%%%%%%%%%%%%%%%%%ant%%%%%%%%%%%%%%%%%%%%%%%%%%%%%
\subsubsection{\textbf{The configurations of $\E_p$, $\mathcal{C}_p$ and $\E_p^*$, $\mathcal{C}_p^*$ for non parabolic $p\in M$ in $\R^4$}}
%%%%%%%%%%%%%%%%%%%%%%%%%%%%%%%%%%%%%%%%%%%%%%%%%%%%%%%%%%%%%%%%%%%%%%%%%%%%%%%%%%%%%%%%%%%%%%%%%%%%%%%%%%%%%%%%%%%%

%%%%%%%%%%%%%%%%%%%%%%%%%%%%%%%%%%%%%%%%%%%%%%%%%%%%%%%%%%%%%%%%%%%%%%%%%%%%%%%%%%%%%%%%%%%%%%%%%%%%%%%%%%%%%%%%%%%%
%%%%%%%%%%%%%%%%%%%%%%%%%%%%%%%       PROPOSITION       %%%%%%%%%%%%%%%%%%%%%%%%%%%%%%%%%%%%%%%%%%%%%%%%%%%%%%%%%%%%
%%%%%%%%%%%%%%%%%%%%%%%%%%%%%%%%%%%%%%%%%%%%%%%%%%%%%%%%%%%%%%%%%%%%%%%%%%%%%%%%%%%%%%%%%%%%%%%%%%%%%%%%%%%%%%%%%%%%
\begin{proposition}\label{prop:Cp_bitangent_Ep^*}
For each point $p$ of a surface in $\R^4$ with $\Delta\neq 0$ we have {\rm (Fig.\,\ref{fig:Cp-Ep_Cp*-Ep*-elliptic})}$:$ 

\noindent
$a)$ The paired indicatrix ellipse $\E_p^*$ is bitangent to the local caustic $\mathcal{C}_p$. 
The two tangency points are those where the line $\{\l\Rbf:\l\in\R\}$ cuts $\mathcal{C}_p$. 
Both tangency directions are orthogonal to $\Hbf$. 
\medskip

\noindent
$b)$ The indicatrix ellipse $\E_p$ is bitangent to the paired local caustic $\mathcal{C}_p^*$. 
The two tangency points are those where the line $\{\l\Hbf:\l\in\R\}$ cuts $\mathcal{C}_p^*$. 
Both tangency directions are orthogonal to $\Rbf$.
\end{proposition}

\begin{proof}
$a)$ By relations \eqref{eq:Gp(q-R)=-N2/4D} and \eqref{eq:<R,H>=1-N^2/4D}, the two points at which the line 
$\{\l\Rbf:\l\in\R\}$ cuts $\mathcal{C}_p$ correspond to the values of $\l$ which satisfy the equation 
\begin{equation}\label{eq:Gp[lR-R]=RH-1}
  2\mathcal{G}_p(\l\Rbf-\Rbf)=\langle\Rbf\,,\,\Hbf\rangle-1\,.
\end{equation}
Equation \eqref{eq:Gp[lR-R]=RH-1} is rewritten as 
$(\l-1)^2\langle \Rbf\,,\,\flecha{\mathcal{G}_p}\Rbf\rangle=\langle\Rbf\,,\,\Hbf\rangle-1$, which is equivalent to 
\begin{equation}\label{eq:eq-bitangent-points}
  \langle\Rbf\,,\,\Hbf\rangle\l^2-2\langle\Rbf\,,\,\Hbf\rangle\l+1=0\,.
\end{equation}
Take $\hat{q}:=\l\Rbf$, being $\l$ a solution of \eqref{eq:eq-bitangent-points}. 
The tangent line to $\mathcal{C}_p$ at $\hat{q}$, $\ell_{\hat{q}^\vee}$, is given by the equation 
\[d{\mathcal{G}_p}_{|\hat{q}}(q-\hat{q})=0\,,\] 
where $\,d{\mathcal{G}_p}_{|\hat{q}}=\l(\K_1\rho_1dq_1+\K_2\rho_2dq_2)=\l(h_1dq_1+h_2dq_2)$,\, that is, 
the equation of the tangent line $\ell_{\hat{q}^\vee}$ is $\,h_1(q_1-\l\rho_1)+h_2(q_2-\l\rho_2)=0$,\, which is equivalent to 
\begin{equation}\label{eq:tangency-line}
 \frac{h_1}{\l\langle\Rbf,\Hbf\rangle}q_1+\frac{h_2}{\l\langle\Rbf,\Hbf\rangle}q_2=1\,. 
\end{equation}

Therefore, the point $\hat{q}^\vee=\frac{1}{\l\langle\Rbf,\Hbf\rangle}\Hbf$ is the p\^ole of 
the tangent line $\ell_{\hat{q}^\vee}$ (see \S\ref{sect-polar_dual-subvariety}). 
So, by Theorem \ref{th:focal_quadric-indicatrix}-\textit{\textbf{a}}, 
$\hat{q}^\vee$ lies on the ellipse $\E_p$. 
To prove that $\hat{q}^\vee$ is also a point of the paired local caustic $\mathcal{C}_p^*$, 
observe that by relations \eqref{eq:Gp(q-R)=-N2/4D} and \eqref{eq:<R,H>=1-N^2/4D} for $\mathcal{C}_p^*$, and 
by equalities \eqref{eq:R*=H-H*=R}, the two points at which the line $\{\a\Hbf:\a\in\R\}$ cuts $\mathcal{C}_p^*$ 
correspond to the two values of $\a$ which are solutions of equation \eqref{eq:eq-bitangent-points}. 

Now $\l$ is a solution of \eqref{eq:eq-bitangent-points} iff the second solution of \eqref{eq:eq-bitangent-points} 
is $\a=1/\l\langle\Rbf,\Hbf\rangle$.
Therefore,  $\hat{q}^\vee$ is a point of $\mathcal{C}_p^*$ and, by Theorem \ref{th:focal_quadric-indicatrix}-\textit{\textbf{a}} 
applied to $\mathcal{C}_p^*$ and $\E_p^*$, its polar line $\ell_{\hat{q}^\vee}$ is tangent to $\E_p^*$ at $\hat{q}$. 
Thus $\mathcal{C}_p$ and $\E_p^*$ are tangent at $\hat{q}$. 
The orthogonality of $\Hbf$ to the line $\ell_{\hat{q}^\vee}$ follows from equation \eqref{eq:tangency-line}.
Proposition \ref{prop:Cp_bitangent_Ep^*}-$a$ is proved. 
\smallskip

\noindent
$b)$ By Proposition \ref{prop:H-R_paired}, Proposition \ref{prop:Cp_bitangent_Ep^*}-$b$ holds iff Proposition \ref{prop:Cp_bitangent_Ep^*}-$a$ holds.
\end{proof}

\begin{figure}[ht] %< préférence de placement h , t , b or p >
\centering
\includegraphics[scale=0.49]{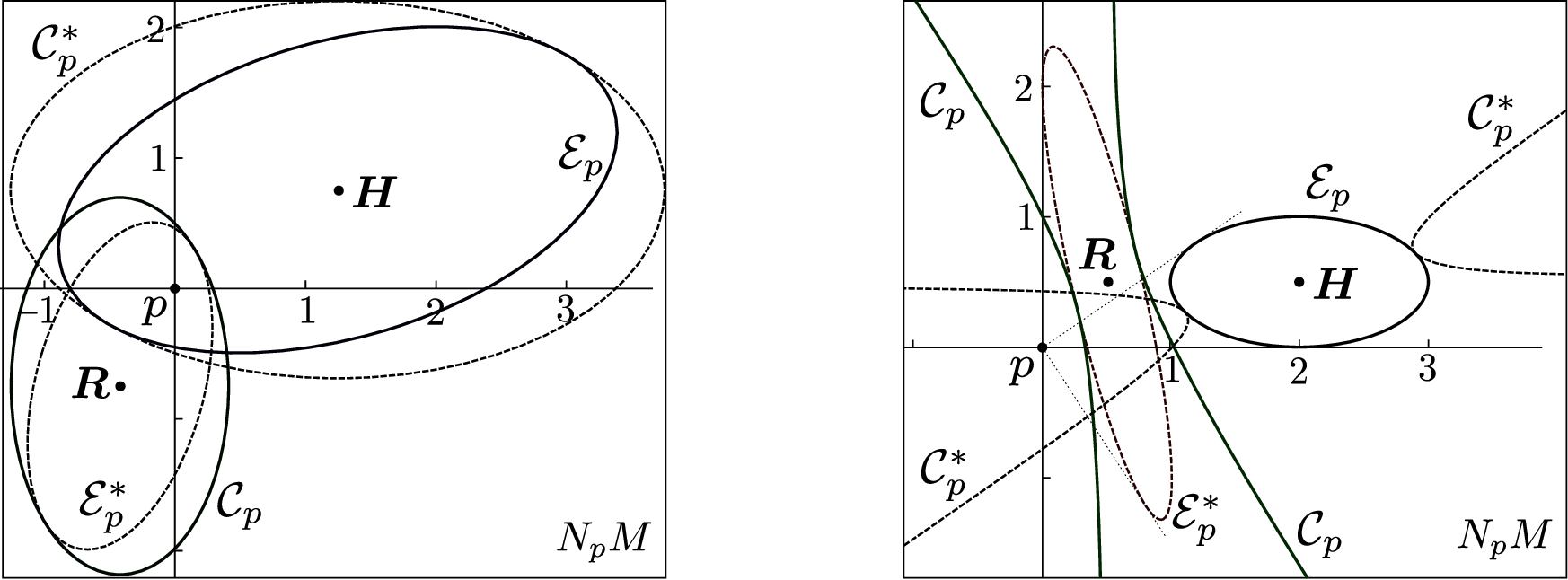}
\caption{\small Configuration of the curves $\E_p$, $\mathcal{C}_p$, $\E_p^*$ and $\mathcal{C}_p^*$ at an elliptic and at a hyperbolic point.}
\label{fig:Cp-Ep_Cp*-Ep*-elliptic}
\end{figure} 

\begin{example*}[\textbf{of elliptic point}]\label{ex:Cp-Ep_Cp*-Ep*}
Take a point $p\in M$ in $\R^4$, having local quadratic map  
$\f_p(s,t)=\left(\frac{1}{2}(2s^2+4st+\frac{1}{2}t^2), \frac{1}{2}(2s^2-\frac{1}{2}t^2)\right)$. 
We get the following curves on $N_pM$:
\[\mathcal{C}_p:\, \textstyle 3\left(q_1+\frac{5}{12}\right)^2+\left(q_2+\frac{3}{4}\right)^2=\frac{25}{12}\,;
\quad \E_p(\theta)=\begin{pmatrix}
\text{\footnotesize 5/4} \\ 
\text{\footnotesize 3/4}
\end{pmatrix}
+\cos 2\theta\begin{pmatrix}
\text{\footnotesize 3/4} \\ 
\text{\footnotesize 5/4}
\end{pmatrix}
+\sin 2\theta\begin{pmatrix}
\text{\footnotesize 2} \\ 
\text{\footnotesize 0}
\end{pmatrix}; \] 
\[\mathcal{C}_p^*:\, \textstyle \frac{1}{3}\left(q_1-\frac{5}{4}\right)^2+\left(q_2-\frac{3}{4}\right)^2=\frac{25}{12}\,;
\quad \E_p^*(\theta)=-\begin{pmatrix}
\text{\footnotesize 5/12} \\ 
\text{\footnotesize 3/4}
\end{pmatrix}
-\cos 2\theta\begin{pmatrix}
\text{\footnotesize 1/4} \\ 
\text{\footnotesize 5/4}
\end{pmatrix}
-\sin 2\theta\begin{pmatrix}
\text{\footnotesize 2/3} \\ 
\text{\footnotesize 0}
\end{pmatrix}. \] 
These four curves are depicted in Fig.\,\ref{fig:Cp-Ep_Cp*-Ep*-elliptic}-left.
\end{example*} 

\begin{example*}[\textbf{of hyperbolic point}]\label{ex:Cp-Ep_Cp*-Ep*-hyp}
Take $p\in M$ in $\R^4$, having local quadratic map  
$\f_p(s,t)=\left(\frac{1}{2}(3s^2-t^2), \frac{1}{2}(\frac{1}{2}s^2+st+\frac{1}{2}t^2)\right)$. 
We get the following curves on $N_pM$:
\[\mathcal{C}_p:\, \textstyle 3\left(q_1-\frac{1}{2}\right)^2+2\left(q_1-\frac{1}{2}\right)\left(q_2-\frac{1}{2}\right)=\frac{1}{4}\,;
\quad \textstyle \E_p(\theta)=\begin{pmatrix}
\text{\footnotesize 2} \\ 
\text{\footnotesize 1/2}
\end{pmatrix}
+\cos 2\theta\begin{pmatrix}
 \text{\footnotesize 1} \\ 
\text{\footnotesize 0}
\end{pmatrix}
+\sin 2\theta\begin{pmatrix}
\text{\footnotesize 0} \\ 
\text{\footnotesize 1/2}
\end{pmatrix}; \] 
\[\mathcal{C}_p^*:\, \textstyle 2\left(q_1-2\right)\left(q_2-\frac{1}{2}\right)-3\left(q_2-\frac{1}{2}\right)^2=\frac{1}{4}\,;
\quad \E_p^*(\theta)=\begin{pmatrix}
\text{\footnotesize 1/2} \\ 
\text{\footnotesize 1/2}
\end{pmatrix}
+\cos 2\theta\begin{pmatrix}
\text{\footnotesize 0} \\ 
\text{\footnotesize 1}
\end{pmatrix}
+\sin 2\theta\begin{pmatrix}
\text{\footnotesize 1/2} \\ 
\text{\footnotesize -3/2}
\end{pmatrix}. \] 
These four curves are depicted in Fig.\,\ref{fig:Cp-Ep_Cp*-Ep*-elliptic}-right.
\end{example*} 

%%%%%%%%%%%%%%%%%%%%%%%%%%%%%%%%%%%%%%%%%%%%%%%%%%%%%%%%%%%%%%%%%%%%%%%%%%%%%%%%%%%%%%%%%%%%%%%%%%%%%%%%%%%%%%%%%%%%
%%%%%%%%%%%%%%%%%%%%%%%%%%%%%%%       THEOREM      %%%%%%%%%%%%%%%%%%%%%%%%%%%%%%%%%%%%%%%%%%%%%%%%%%%%%%%%%%%%
%%%%%%%%%%%%%%%%%%%%%%%%%%%%%%%%%%%%%%%%%%%%%%%%%%%%%%%%%%%%%%%%%%%%%%%%%%%%%%%%%%%%%%%%%%%%%%%%%%%%%%%%%%%%%%%%%%%%
\begin{theorem}\label{th:equiv-binormal-asymptotic}
Let $\f_p=(Q_1,Q_2)$ be the local quadratic map at a hyperbolic point $p$ of $M$ in $\R^4$ $(\Delta<0)$ and $\ubf\in T_pM$. 
The following six statements are equivalent$:$ 
\smallskip

\noindent
$a)$ The vector $\ubf\in T_pM$ determines an asymptotic direction of $M$ at $p$; 
\smallskip

\noindent
$b)$ The Poisson bracket of $Q_1$, $Q_2$ vanishes at $\ubf:$ $\{Q_1,Q_2\}(\ubf)=0$; 
\smallskip

\noindent
$c)$ The tangent line to $\E_p$ at $\E_p(\ubf)$ contains $p$ {\rm (it is generated by the vector $\E_p(\ubf)$)}; 
\smallskip

\noindent
$d)$  The tangent line to $\E_p^*$ at $\E_p^*(\ubf)$ contains $p$ {\rm (it is generated by the vector $\E_p^*(\ubf)$)}; 
\smallskip

\noindent
$e)$ The vector $\E_p^*(\ubf)$ is binormal {\rm (i.e. $\mathcal{G}_p(\E_p^*(\ubf))=0$)};  
\smallskip

\noindent
$f)$ The vector  $\E_p^*(\ubf)$ is orthogonal to $\E_p(\ubf)$, that is $\left\langle \E_p^*(\ubf),\E_p(\ubf)\right\rangle=0$. 
\end{theorem} 

{\footnotesize
\noindent
\textbf{Note}. The equivalence $a\Leftrightarrow c$ has been well known (Cf. \protect\cite{Goncalves, IRFRT}). }
\smallskip

To prove Theorem\,\ref{th:equiv-binormal-asymptotic}, we need the following lemma (proved below): 

%%%%%%%%%%%%%%%%%%%%%%%%%%%%%%%%%%%%  LEMMA  %%%%%%%%%%%%%%%%%%%%%%%%%%%%%%%%%%%%%%%%%%%%%%%%%%%%%%%%%%%%%%%%%%%
%%%%%%%%%%%%%%%%%%%%%%%%%%%%%%%%%%%%%%%%%%%%%%%%%%%%%%%%%%%%%%%%%%%%%%%%%%%%%%%%%%%%%%%%%%%%%%%%%%%%%%%%%%%%%%%%
\begin{lemma}\label{lemma:f_pf_p^*=-{Q_1Q_2}^2/D}
At any point $p$ of $M$ in $\R^4$, with $\Delta\neq 0$, we have $\langle \f_p\,,\, \f_p^*\rangle\equiv -\{Q_1,Q_2\}^2/\Delta$, that is 
\begin{equation}\label{eq:<phi,phi^*>=-{Q1,Q2}/D}
  (Q_1Q_1^*+Q_2Q_2^*)(\ubf)=-\frac{\{Q_1,Q_2\}^2(\ubf)}{\Delta} \quad \mbox{for any }\ubf\in T_pM\,.
\end{equation}
\end{lemma}

\begin{remark*}
Of course, we have also the relation $\langle \f_p\,,\, \f_p^*\rangle\equiv -\{Q_1^*,Q_2^*\}^2/\Delta^*$
\end{remark*}

\begin{proof}[\textbf{\small Proof of Theorem\,\ref{th:equiv-binormal-asymptotic}}]
\underline{$c\Leftrightarrow d$}: the equivalence holds since $\flecha{\mathcal{G}_p}\comp\E_p^*=\E_p$ (Prop.\,\ref{prop:H-R_paired}). 
\smallskip

\noindent
\underline{$e\Leftrightarrow f\Leftrightarrow b$}: by Lemma\,\ref{lemma:f_pf_p^*=-{Q_1Q_2}^2/D} and 
Prop.\,\ref{prop:H-R_paired}, for any $\ubf\in T_pM$ (with $\|\ubf\|=1$) we have: 
\[2\mathcal{G}_p(\E_p^*(\ubf))=\left\langle\E_p^*(\ubf),\E_p(\ubf)\right\rangle=-\frac{4\{Q_1,Q_2\}(\ubf)}{\Delta}\,,\]
which imply the equivalence of $e$, $f$ and $b$. 
\smallskip

\noindent
\underline{$a\Leftrightarrow c$}: by definition, a unit vector $\ubf\in T_pM$ defining a sphere-contact direction 
is asymptotic iff its associated focal centre $q\in\mathcal{C}_p$ is at infinity; this holds iff 
the polar of $q$ (the tangent to $\E_p$ at $\E_p(\ubf)$, by Th.\,\ref{th:focal_quadric-indicatrix}) contains $p$. 
This proves that $a\Leftrightarrow c$. 
\smallskip

\noindent
\underline{$e\Leftrightarrow c$}:
A vector in $N_pM$ is binormal iff it is orthogonal to a line through $p$ tangent to $\E_p$.  
We have proved that $\E_p^*(\ubf)$ is binormal iff $\left\langle\E_p^*(\ubf),\E_p(\ubf)\right\rangle=0$. 
Thus $\E_p^*(\ubf)$ is binormal iff $\E_p(\ubf)$ has the direction of the tangent to $\E_p$ 
at $\E_p(\ubf)$. That is,  $e\Leftrightarrow c$
\end{proof}

\begin{proof}[\textbf{\small Proof of Lemma\,\ref{lemma:f_pf_p^*=-{Q_1Q_2}^2/D}}]
Choosing a principal basis, the matrix $[\mathcal{G}_p^*]$ is diagonal with diagonal entries $1/\mathcal{K}_1$, $1/\mathcal{K}_2$. 
Thus $Q_1^*=Q_1/\mathcal{K}_1$ and $Q_2^*=Q_2/\mathcal{K}_2$ and then 
\[Q_1Q_1^*+Q_2Q_2^*=\frac{Q_1^2}{\mathcal{K}_1}+\frac{Q_2^2}{\mathcal{K}_2}=\frac{\mathcal{K}_2Q_1^2+\mathcal{K}_1Q_2^2}{\Delta}\,.\]

To prove equality \eqref{eq:<phi,phi^*>=-{Q1,Q2}/D}, it suffices to verify that the coefficients of the quartic form 
$\mathcal{K}_2Q_1^2+\mathcal{K}_1Q_2^2$ coincide with the respective coefficients 
of $-\{Q_1,Q_2\}^2$ - see \eqref{eq:Poisson-bracket-form} in \S\,\ref{sect-pseudi-euclidian_space-quad}. 

To check that the coefficient of $s^4$ in $\mathcal{K}_2Q_1^2+\mathcal{K}_1Q_2^2$ and in 
$-\{Q_1,Q_2\}^2$ is the same, that is  
 \[(a_2c_2-b_2^2)a_1^2/4+(a_1c_1-b_1^2)a_2^2/4=-(a_1b_2-a_2b_1)^2/4\,,\]
one uses that, in a principal basis, $Q_1$ and $Q_2$ are pseudo-orthogonal.

Similarly, one verifies that the respective coefficients of $s^3t$ 
%in $\mathcal{K}_2Q_1^2+\mathcal{K}_1Q_2^2$ and in $-\{Q_1,Q_2\}^2$ 
coincide:
\[(a_2c_2-b_2^2)a_1b_1+(a_1c_1-b_1^2)a_2b_2=-(a_1b_2-a_2b_1)(a_1c_2-a_2c_1)/2\,.\]

Writing down the coefficients of the remaining terms ($s^2t^2$, $st^3$, $t^4$) of $\mathcal{K}_2Q_1^2+\mathcal{K}_1Q_2^2$,
one checks that they coincide with those of $-\{Q_1,Q_2\}^2$.
\end{proof}

%%%%%%%%%%%%%%%%%%%%%%%%%%%%%%%%%%%%%%%%%%%%%%%%%%%%%%%%%%%%%%%%%%%%%%%%%%%%%%%%%%%%%%%%%%%%%%%%%%%%%%%%%%%%%%%%%%%%
\subsection{\textbf{The cones $\Sigma_p$, $\Sigma_p^*$ related to $\E_p$, $\E_p^*$ and to $K$,  $K^*=\mathcal{A}/\Delta$ for $M$ in $\R^5$}}
%%%%%%%%%%%%%%%%%%%%%%%%%%%%%%%%%%%%%%%%%%%%%%%%%%%%%%%%%%%%%%%%%%%%%%%%%%%%%%%%%%%%%%%%%%%%%%%%%%%%%%%%%%%%%%%%%%%%

%%%%%%%%%%%%%%%%%%%%%%%%%%%%%%%%%%%%  LEMMA  %%%%%%%%%%%%%%%%%%%%%%%%%%%%%%%%%%%%%%%%%%%%%%%%%%%%%%%%%%%%%%%%%%%%%%
\begin{lemma}\label{lemma::<phi,phi^*>=0}
At any point $p$ of $M$ in $\R^5$, with $\Delta\neq 0$, we have $\langle \f_p\,,\, \f_p^*\rangle\equiv 0$, that is 
\begin{equation}\label{eq:<phi,phi^*>=0}
  (Q_1Q_1^*+Q_2Q_2^*+Q_3Q_3^*)(\ubf)=0 \quad \mbox{for any }\ubf\in T_pM\,.
\end{equation}
\end{lemma}

Lemma \ref{lemma::<phi,phi^*>=0} is proved in \S\ref{subsect-relations_Qi-Qi*}.

%%%%%%%%%%%%%%%%%%%%%%%%%%%%%%%%%%%%%%%%%%%%%%%%%%%%%%%%%%%%%%%%%%%%%%%%%%%%%%%%%%%%%%%%%%%%%%%%%%%%%%%%%%%%%%%%%%%%
%%%%%%%%%%%%%%%%%%%%%%%%%%%%%%%        COROLLARY        %%%%%%%%%%%%%%%%%%%%%%%%%%%%%%%%%%%%%%%%%%%%%%%%%%%%%%%%%%%%
%%%%%%%%%%%%%%%%%%%%%%%%%%%%%%%%%%%%%%%%%%%%%%%%%%%%%%%%%%%%%%%%%%%%%%%%%%%%%%%%%%%%%%%%%%%%%%%%%%%%%%%%%%%%%%%%%%%%
\begin{corollary}[]
At a point $p$ of $M$ in $\R^5$, with $\Delta\neq 0$, each vector of $\E_p$ is orthogonal to its 
paired vector of $\E_p^*$$:$\, $\left\langle\E_p(\ubf)\,,\,\E_p^*(\ubf))\right\rangle=0$.
\end{corollary} 
\begin{proof}
$\left\langle\E_p(\ubf)\,,\,\E_p^*(\ubf))\right\rangle=4\left\langle\f_p(\ubf)\,,\,\f_p^*(\ubf)\right\rangle=0 
\quad\mbox{(by Lemma\,\ref{lemma::<phi,phi^*>=0}).}$
\end{proof}

%%%%%%%%%%%%%%%%%%%%%%%%%%%%%%%%%%%%%%%%%%%%%%%%%%%%%%%%%%%%%%%%%%%%%%%%%%%%%%%%%%%%%%%%%%%%%%%%%%%%%%%%%%%%%%%%%%%%
%%%%%%%%%%%%%%%%%%%%%%%%%%%%%%%       PROPOSITION       %%%%%%%%%%%%%%%%%%%%%%%%%%%%%%%%%%%%%%%%%%%%%%%%%%%%%%%%%%%%
%%%%%%%%%%%%%%%%%%%%%%%%%%%%%%%%%%%%%%%%%%%%%%%%%%%%%%%%%%%%%%%%%%%%%%%%%%%%%%%%%%%%%%%%%%%%%%%%%%%%%%%%%%%%%%%%%%%%
\begin{proposition}\label{prop:G(E^*)=0}
At any point $p$ of $M$ in $\R^5$, with $\Delta\neq 0$, we have $\mathcal{G}_p^*\comp\f_p\equiv 0\equiv\mathcal{G}_p\comp\f_p^*$, that is 
\[\mathcal{G}_p^*(\f_p(\ubf))=0 \quad \mbox{and} \quad \mathcal{G}_p(\f_p^*(\ubf))=0 \quad \mbox{for any }\ubf\in T_pM\,.\]
\end{proposition}

\begin{proof}
On gets directly the following equalities
\[2\mathcal{G}_p^*(\f_p(\ubf))=\left\langle\f_p(\ubf)\,,\,\flecha{\mathcal{G}_p}^*(\f_p(\ubf))\right\rangle
=\langle\f_p(\ubf)\,,\,\f_p^*(\ubf)\rangle=0 \quad\mbox{(by Lemma\,\ref{lemma::<phi,phi^*>=0}).}\]
The second equality is proved in the same way.
\end{proof}

%%%%%%%%%%%%%%%%%%%%%%%%%%%%%%%%%%%%%%%%%%%%%%%%%%%%%%%%%%%%%%%%%%%%%%%%%%%%%%%%%%%%%%%%%%%%%%%%%%%%%%%%%%%%%%%%%%%%
%%%%%%%%%%%%%%%%%%%%%%%%%%%%%%%       THEOREM      %%%%%%%%%%%%%%%%%%%%%%%%%%%%%%%%%%%%%%%%%%%%%%%%%%%%%%%%%%%%
%%%%%%%%%%%%%%%%%%%%%%%%%%%%%%%%%%%%%%%%%%%%%%%%%%%%%%%%%%%%%%%%%%%%%%%%%%%%%%%%%%%%%%%%%%%%%%%%%%%%%%%%%%%%%%%%%%%%
\begin{theorem}\label{th:equationsSigma_pSigma_p^*}
At a point $p\in M$ in $\R^5$, with $\Delta\neq 0$, the cone $\Sigma_p$ of degenerate normal vectors 
%of the Gauss quadratic form $\mathcal{G}_p$ 
is the cone based on the paired indicatrix ellipse $\E_p^*$. 
Conversely, the cone $\Sigma_p^*$ based on the ellipse $\E_p$ is the set of degenerate vectors 
of the paired quadratic form $\mathcal{G}_p^*:$
\begin{equation}\label{eq:Gp(Ep*)=Gp*(Ep)=0}
  \mathcal{G}_p(\E_p^*(\ubf))=0 \quad \mbox{and} \quad \mathcal{G}_p^*(\E_p(\ubf))=0 \quad \mbox{for any }\ubf\in T_pM\,.
\end{equation}
\end{theorem}
\begin{proof}
It follows because $\E_p=2\f_p$, $\E_p^*=2\f_p^*$ and from Propositions \ref{prop:Cone_degenerate_normals} and \ref{prop:G(E^*)=0}. 
\end{proof}

\noindent
\textbf{\textit{\small Principal Diagonals}}. The \textit{principal diagonal lines} in $N_pM$ 
are the four lines spanned by the vectors $\epsilon_1\nbf_1+\epsilon_2\nbf_2+\epsilon_3\nbf_3$ with $\epsilon_i=\pm 1$, 
$\nbf_1$, $\nbf_2$, $\nbf_3$ being a principal basis. 

%%%%%%%%%%%%%%%%%%%%%%%%%%%%%%%%%%%%%%%%%%%%%%%%%%%%%%%%%%%%%%%%%%%%%%%%%%%%%%%%%%%%%%%%%%%%%%%%%%%%%%%%%%%%%%%%%%%%
%%%%%%%%%%%%%%%%%%%%%%%%%%%%%%%       PROPOSITION       %%%%%%%%%%%%%%%%%%%%%%%%%%%%%%%%%%%%%%%%%%%%%%%%%%%%%%%%%%%%
%%%%%%%%%%%%%%%%%%%%%%%%%%%%%%%%%%%%%%%%%%%%%%%%%%%%%%%%%%%%%%%%%%%%%%%%%%%%%%%%%%%%%%%%%%%%%%%%%%%%%%%%%%%%%%%%%%%%
\begin{proposition}\label{prop:paired-cones-K-A}
Let $p$ be a point of a smooth surface $M$ in $\R^5$, with $\Delta\neq 0$. 

\noindent
$(a)$ The cones $\Sigma_p$ and $\Sigma_p^*$ are orthogonal (\protect\cite{Costa_Moraes_R-F}) and 
pseudo-orthogonal$:$ each line of $\Sigma_p$ is orthogonal to a line of $\Sigma_p^*$ and 
pseudo-orthogonal to another line of $\Sigma_p^*$, and vice-versa. 
\smallskip

\noindent
$(b)$ The symmetry planes of both cones $\Sigma_p$ and $\Sigma_p^*$ are the planes spanned 
by pairs of vectors of a principal basis of $\mathcal{G}_p$. Their common axis  
is the line spanned by $\ebf_3$. 
\smallskip

\noindent
$(c)$ In the plane $q_3=1$, the ellipses $\widehat{\Sigma}_p:=\Sigma_p\cap(q_3=1)$ and 
$\widehat{\Sigma}_p^*:=\Sigma_p^*\cap(q_3=1)$ are 
polar dual to each other with respect the unit circle {\rm (in Fig.\,\ref{fig:dual-ellipses} 
the dotted circle is $\sph^1$)}. 
\smallskip

\noindent
$(d)$ The invariants $K$ and $\mathcal{A}$ are obtained respectively by evaluating $\mathcal{G}_p$ 
and $\mathcal{G}_p^*$ at principal diagonal vectors whose respective components have length $1$ and $\tau$: 
\[\mathcal{G}_p\begin{pmatrix}
 \epsilon_1  \\ 
 \epsilon_2  \\
 \epsilon_3 
\end{pmatrix}=K \qquad \mbox{ and } \qquad 
\mathcal{G}_p^*\begin{pmatrix}
 \tau\epsilon_1 \\ 
 \tau\epsilon_2 \\
 \tau\epsilon_3
\end{pmatrix}=\mathcal{A}\,. \qquad \epsilon_i=\pm 1\,.
\]

\noindent
$(e)$ The sign of $K$ is positive $($negative, zero$)$ iff the principal diagonal lines 
lie inside $($resp. outside, on$)$ the cone $\Sigma_p$. 
The sign of $\mathcal{A}$ is positive $($negative, zero$)$ iff the diagonal lines 
lie inside $($resp. outside, on$)$ the cone $\Sigma_p^*$ {\rm (see Fig.\,\ref{fig:paired-conesSigmapSigmap*})}. 
\smallskip

\noindent
$(f)$ At least one of the invariants $\mathcal{A}$ or $K$ is negative. 
\end{proposition}

\begin{proof}
$(a)$ The `positive' half of $\Sigma_p$ is the image of $\f_p^*$ and the `positive' half of $\Sigma_p^*$ 
is the image of $\f_p$ (Th.\,\ref{th:equationsSigma_pSigma_p^*} and Prop.\,\ref{prop:G(E^*)=0}). 
Therefore, a vector of $\Sigma_p$ is of the form $\pm\f_p^*(\ubf)=\pm(Q_1/K_1,Q_2/\K_2,Q_3/\K_3)(\ubf)$ for some $\ubf\in T_pM$ 
and a vector of $\Sigma_p^*$ is of the form $\pm\f_p(\vbf)=\pm(Q_1,Q_2,Q_3)(\vbf)$ for some $\vbf\in T_pM$.

The two vectors of $\Sigma_p^*$, $\f_p(\ubf)=(Q_1,Q_2,Q_3)(\ubf)$ and $(-Q_1,-Q_2,Q_3)(\ubf)$, are respectively 
orthogonal and pseudo-orthogonal (with respect to the pseudo-metric $(-1,-1,1)$) to the line of $\Sigma_p$ generated by $\f_p^*(\ubf)$, namely 
$\langle \f_p(\ubf),\f_p^*(\ubf)\rangle=0$ (Lemma\,\ref{lemma::<phi,phi^*>=0}) and 
\[\left\langle\left(-Q_1,-Q_2,Q_3\right)(\ubf),\left(\frac{Q_1}{\K_1},\frac{Q_2}{\K_2},\frac{Q_3}{\K_3}\right)(\ubf)\right\rangle_{(-1,-1,1)}
=\langle \f_p(\ubf),\f_p^*(\ubf)\rangle=0\,.\] 
$(b)$ One uses that, in a principal basis, the respective equations of $\Sigma_p$ and $\Sigma_p^*$ are 
both in canonical form (with $\K_1\leq\K_2<0<\K_3$):
\[\K_1q_1^2+\K_2q_2^2+\K_3q_3^2=0 \qquad \mbox{and} \qquad \K_1^{-1}q_1^2+\K_2^{-1}q_2^2+\K_3^{-1}q_3^2=0\,.\]
$(c)$ The statement follows because the equations of the respective ellipses are 
\[(-\K_1/\K_3)q_1^2+(-\K_2/\K_3)q_2^2=1 \qquad \mbox{and} \qquad (-\K_3/\K_1)q_1^2+(-\K_3/\K_2)q_2^2=1\,.\]
$(d)$ One just evaluates $\mathcal{G}_p$ and $\mathcal{G}_p^*$ on the respective mantioned diagonal vectors. 
\smallskip

\noindent
$(e)$ It follows from item $d$. 
\smallskip

\noindent
$(f)$ It follows from item $d$ and $e$ (see  Figs.\,\ref{fig:paired-conesSigmapSigmap*} and \ref{fig:dual-ellipses}).
\end{proof}

\begin{proposition}
At the points $p\in M$ in $\R^5$ with $\tau\neq 0$, 
there are five possible generic configurations of the cones $\Sigma_p$ and $\Sigma_p^*$ relative to the principal diagonals {\rm  (Fig.\,\ref{fig:paired-conesSigmapSigmap*})}. 
%Moreover, $K$ (resp. $\mathcal{A}$) is positive iff the cone $\Sigma_p$ $($resp. $\Sigma_p^*$$)$ contains 
%the principal diagonals.
\end{proposition}

\begin{proof}
It follows from items $c$ and $d$ of Proposition\,\ref{prop:paired-cones-K-A}.
\end{proof}

\begin{figure}[ht] %< préférence de placement h , t , b or p >
\centering
\includegraphics[scale=0.085]{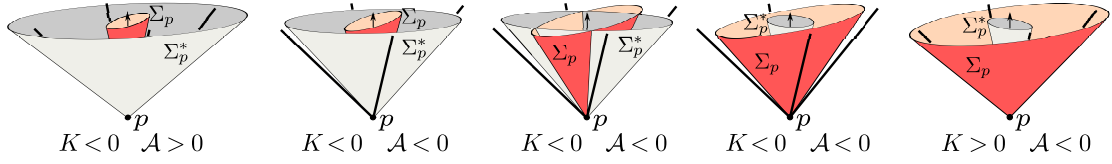}
\caption{\small The five generic configurations of the cones $\Sigma_p$ and $\Sigma_p^*$ with the principal diagonals.} 
\label{fig:paired-conesSigmapSigmap*}
\end{figure}

\begin{figure}[ht] %< préférence de placement h , t , b or p >
\centering
\includegraphics[scale=0.084]{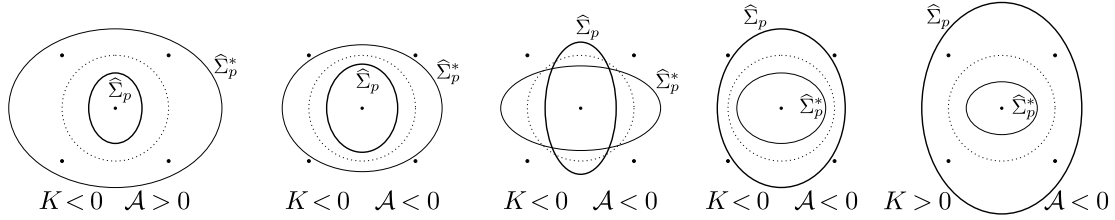}
\caption{\small The sections $q_3=1$ of $\Sigma_p$ and $\Sigma_p^*$ are ellipses 
$\widehat{\Sigma}_p$ and $\widehat{\Sigma}_p^*$ dual to each other. The corner points $(\pm 1, \pm 1)$ lie inside $\widehat{\Sigma}_p$ (resp. $\widehat{\Sigma}_p^*$) 
iff $K$ (resp. $\mathcal{A}$) is positive. 
The unit circle is dotted.} 
\label{fig:dual-ellipses}
\end{figure}

%%%%%%%%%%%%%%%%%%%%%%%%%%%%%%%%%%%%%%%%%%%%%%%%%%%%%%%%%%%%%%%%%%%%%%%%%%%%%%%%%%%%%%%%%%%%%%%%%%%%%%%%%%%%%%%%%
%%%%%%%%%%%%%%%%%%%%%%%%%%%%%%%%%%%%%%%%%%%%%%%%%%%%%%%%%%%%%%%%%%%%%%%%%%%%%%%%%%%%%%%%%%%%%%%%%%%%%%%%%%%%%%%%%
\subsection{\textbf{Relations between the components of two paired quadratic maps}\label{subsect-relations_Qi-Qi*}}
%%%%%%%%%%%%%%%%%%%%%%%%%%%%%%%%%%%%%%%%%%%%%%%%%%%%%%%%%%%%%%%%%%%%%%%%%%%%%%%%%%%%%%%%%%%%%%%%%%%%%%%%%%%%%%%%%
%%%%%%%%%%%%%%%%%%%%%%%%%%%%%%%%%%%%%%%%%%%%%%%%%%%%%%%%%%%%%%%%%%%%%%%%%%%%%%%%%%%%%%%%%%%%%%%%%%%%%%%%%%%%%%%%%
The quadratic forms $Q_i$ and their paired quadratic forms $Q_i^*$ satisfy several relations. 

\begin{lemma}\label{lemma:<Q_i,Q_j^*>}
At any point $p$ of $M$ in $\R^4$ $($or in $\R^5$$)$, with $\Delta\neq 0$, we have 
\begin{equation}\label{eq:<Q_i,Q_j^*>=d_ij}
\langle Q_i\,,\,Q_i^*\rangle_\psi=1 \quad \mbox{and} \quad \langle Q_i\,,\,Q_j^*\rangle_\psi=0 \ \mbox{for } i\neq j\,.
\end{equation}
\end{lemma} 

\begin{proof}
Write $k_{ij}^*:=\langle Q_i^*\,,\,Q_j^*\rangle_\psi$ (with $k_i^*:=k_{ii}^*$) for the entries of the matrix $[\mathcal{G}_p^*]$. 
For $M$ in $\R^4$ we have that $Q_1^*=k_1^*Q_1+k_{12}^*Q_2$ and $Q_2^*=k_{12}^*Q_1+k_2^*Q_2$. Then 
\[\langle Q_1\,,\,Q_1^*\rangle_\psi=\langle Q_1\,,\,k_1^*Q_1+k_{12}^*Q_2\rangle_\psi=k_1^*k_1+k_{12}^*k_{12}=1\quad 
\mbox{(because $[\mathcal{G}_p^*][\mathcal{G}_p]=\mathrm{Id}$)}\,.\]
In the same way, one proves that $\langle Q_2\,,\,Q_2^*\rangle_\psi=1$ and 
$\langle Q_1\,,\,Q_2^*\rangle_\psi=\langle Q_2\,,\,Q_1^*\rangle_\psi=0$. 

For $M$ in $\R^5$ we have $Q_1^*=k_1^*Q_1+k_{12}^*Q_2+k_{13}^*Q_3$ and we get in the same way that
\[\langle Q_1\,,\,Q_1^*\rangle_\psi=\langle Q_1\,,\,k_1^*Q_1+k_{12}^*Q_2+k_{13}^*Q_3\rangle_\psi=k_1^*k_1+k_{12}^*k_{12}+k_{13}^*k_{13}=1\,.\] 
The remaining equalities are proved in the same way. 
\end{proof}

%%%%%%%%%%%%%%%%%%%%%%%%%%%%%%%%%%%%%%%%%%%%%%%%%%%%%%%%%%%%%%%%%%%%%%%%%%%%%%%%%%%%%%%%%%%%%%%%%%%%%%%%%%%%%%%%%%%%
%%%%%%%%%%%%%%%%%%%%%%%%%%%%%%%       PROPOSITION       %%%%%%%%%%%%%%%%%%%%%%%%%%%%%%%%%%%%%%%%%%%%%%%%%%%%%%%%%%%%
%%%%%%%%%%%%%%%%%%%%%%%%%%%%%%%%%%%%%%%%%%%%%%%%%%%%%%%%%%%%%%%%%%%%%%%%%%%%%%%%%%%%%%%%%%%%%%%%%%%%%%%%%%%%%%%%%%%%
\begin{proposition}
For every $p\in M$ in $\R^5$ with $\tau\neq 0$ the paired quadratic forms $Q_i^*$ are given in terms of the 
Poisson bracket as follows {\rm (compare with formula \eqref{eq:R=(<,>,<,>,<,>)})}
\begin{equation}\label{eq:Q1^*={Q2,Q3}/tau}
 Q_1^*=\frac{\{Q_2,Q_3\}}{\tau}\,\quad Q_2^*=\frac{\{Q_3,Q_1\}}{\tau}\,\quad Q_3^*=\frac{\{Q_1,Q_2\}}{\tau}\,. 
\end{equation}
\end{proposition}

\begin{proof}
Observe that $\tau=\langle Q_1\,,\,\{Q_2,Q_3\}\rangle_\psi=\langle Q_2\,,\,\{Q_3,Q_1\}\rangle_\psi=\langle Q_3\,,\,\{Q_1,Q_2\}\rangle_\psi$. 

For $M$ in $\R^5$ the paired quadratic forms $Q_i^*$ are completely determined 
(in terms of the quadratic forms $Q_j$) by the identities \eqref{eq:<Q_i,Q_j^*>=d_ij}. 
Hence, it suffices to prove that the right hand side terms of the relations \eqref{eq:Q1^*={Q2,Q3}/tau} satisfy identities \eqref{eq:<Q_i,Q_j^*>=d_ij}.
Namely, \[\langle Q_1\,,\, \{Q_2,Q_3\}/\tau\rangle_\psi=\tau/\tau=1\,,\]
and, similarly, we get $\langle Q_3\,,\, \{Q_1,Q_2\}/\tau\rangle_\psi=\langle Q_2\,,\, \{Q_3,Q_1\}/\tau\rangle_\psi=1$. 

Finally, $\langle Q_1\,,\, \{Q_1,Q_2\}/\tau\rangle_\psi=\langle Q_2\,,\, \{Q_1,Q_2\}/\tau\rangle_\psi=0$ because 
$\{Q_1,Q_2\}$ is pseudo-orthogonal to $Q_1$ and to $Q_2$. The remaining identities are verified in the same way. 
\end{proof}

%%%%%%%%%%%%%%%%%%%%%%%%%%%%%%%%%%%%%%%%%%%%%   PROOF OF LEMMA   %%%%%%%%%%%%%%%%%%%%%%%%%%%%%%%%%%%%%%%%%%%%%%%%%%%%%%%
\begin{proof}[\textbf{Proof of Lemma \ref{lemma::<phi,phi^*>=0}}]
By equalities \eqref{eq:Q1^*={Q2,Q3}/tau}, to prove \eqref{eq:<phi,phi^*>=0} is equivalent to prove that
\begin{equation}\label{eq:<Q1,{Q2,Q_3}+...>=0}
  Q_1\{Q_2,Q_3\}+Q_2\{Q_3,Q_1\}+Q_3\{Q_1,Q_2\}\equiv 0\,.
\end{equation}
Now, to prove that each coefficient of the quartic form \eqref{eq:<Q1,{Q2,Q_3}+...>=0} is zero, one verifies that 
every term on each such coefficient appears twice but with opposite signs. 
For example, the respective coefficients of $s^4$ and of $t^4$ in \eqref{eq:<Q1,{Q2,Q_3}+...>=0}, given by 
\[(a_1(a_2b_3-a_3b_2)+a_2(a_3b_1-a_1b_3)+a_3(a_1b_2-a_2b_1))/2 \quad \mbox{and}\]
\[(c_1(b_2c_3-b_3c_2)+c_2(b_3c_1-b_1c_3)+c_3(b_1c_2-b_2c_1))/2\,,\]  
are clearly zero. Writing down the coefficients of the remaining monomials ($s^3t$, $s^2t^2$, $st^3$) one checks directly 
that they also vanish. 
\end{proof}

%%%%%%%%%%%%%%%%%%%%%%%%%%%%%%%%%%%%%%%%%%%%%%%%%%%%%%%%%%%%%%%%%%%%%%%%%%%%%%%%%%%%%%%%%%%%%%%%%%%%%%%%%%%%%%%%%%%%
%%%%%%%%%%%%%%%%%%%%%%%%%%%%%%%       PROPOSITION       %%%%%%%%%%%%%%%%%%%%%%%%%%%%%%%%%%%%%%%%%%%%%%%%%%%%%%%%%%%%
%%%%%%%%%%%%%%%%%%%%%%%%%%%%%%%%%%%%%%%%%%%%%%%%%%%%%%%%%%%%%%%%%%%%%%%%%%%%%%%%%%%%%%%%%%%%%%%%%%%%%%%%%%%%%%%%%%%%
\begin{proposition}
At any point $p$ of $M$ in $\R^4$, with $\Delta\neq 0$, we have 
\[\{Q_1,Q_1^*\}=-\{Q_2,Q_2^*\}\,.\]
\end{proposition}
\begin{proof}
As in the proof of Lemma\,\ref{lemma:<Q_i,Q_j^*>}, we use the equalities $Q_1^*=k_1^*Q_1+k_{12}^*Q_2$ 
and $Q_2^*=k_{12}^*Q_1+k_2^*Q_2$ to get 
\begin{align*}
\{Q_1,Q_1^*\}+\{Q_2,Q_2^*\} & =\{Q_1,k_1^*Q_1+k_{12}^*Q_2\}+\{Q_2,k_{12}^*Q_1+k_2^*Q_2\} \\ 
 & =k_{12}^*\{Q_1,Q_2\}+k_{12}^*\{Q_2,Q_1\}=0\,.
\end{align*}
\end{proof}

%%%%%%%%%%%%%%%%%%%%%%%%%%%%%%%%%%%%%%%%%%%%%%%%%%%%%%%%%%%%%%%%%%%%%%%%%%%%%%%%%%%%%%%%%%%%%%%%%%%%%%%%%%%%%%%%%%%%
%%%%%%%%%%%%%%%%%%%%%%%%%%%%%%%       PROPOSITION       %%%%%%%%%%%%%%%%%%%%%%%%%%%%%%%%%%%%%%%%%%%%%%%%%%%%%%%%%%%%
%%%%%%%%%%%%%%%%%%%%%%%%%%%%%%%%%%%%%%%%%%%%%%%%%%%%%%%%%%%%%%%%%%%%%%%%%%%%%%%%%%%%%%%%%%%%%%%%%%%%%%%%%%%%%%%%%%%%
\begin{proposition}\label{prop:Sum{Qi,Qi^*}=0}
At any point $p$ of $M$ in $\R^5$, with $\tau\neq 0$, we have 
\[\{Q_1,Q_1^*\}+\{Q_2,Q_2^*\}+\{Q_3,Q_3^*\}=0\,.\] 
\end{proposition}
\begin{proof}
Using equalities \eqref{eq:Q1^*={Q2,Q3}/tau} we get the equality
\[\{Q_1,Q_1^*\}+\{Q_2,Q_2^*\}+\{Q_3,Q_3^*\}=
\frac{1}{\tau}\left(\left\{Q_1,\{Q_2,Q_3\}\right\}+\left\{Q_2,\{Q_3,Q_1\}\right\}+\left\{Q_3,\{Q_1,Q_2\}\right\}\right)\]
whose right-hand side term is zero by the Jacobi identity. 
\end{proof}

\end{document}